\documentclass[11pt,twoside,a4paper]{amsart}
\usepackage{amsfonts, amsthm, amssymb, amsmath, stmaryrd}
\usepackage{mathrsfs,array}
\usepackage{eucal,times,color,enumerate,accents}
\usepackage[all]{xy}
\usepackage{url}
\usepackage{enumitem}
\usepackage{comment} 
\usepackage{mathtools}
\usepackage[utf8]{inputenc}
\usepackage[T1]{fontenc} 

\usepackage{enumitem}
\usepackage[normalem]{ulem} 

\usepackage{graphicx}
\usepackage{tikz-cd}
\usetikzlibrary{calc}
\usetikzlibrary{matrix,arrows,decorations.pathmorphing}

\usepackage{stmaryrd}
\usepackage{trimclip}

\usepackage{relsize} 

\usepackage{indentfirst} 

\usepackage{MnSymbol} 

\makeatletter
\DeclareRobustCommand{\shortto}{%
  \mathrel{\mathpalette\short@to\relax}%
}

\newcommand{\short@to}[2]{%
  \mkern2mu
  \clipbox{{.5\width} 0 0 0}{$\m@th#1\vphantom{+}{\shortrightarrow}$}%
  }
\makeatother

\input xy
\xyoption{all}

\newcommand{\calO}{\mathcal{O}}
\newcommand{\calC}{\mathcal{C}}
\newcommand{\bbZ}{\mathbb{Z}}
\newcommand{\bbR}{\mathbb{R}}
\newcommand{\bbQ}{\mathbb{Q}}
\newcommand{\bbP}{\mathbb{P}}
\newcommand{\piet}{\pi_1^{\et}}

\newcommand{\pipet}{\pi_1^{\mathrm{pro\et}}}

\newcommand{\pisga}{\pi_1^{\mathrm{SGA3}}}
\newcommand{\Spec}{\mathrm{Spec}}
\newcommand{\Gal}{\mathrm{Gal}}
\newcommand{\Hom}{\mathrm{Hom}}
\newcommand{\Aut}{\mathrm{Aut}}
\newcommand{\GL}{\mathrm{GL}}
\newcommand{\id}{\mathrm{id}}
\newcommand{\et}{\mathrm{\acute{e}t}}
\newcommand{\Et}{\mathrm{\acute{E}t}}
\newcommand{\Image}{\mathrm{Im}}
\newcommand{\Ker}{\mathrm{Ker}}

\newcommand{\bk}{\bar{k}}
\newcommand{\bs}{\bar{s}}
\newcommand{\bt}{\bar{t}}
\newcommand{\bx}{\bar{x}}
\newcommand{\bw}{\bar{w}}
\newcommand{\by}{\bar{y}}
\newcommand{\bp}{\bar{p}}
\newcommand{\bh}{\bar{h}}

\newcommand{\bG}{\bar{G}}
\newcommand{\bpartial}{\bar{\partial}}

\newcommand{\scrG}{\mathscr{G}}
\newcommand{\rmSets}{\mathrm{Sets}}
\newcommand{\rmNoohi}{\mathrm{Noohi}}

\newcommand{\rmStab}{\mathrm{Stab}}
\newcommand{\rmLoc}{\mathrm{Loc}}
\newcommand{\bbN}{\mathbb{N}}
\newcommand{\bbG}{\mathbb{G}}
\newcommand{\rmtop}{\mathrm{top}}
\newcommand{\rmim}{\mathrm{im}}
\newcommand{\rmsep}{\mathrm{sep}}
\newcommand{\rmlcs}{\mathrm{lcs}}
\newcommand{\rmCov}{\mathrm{Cov}}
\newcommand{\rmAut}{\mathrm{Aut}}
\newcommand{\rmShv}{\mathrm{Shv}}
\newcommand{\rmVert}{\mathrm{Vert}}
\newcommand{\tX}{\widetilde{X}}
\newcommand{\tC}{\widetilde{C}}

\newcommand{\bX}{\overline{X}}
\newcommand{\bbH}{\overline{\overline{H}}}
\newcommand{\hbbZ}{\widehat{\mathbb{Z}}}
\newcommand{\rmdiscr}{\mathrm{discr}}
\newcommand{\proet}{\mathrm{pro\et}}

\newcommand{\scrF}{\mathscr{F}}
\newcommand{\scrC}{\mathscr{C}}
\newcommand{\rmDD}{\mathrm{DD}}
\newcommand{\calB}{\mathcal{B}}
\newcommand{\rmFr}{\mathrm{Fr}}
\newcommand{\frakp}{\mathfrak{p}}
\newcommand{\rarr}{\rightarrow}
\newcommand{\epirarr}{\twoheadrightarrow}
\newcommand{\monorarr}{\hookrightarrow}

\newcommand{\cupt}{\cup_{\bullet}}
\newcommand{\topfp}{\mathrm{*^{\rmtop}}} 

\newcommand*{\triple}[2][.1ex]{%
  \mathrel{\vcenter{\offinterlineskip%
  \hbox{$#2$}\vskip#1\hbox{$#2$}\vskip#1\hbox{$#2$}}}}
\newcommand*{\tripleleftarrow}{\triple{\leftarrow}}

\usepackage{scalerel,stackengine}
\stackMath
\newcommand\reallywidehat[1]{%
\savestack{\tmpbox}{\stretchto{%
  \scaleto{%
    \scalerel*[\widthof{\ensuremath{#1}}]{\kern-.6pt\bigwedge\kern-.6pt}%
    {\rule[-\textheight/2]{1ex}{\textheight}}
  }{\textheight}%
}{0.5ex}}%
\stackon[1pt]{#1}{\tmpbox}%
}

\newtheorem{theorem}{Theorem}[section]
\newtheorem{proposition}[theorem]{Proposition}
\newtheorem{lemma}[theorem]{Lemma}
\newtheorem{corollary}[theorem]{Corollary}

\newtheorem{definition}[theorem]{Definition}

\newtheorem{defn}[theorem]{Definition}
\newtheorem{prop}[theorem]{Proposition}
\newtheorem*{theorem*}{Theorem}
\newtheorem*{lemma*}{Lemma}
\newtheorem*{proposition*}{Proposition}

\theoremstyle{definition}
\newtheorem{question}[theorem]{Question}

\newtheorem{example}[theorem]{Example}

\newtheorem{fact}[theorem]{Fact}
\newtheorem{obs}[theorem]{Observation}
\newtheorem{rmk}[theorem]{Remark}
\newtheorem*{rmk*}{Remark}
\newtheorem*{obs*}{Observation}

\numberwithin{theorem}{section} 

\newcommand{\adjunction}[4]{\xymatrix@1{#1{\ } \ar@<-0.3ex>[r]_{ {\scriptstyle #2}} & {\ } #3 \ar@<-0.3ex>[l]_{ {\scriptstyle #4}}}}

\title{Fundamental Exact Sequence for the Pro-\'Etale Fundamental Group}
\date{\today}
\author{Marcin Lara}

\address{Institute of Mathematics, Faculty of Mathematics and Computer Science, Jagiellonian University, {\L}ojasiewicza 6, 30-348 Krak\'ow, Poland
\newline  \indent Institute of Mathematics, Goethe University Frankfurt, Robert-Mayer-Str.\ 6--8, 60325 Frankfurt am Main, Germany}
  \email{marcin.lara@uj.edu.pl}
  \email{lara@math.uni-frankfurt.de}

  \keywords{pro-\'etale topology, pro-\'etale fundamental group, \'etale fundamental group, homotopy exact sequence, fundamental exact sequence}

\begin{document}

\begin{abstract} The pro-\'etale fundamental group of a scheme, introduced by Bhatt and Scholze, generalizes formerly known fundamental groups -- the usual \'etale fundamental group $\piet$ defined in SGA1 and the more general $\pisga$. It controls local systems in the pro-\'etale topology and leads to an interesting class of ``geometric coverings'' of schemes, generalizing finite \'etale coverings.

We prove exactness of the fundamental sequence for the pro-\'etale fundamental group of a geometrically connected scheme $X$ of finite type over a field $k$, i.e.\ that the sequence
\begin{displaymath}
1 \rarr \pipet(X_{\bk}) \rarr \pipet(X) \rarr \Gal_k \rarr 1
\end{displaymath}
is exact as abstract groups and the map $\pipet(X_{\bk}) \rarr \pipet(X)$ is a topological embedding.

 On the way, we prove a general van Kampen theorem and the K\"unneth formula for the pro-\'etale fundamental group.
\end{abstract}

\maketitle

 \tableofcontents


\section{Introduction}

In \cite{BhattScholze}, the authors introduced the pro-\'etale topology for schemes. The main motivation was that the definitions of $\ell$-adic sheaves and cohomologies in the usual \'etale topology are rather indirect. In contrast, the naive definition of e.g.\ a constant $\bbQ_\ell$-sheaf in the pro-\'etale topology as $X_{\proet} \ni U \mapsto \mathrm{Maps}_{\mathrm{cts}}(U,\bbQ_\ell)$ is a sheaf and if $X$ is a variety over an algebraically closed field, then $H^i(X_{\et}, \bbQ_\ell) = H^i(X_{\proet},\bbQ_\ell)$, where the right hand side is defined ``naively'' by applying the derived functor $\mathrm{R}\Gamma(X_{\proet}, -)$ to the described constant sheaf.

Along with the new topology, the authors of \cite{BhattScholze} introduced a new fundamental group -- {\bf the pro-\'etale fundamental group}. It is defined for a connected locally topologically noetherian scheme $X$ with a geometric point $\bx$ and denoted $\pipet(X,\bx)$. The name ``pro-\'etale'' is justified by the fact that there is an equivalence $\pipet(X,\bx)-\rmSets \simeq  \rmLoc_{X_{\proet}}$ between the categories of (possibly infinite) discrete sets with continuous action by $\pipet(X,\bx)$ and locally constant sheaves of (discrete) sets in $X_{\proet}$. This is analogous to the classical fact that $\piet(X,\bx) -\mathrm{FSets}$ is equivalent to the category of lcc sheaves on $X_{\et}$, where $G - \mathrm{FSets}$ denotes \underline{finite} sets with a continuous $G$ action. This is the first striking difference between these fundamental groups: $\pipet$ allows working with sheaves of infinite sets. In fact, the authors of \cite{BhattScholze} study abstract ``infinite Galois categories'', which are pairs $(\calC,F)$ satisfying certain axioms that (together with an additional tameness condition) turn out to be equivalent to a pair $(G - \rmSets, F_G : G-\rmSets \rarr \rmSets)$ for a Hausdorff topological group $G$ and the forgetful functor $F_G$. In fact, one takes $G=\rmAut(F)$ with a suitable topology. This generalizes the usual Galois categories, introduced by Grothendieck to define $\piet(X,\bx)$. In Grothendieck's approach, one takes the category $\mathrm{F\Et}_X$ of finite \'etale coverings together with the fibre functor $F_{\bx}$ and obtains that $\piet(X,\bx) - \mathrm{FSets} \simeq \mathrm{F\Et}_X$. Discrete sets with a continuous $\pipet(X,\bx)$-action correspond to a larger class of coverings, namely ``{\bf geometric coverings}'', which are defined to be schemes $Y$ over $X$ such that $Y \rarr X$:
\begin{enumerate}
    \item is \'etale (not necessarily quasi-compact!)
    \item satisfies the valuative criterion of properness.
\end{enumerate}

  We denote the category of geometric coverings by $\rmCov_X$ (seen as a full subcategory of $\mathrm{Sch}_X$). It is clear that $\mathrm{F\Et} \subset \rmCov_X$. As $Y$ is not assumed to be of finite type over $X$, the valuative criterion does not imply that $Y \rarr X$ is proper (otherwise we would get finite \'etale morphisms again) and so in general we get more. A basic example of a non-finite covering in $\rmCov_X$ can be obtained by viewing an infinite chain of (suitably glued) $\bbP^1_k$'s as a covering of the nodal curve $X=\bbP^1/\{0,1\}$ obtained by gluing $0$ and $1$ on $\bbP^1_k$ (to formalize the gluing one can use \cite{Schwede}). Then, if $k=\bk$, $\pipet(X,\bx) = \bbZ$ and $\piet(X,\bx) = \hbbZ$.
In this example, the prodiscrete group $\pisga$ defined in Chapter X.6 of \cite{SGA3vol2} would give the same answer. This is essentially because our infinite covering is a torsor under a discrete group in $X_{\et}$. However, for more general schemes (e.g.\ an elliptic curve with two points glued), the category $\rmCov_X$ contains more. So far, all the new examples were coming from non-normal schemes. This is not a coincidence, as for a normal scheme $X$, any $Y \in \rmCov_X$ is a (possibly infinite) disjoint union of finite \'etale coverings. In this case, $\pipet(X,\bx) = \pisga(X,\bx) = \piet(X,\bx)$. In general $\piet$ can be recovered as the profinite completion of $\pipet$ and $\pisga$ is the prodiscrete completion of $\pipet$.

The groups $\pipet$ belong in general to a class of {\bf Noohi groups}. These can be characterized as Hausdorff topological groups $G$ that are Ra{\u \i}kov complete and such that the open subgroups form a basis of neighbourhoods at $1_G$. However, \uline{open \emph{normal} subgroups do not necessarily form a basis} of open neighborhoods of $1_G$ in a Noohi group. In the case of $\pipet$, this means that there might exist a connected $Y \in \rmCov_X$ that do not have a Galois closure. Examples of Noohi groups include: profinite groups, (pro)discrete groups, but also $\bbQ_\ell$ and $\GL_n(\bbQ_\ell)$. A slightly different example would be $\Aut(S)$, where $S$ is a discrete set and $\Aut$ has the compact-open topology.

The fact that groups like $\GL_n(\bbQ_\ell)$ are Noohi (but not profinite or prodiscrete) makes $\pipet$ better suited to work with $\bbQ_\ell$ (or $\overline{\bbQ}_\ell$) local systems. Indeed, denoting by $\rmLoc_{X_\proet}(\bbQ_\ell)$ the category of $\bbQ_\ell$-local systems on $X_{\proet}$, i.e.\ locally constant sheaves of finite-dimensional $\bbQ_\ell$-vector spaces (again, the ``naive'' definition works in $X_{\proet}$), one has an equivalence $\mathrm{Rep}_{\mathrm{cts},\bbQ_\ell}(\pipet(X,\bx)) \simeq \rmLoc_{X_{\proet}}(\bbQ_{\ell})$. This fails for $\piet$, as any $\bbQ_\ell$-representation of a profinite group must stabilize a $\bbZ_\ell$-lattice, while $\bbQ_\ell$-local systems (in the above sense) stabilize lattices only \'etale locally. The group $\pisga$ is not enough either; as shown by \cite[Example 7.4.9]{BhattScholze} (due to Deligne), if $X$ is the scheme obtained by gluing two points on a smooth projective curve of genus $g\geq 1$,  there are $\bbQ_\ell$-local systems on $X$ that do not come from a representation of $\pisga(X)$.

We will often drop $\bx$ from the notation for brevity. This usually does not matter much, as a different choice of the base point leads to an isomorphic group.

\subsection*{Classical results}
In \cite{SGA1}, Grothendieck proved some foundational results regarding the \'etale fundamental group. Among them:
\begin{enumerate}
  \item \label{item:FES-item-intro} The fundamental exact sequence, i.e.\ the comparison between the ``arithmetic'' and ``geometric'' fundamental groups:

\begin{theorem}(\cite[Exp. IX, Théorème 6.1]{SGA1})
  Let $k$ be a field with algebraic closure $\bk$. Let $X$ be a quasi-compact and quasi-separated scheme over $k$. If the base change $X_{\bk}$ is connected, then there is a short exact sequence
  \begin{displaymath}
  1 \rarr \piet(X_{\bk}) \rarr \piet(X) \rarr \Gal_k \rarr 1
  \end{displaymath}
of profinite topological groups. 
\end{theorem}

  \item \label{item:HES-item-intro} The homotopy exact sequence:
\begin{theorem}(\cite[Exp. X, Corollaire 1.4]{SGA1})
  Let $f:X \rarr S$ be a flat proper morphism of finite presentation whose geometric fibres are connected and reduced. Assume $S$ is connected and let $\bs$ be a geometric point of $S$. Then there is an exact sequence 
  \begin{displaymath}
  \piet(X_{\bs}) \rarr \piet(X) \rarr \piet(S) \rarr 1
  \end{displaymath}
  of fundamental groups.
\end{theorem}
\item ``K\"unneth formula":
\begin{proposition}(\cite[Exp. X, Cor.\ 1.7]{SGA1})
Let $X, Y$ be two connected schemes locally of finite type over an algebraically closed field $k$ and assume that $Y$ is proper. Let $\bx, \by$ be geometric points of $X$ and $Y$ respectively with values in the same algebraically closed field extension $K$ of $k$. Then the map induced by the projections is an isomorphism
\begin{displaymath}
\piet(X \times_k Y,(\bx,\by)) \stackrel{\sim}{\rarr} \piet(X,\bx) \times \piet(Y,\by)
\end{displaymath}
\end{proposition}

\item Invariance of $\piet$ under extensions of algebraically closed fields for proper schemes (\cite[Exp. X, Corollaire 1.8]{SGA1});
\item General van Kampen theorem (proved in a special case in \cite[IX \S 5]{SGA1} and generalized in \cite{Stix});
  
\end{enumerate}

The aim of this and the subsequent article \cite{HES} is to generalize statements (\ref{item:FES-item-intro}) and (\ref{item:HES-item-intro}), correspondingly, to the case of $\pipet$. In the present article, we also establish the generalizations of all the other points besides (\ref{item:HES-item-intro}). The main difficulties in trying to directly generalize the proofs of Grothendieck are as follows:
\begin{itemize}
  \item geometric coverings of schemes (i.e.\ elements of $\rmCov_X$ defined above) are often not quasi-compact, unlike elements of $\mathrm{F\Et}_X$. For example, for $X$ a variety over a field $k$ and connected $Y \in \rmCov_{X_{\bk}}$, there may be no finite extension $l/k$ such that $Y$ would be defined over $l$. Similarly, some useful constructions (like Stein factorization) no longer work (at least without significant modifications).
  \item for a connected geometric covering $Y \in \rmCov_X$, there is in general no Galois geometric covering dominating it. Equivalently, there might exist an open subgroup $U < \pipet(X)$ that does not contain an open normal subgroup. This prevents some proofs that would work for $\pisga$ to carry over to $\pipet$.
  \item The topology of $\pipet$ is more complicated than the one of $\piet$, e.g.\ it is not necessarily compact, which complicates the discussion of exactness of sequences.
\end{itemize}

\subsection*{Our results}
Our main theorem is the generalization of the fundamental exact sequence. More precisely, we prove the following.
\begin{theorem*}[Thm.\ \ref{exactness-in-geometric-to-arithmetic-as-abstract}]
  Let $X$ be a geometrically connected scheme of finite type over a field $k$. Then the sequence
  \begin{displaymath}
  1 \rarr \pipet(X_{\bk}) \rarr \pipet(X) \rarr \Gal_k \rarr 1
  \end{displaymath}
  is exact as abstract groups. 
  
  Moreover, the map $\pipet(X_{\bk}) \rarr \pipet(X)$ is a topological embedding and the map $\pipet(X) \rarr \Gal_k$ is a quotient map of topological groups.
  \end{theorem*}
  The {\bf most difficult part is showing that $\pipet(X_{\bk}) \rarr \pipet(X)$ is} injective or, more precisely, {\bf a topological embedding}. This is Theorem \ref{injectivity-on-the-left}.

  As in the case of usual Galois categories, statements about exactness of sequences of Noohi groups translate to statements on the corresponding categories of $G-\rmSets$. If the groups involved are the pro-\'etale fundamental groups, this translates to statements about geometric coverings. We give a detailed dictionary in Prop.\ \ref{dictionary}. As Noohi groups are not necessarily compact, the statements on coverings are equivalent to some weaker notions of exactness (e.g.\ preserving connectedness of coverings is equivalent to the map of groups having dense image). In fact, we first prove a ``near-exact'' version of Theorem \ref{exactness-in-geometric-to-arithmetic-as-abstract} and obtain the above one as a corollary using an extra argument.

For $\pipet(X_{\bk}) \rarr \pipet(X)$ \uline{to be a topological embedding boils down to the following statement}: every geometric covering $Y$ of $X_{\bk}$ can be dominated by a covering $Y'$ that embeds into a base-change to $\bk$ of a geometric covering $Y''$ of $X$ (i.e.\ defined over $k$).

\[
  \xymatrix{
  Y' \subset Y''_{\bk} \ar[r] \ar@{->>}@<-15pt>[d] & Y''  \\
   Y \phantom{ ' \subset Y''_{\bk} } &
  }
\]

For finite coverings, the analogous statement is easy to prove; by finiteness, the given covering is defined over a finite field extension $l/k$ and one concludes quickly. This is also the case for infinite coverings detected by $\pisga$, see Prop.\ \ref{injectivity-for-prodiscrete}.  But for general geometric coverings, the situation is much less obvious; as we show by counterexamples (Ex.\ \ref{counterexample-with-picture} and Ex.\ \ref{counterexample-with-matrices}), {\bf it is \uline{not} true in general that a connected geometric covering of $X_{\bk}$ is isomorphic to a base-change of a covering of $X_l$ for some finite extension $l/k$}. This property is crucially used in the proof of \cite[Exp. IX, Theorem 6.1]{SGA1}, and thus {\bf trying to carry the classical proof of SGA over to $\pipet$ fails}. This last statement is, however, stronger than what we need to prove, and so does not contradict our theorem. 

A useful technical tool across the article is the \emph{van Kampen theorem} for $\pipet$. Its abstract form is proven by adapting the proof in \cite{Stix} to the case of Noohi groups and infinite Galois categories. For a morphism of schemes $X' \epirarr X$  of effective descent for $\rmCov$ (satisfying some extra conditions), it allows one to write the pro-\'etale fundamental group of $X$ in terms of the pro-\'etale fundamental groups of the connected components of $X'$ and certain relations. 
By the results of \cite{Rydh}, one can take $X' = X^\nu \to X$ to be the normalization morphism of a Nagata scheme $X$. As $\pipet$ and $\piet$ coincide for normal schemes, this allows us to present $\pipet(X)$ in terms of $\piet(X^\nu_w)$, where $X^\nu = \sqcup_w X^\nu_w$, and the (discrete) topological fundamental group of a suitable graph.
In this case, the van Kampen theorem takes on concrete form and generalizes \cite[Thm.\ 1.17]{ElenaPaper}.
\begin{theorem*}[van Kampen theorem, Cor.\ \ref{geomvK} + Rmk.\ \ref{remark-relations-in-vK-for-normalization} + Prop.\ \ref{properdescent}, cf.\ \cite{Stix}]
  Let $X$ be a Nagata scheme and $X^\nu = \sqcup_w X^\nu_w$ its normalization written as a union of connected components. Then, after a  choice of geometric points, \'etale paths between them and a maximal tree $T$ within a suitable ``intersection'' graph $\Gamma$, there is an isomorphism
  \[
    \pipet(X,\bx) \simeq  \Big(\big(*_w^{\rmtop} \piet(X^\nu_w,\bx_w)*^{\rmtop}\pi_1^{\rmtop}(\Gamma,T)\big)/\langle R_1, R_2 \rangle\Big)^{\rmNoohi}
  \]   
  where $R_1, R_2$ are two sets of relations described in Cor.\ \ref{geomvK} and $( - )^{\rmNoohi}$ is the \emph{Noohi completion} defined in Section \ref{section:infinite-Galois-Noohi}.
\end{theorem*}

In the proof of the main theorem, the van Kampen theorem allows us to construct $\pipet(X_{\bk})$- and $\pipet(X)$-sets in more concrete terms of graphs of groups involving $\piet$'s. We ``explicitly'' construct a Galois invariant open subgroup of a given open subgroup $U < \pipet(X_{\bk},\bx)$ in terms of {\bf ``regular loops"} (with respect to $U$), see Defn.\ \ref{defn:regular-loops}.

In fact, the existence of elements that are too far from being a product of regular loops is tacitly behind the counterexamples Ex.\ \ref{counterexample-with-picture} and \ref{counterexample-with-matrices}, while the fact that, despite this, there is still an abundance of (products of) regular loops (i.e.\ their closure is open) is behind our main proof. We also sketch a quicker but less constructive approach in Rmk.\ \ref{rmk:quick-approach}.

Another interesting result proven with the help of the van Kampen theorem is the K\"unneth formula.

\begin{proposition*}[K\"unneth formula for $\pipet$, Prop.\ \ref{Kunneth-proetale}]
  Let $X, Y$ be two connected schemes locally of finite type over an algebraically closed field $k$ and assume that $Y$ is proper. Let $\bx, \by$ be geometric points of $X$ and $Y$ respectively with values in the same algebraically closed field extension $K$ of $k$. Then the map induced by the projections is an isomorphism
  \begin{displaymath}
  \pipet(X \times_k Y,(\bx,\by)) \stackrel{\sim}{\rarr} \pipet(X,\bx) \times \pipet(Y,\by)
  \end{displaymath}
  \end{proposition*}

Along the way, we prove the invariance of $\pipet$ under extensions of algebraically closed fields for proper schemes (see Prop.\ \ref{algclosed-to-algclosed}) and give a short direct proof of the fact that $\pisga(X_{\bk},\bx) \hookrightarrow \pisga(X,\bx)$, see Cor.\ \ref{cor:injectivity-for-prodiscrete-cor}.

In a separate article \cite{HES}, we discuss the homotopy exact sequence for $\pipet$. It is proven by constructing an infinite (i.e.\ non-quasi-compact) analogue of the Stein factorization. Although the construction does not use the main results of this article, the auxiliary results on Noohi groups and $\pipet$ have proven to be very handy.

We hope that our techniques, with some extra tweaks and work, will allow to draw similar conclusions about other Noohi fundamental groups arising from the infinite Galois formalism. One such example could be the de Jong fundamental group $\pi_1^{\mathrm{dJ}}$, defined in the rigid-analytic setting in \cite{deJong}. In a later joint work \cite{ALY-specialization}, we have proven the existence of a specialization morphism between $\pipet$ and $\pi_1^{\mathrm{dJ}}$, relating $\pipet$ to this more established fundamental group.

\subsection*{Acknowledgements}
The main ideas and results contained in this article are a part of my PhD thesis. I express my gratitude to my advisor H\'el\`ene Esnault for introducing me to the topic and her constant encouragement. I would like to thank my co-advisor Vasudevan Srinivas for his support and suggestions. I am thankful to Peter Scholze for explaining some parts of his work to me via e-mail. I thank Jo\~ao Pedro dos Santos for for his comments and feedback. I owe special thanks to  Fabio Tonini, Lei Zhang and Marco D'Addezio from our group in Berlin for many inspiring mathematical  discussions. I thank Piotr Achinger and Jakob Stix for their support. I would also like to thank the referee for careful reading, valuable remarks and urging me to write a more streamlined version of the main proof.

My PhD was funded by the Einstein Foundation. This  work  is a  part  of  the  project ``Kapibara'' supported by the funding from the European Research Council (ERC) under the European Union’s Horizon 2020 research and innovation programme (grant agreement No 802787).

The major revision was prepared at JU Kraków and GU Frankfurt.
I was supported by the Priority Research Area SciMat under the
program Excellence Initiative - Research University at the Jagiellonian University in Kraków. This research was also funded by the Deutsche Forschungsgemeinschaft (DFG, German Research Foundation) TRR 326 Geometry and Arithmetic of Uniformized Structures, project number 444845124.

\subsection{Conventions and notations}
\begin{itemize}
  \item For us, compact = quasi-compact + Hausdorff.
    \item $H<^\circ G$ will mean that $H$ is an open subgroup of $G$.
    \item For subgroups $H<G$, $H^{nc}$ will denote the normal closure of $H$ in $G$, i.e.\ the smallest normal subgroup of $G$ containing $H$. We will use $\langle\langle - \rangle\rangle$ to denote the normal closure of the subgroup generated by some subset of $G$, i.e.\ $\langle\langle - \rangle\rangle = \langle - \rangle^{nc}$.
    \item For a field $k$, we will use $\bk$ to denote its (fixed) algebraic closure and  $k^{\rmsep}$ or $k^s$ to denote its separable closure (in $\bk$).
    \item The topological groups are assumed to be Hausdorff unless specified otherwise or appearing in a context where it is not automatically satisfied (e.g.\ as a quotient by a subgroup that is not necessarily closed). We will usually comment whenever a non-Hausdorff group appears.
    \item We assume (almost) every base scheme to be locally topologically noetherian. This does not cause problems when considering geometric coverings, as a geometric covering of a locally topologically noetherian scheme is locally topologically noetherian again - this is \cite[Lm.\ 6.6.10]{BhattScholze}.
    \item A ``$G$-set'' for a topological group $G$ will mean a discrete set with a continuous action of $G$ unless specified otherwise. We will denote the category of $G$-sets by $G-\rmSets$. We will denote the category of sets by $\rmSets$.
    \item We will often omit the base points from the statements and the discussion; by Cor.\ \ref{base-points}, this usually  does not change much. In some proofs (e.g.\ involving the van Kampen theorem), we keep track of the base points.
\end{itemize}

\section{Infinite Galois categories, Noohi groups and $\pipet$}\label{section:infinite-Galois-Noohi}
\subsection{Overview of the results in \cite{BhattScholze}}
Throughout the entire article we use the language and results of \cite{BhattScholze}, especially of Chapter 7, as this is where the pro-\'etale fundamental group was defined. Some familiarity with the results of \cite[\S 7]{BhattScholze} is a prerequisite to read this article. We are going to give a quick overview of some of these results below, but we recommend keeping a copy of \cite{BhattScholze} at hand.

\begin{defn}(\cite[Defn. 7.1.1]{BhattScholze})
Fix a topological group $G$. Let $G-\rmSets$ be the category of discrete sets with a continuous $G$-action, and let $F_G : G-\rmSets \rarr \rmSets$ be the forgetful functor. We say that $G$ is a \emph{Noohi group} if the natural map induces an isomorphism $G \rarr \rmAut(F_G)$ of topological groups. Here, $S \in \rmSets$ are considered with the discrete topology, $\rmAut(S)$ with the compact-open topology and $\rmAut(F_G)$ is topologized using $\rmAut(F_G(S))$ for $S \in G-\rmSets$.  More precisely, the stabilizers $\mathrm{Stab}_{F(S),s}^{\rmAut(F_G)}$ for connected $S \in G-\rmSets$, $s \in F(S)$, form a basis of neighbourhoods of $1 \in \rmAut(F_G)$.
\end{defn}

In particular, it follows from the definition that open subgroups form a basis of neighbourhoods of $1$ in a Noohi group. Now, by \cite[Prop.\ 7.1.5]{BhattScholze}, it follows that a topological group is Noohi if and only if it satisfies the following conditions:
\begin{itemize}
    \item its open subgroups form a basis of open neighbourhoods of $1 \in G$,
    \item it is Ra{\u \i}kov complete.
\end{itemize}
A topological group $G$ is Ra{\u \i}kov complete if it is complete for its two-sided uniformity (see \cite{Dikranjan} or \cite[Chapter 3.6]{AT} for an introduction to the Ra{\u \i}kov completion).
Using the above proposition it is easy to give examples of Noohi groups.
\begin{example}
The following classes of topological groups are Noohi: discrete groups, profinite groups, $\rmAut(S)$ with the compact-open topology for $S$ a discrete set (see \cite[Lm.\ 7.1.4]{BhattScholze}), groups containing an open subgroup which is Noohi (see \cite[Lm.\ 7.1.8]{BhattScholze}).

The following groups are Noohi: $\bbQ_\ell$, $\overline{\bbQ_\ell}$ for the colimit topology induced by expressing $\overline{\bbQ_\ell}$ as a union of finite extensions (in
contrast with the situation for the $\ell$-adic topology),  $\GL_n(\bbQ_\ell)$ for the colimit topology (see \cite[Example 7.1.7]{BhattScholze}).
\end{example}

The notion of a Noohi group is tightly connected to a notion of an infinite Galois category, which we are about to introduce. Here, an object $X \in \calC$ is called connected if it is not empty (i.e., initial), and for every subobject $Y \rarr X$ (i.e., $Y \stackrel{\sim}{\rarr}  Y \times_X Y$), either $Y$ is empty or $Y = X$.

\begin{defn}(\cite[Defn. 7.2.1]{BhattScholze}) An \emph{infinite Galois category} $\calC$ is a pair $(\calC,F : \calC \rarr \rmSets)$ satisfying:
\begin{enumerate}[label={(\arabic*)}]
    \item $\calC$ is a category admitting colimits and finite limits.
    \item Each $X \in \calC$ is a disjoint union of connected (in the sense explained above) objects.
    \item $\calC$ is generated under colimits by a set of connected objects.
    \item $F$ is faithful, conservative, and commutes with colimits and finite limits.
\end{enumerate}    
The \emph{fundamental group of} $(\calC,F)$ is the topological group $\pi_1(\calC,F) := \rmAut(F)$, topologized by the compact-open topology on $\rmAut(S)$ for any $S \in \rmSets$.

An infinite Galois category $(\calC,F)$ is \emph{tame} if for any connected $X \in \calC$, $\pi_1(\calC,F)$ acts transitively on $F(X)$.
\end{defn}

\begin{example} If $G$ is a topological group, then $(G-\rmSets,F_G)$ is a tame infinite Galois category.
\end{example}

\begin{theorem}(\cite[Thm.\ 7.2.5]{BhattScholze}) Fix an infinite Galois category $(\calC,F)$ and a Noohi group $G$. Then
  \begin{enumerate}
      \item $\pi_1(\calC,F)$ is a Noohi group.
      \item  There is a natural identification of $\Hom_{\mathrm{cont}}(G,\pi_1(\calC,F))$ with the groupoid of functors $\calC \rarr G-\rmSets$ that commute with the fibre functors.
      \item If $(\calC,F)$ is tame, then $F$ induces an equivalence $\calC \simeq \pi_1(\calC,F)-\rmSets$.
\end{enumerate}
\end{theorem}

The ``tameness'' assumption cannot be dropped as there exist infinite Galois categories that are not of the form $(G-\rmSets,F_G)$, see \cite[Ex.\ 7.2.3]{BhattScholze}. This was overlooked in \cite{Noohi}, where a similar formalism was considered.

\begin{rmk}
The above formalism was also studied in \cite[Chapter 4]{Lepage} under the names of ``quasiprodiscrete'' groups and ``pointed classifying categories''.
\end{rmk}

In Section \ref{section-Noohization} below we will study ``Noohi completion'' and the dictionary between Noohi groups and $G-\rmSets$ (see Section \ref{section-Dictionary}). For now, let us return to gathering the results from \cite{BhattScholze}.

\subsubsection*{Pro-\'etale topology and the definition of $\pipet(X)$} 
\begin{defn} Let $X$ be a locally topologically noetherian scheme.
Let $Y \rarr X$ be a morphism of schemes such that:
\begin{enumerate}
    \item it is \'etale (not necessarily quasi-compact!)
    \item it satisfies the valuative criterion of properness.
\end{enumerate}
We will call $Y$ a \emph{geometric covering} of $X$. We will denote the category of geometric coverings by $\rmCov_X$.
\end{defn}
As $Y$ is not assumed to be of finite type over $X$, the valuative criterion does not imply that $Y \rarr X$ is proper (otherwise we would simply get a finite \'etale morphism).

\begin{example}\label{covers-of-a-field}
For an algebraically closed field $\bk$, the category $\rmCov_{\Spec(\bk)}$ consists of (possibly infinite) disjoint unions of $\Spec(\bk)$ and we have $\rmCov_{\Spec(\bk)} \simeq \rmSets$.
\end{example}

More generally, one has:
\begin{lemma}(\cite[Lm.\ 7.3.8]{BhattScholze}) If $X$ is a henselian local scheme, then any $Y \in \rmCov_X$ is a disjoint union of finite \'etale $X$-schemes.
\end{lemma}

Let us choose a geometric point $\bx: \Spec(\bk) \rarr X$ on $X$. By Example \ref{covers-of-a-field}, this gives a fibre functor $F_{\bx}: \rmCov_X \rarr \rmSets$.
By \cite[Lemma 7.4.1]{BhattScholze}), the pair $(\rmCov_X,F_{\bx})$ is a tame infinite Galois category. Then one defines
\begin{defn}
The \emph{pro-\'etale fundamental group} is defined as
\begin{displaymath}
\pipet(X,\bx) = \pi_1(\rmCov_X,F_{\bx}).
\end{displaymath}
\end{defn}
In other words, $\pipet(X,\bx)=\rmAut(F_{\bx})$ and this group is topologized using the compact-open topology on $\rmAut(S)$ for any $S \in \rmSets$.

One can compare the groups $\pipet(X,\bx)$, $\piet(X,\bx)$ and  $\pi_1^{\mathrm{SGA3}}(X,\bx)$, where the last group is the group introduced in Chapter X.6 of \cite{SGA3vol2}.
\begin{lemma}For a scheme $X$, the following relations between the fundamental groups hold
\begin{enumerate}
    \item The group $\piet(X,\bx)$ is the profinite completion of $\pipet(X)$.
    \item The group $\pi_1^{\mathrm{SGA3}}(X,\bx)$ is the prodiscrete completion of $\pipet(X,\bx)$.
\end{enumerate}
\end{lemma}
\begin{proof}
This follows from \cite[Lemma 7.4.3]{BhattScholze} and \cite[Lemma 7.4.6]{BhattScholze}.
\end{proof}
As shown in \cite[Example 7.4.9]{BhattScholze}, $\pipet(X,\bx)$ is indeed more general than $\pi_1^{\mathrm{SGA3}}(X,\bx)$. This can be also seen by combining Example \ref{counterexample-with-picture} with Prop.\ \ref{injectivity-for-prodiscrete} below.

The following lemma is extremely important to keep in mind and will be used many times throughout the paper. Recall that, for example, a normal scheme is geometrically unibranch.
\begin{lemma}(\cite[Lm.\ 7.4.10]{BhattScholze})\label{proetale-of-normal}
If $X$ is geometrically unibranch, then $\pipet(X,\bx) \simeq \piet(X,\bx)$.
\end{lemma}

There is another way of looking at the pro-\'etale fundamental group, which justifies the name ``pro-\'etale''. 
\begin{defn}
\begin{enumerate}
  \item A map $f : Y \rarr X$ of schemes is called \emph{weakly \'etale} if $f$ is flat and the diagonal $\Delta_f : Y \rarr Y\times_XY$ is flat.
  \item The pro-\'etale site $X_\proet$ is the site of weakly \'etale $X$-schemes, with covers given by fpqc covers.  
\end{enumerate}  
\end{defn}
This definition of the pro-\'etale site is justified by a foundational theorem -- part \ref{olivier-item} of the following fact.
\begin{fact}
  Let $f : A \rarr B$ be a map of rings.
  \begin{enumerate}[label=\alph*)]
      \item $f$ is \'etale if and only if $f$ is weakly \'etale and finitely presented.
      \item If $f$ is ind-\'etale, i.e.\ $B$ is a filtered colimit of \'etale A-algebras, then $f$ is weakly \'etale.
      \item \label{olivier-item} (\cite[Theorem 2.3.4]{BhattScholze}) If $f$ is weakly \'etale, then there exists a faithfully flat ind-\'etale $g : B \rarr C$ such that $g\circ f$ is ind-\'etale.
  \end{enumerate}
\end{fact}

\begin{defn}(\cite[Defn. 7.3.1.]{BhattScholze}) 
We say that $F \in \rmShv(X_{\proet})$ is \emph{locally constant} if there exists a cover $\{Y_i \rarr X\}$ in $X_{\proet}$ with $F|_{Y_i}$ constant. We write $\rmLoc_X$ for the corresponding full subcategory of $\rmShv(X_{\proet})$.
\end{defn}

We are ready to state the following important result.
\begin{theorem}(\cite[Lemma 7.3.9.]{BhattScholze}) Let $X$ to be locally topologically noetherian scheme. One has $\rmLoc_X = \rmCov_X$ as subcategories of $\rmShv(X_{\proet})$.
\end{theorem}

\subsubsection*{Topological invariance of the pro-\'etale fundamental group}
We note that universal homeomorphisms of schemes induce equivalences on the corresponding categories of geometric coverings.

\begin{prop}(\cite[Lm. 5.4.2]{BhattScholze})\label{topoinv}
  Let $h: X' \rarr X$ be a universal homeomorphism of topologically noetherian schemes (i.e.\ induces a homeomorphism of topological spaces after any base-change). Then the pullback
  \begin{displaymath}
    h^* : \rmCov_{X} \rarr \rmCov_{X'} \text{ , } Y \mapsto Y' = Y \times_X X'
  \end{displaymath}
  is an equivalence of categories.
\end{prop}
\begin{proof}
As $\rmCov_X \simeq \rmLoc_X$, the theorem follows by the same proof as in \cite[Lm. 5.4.2]{BhattScholze}.

Alternatively, one can argue more directly (i.e.\ avoiding the equivalence with $\rmLoc_X$) as follows. By \cite[Theorem 04DZ]{StacksProject}, $V \mapsto V' = V \times_X X'$ induces an equivalence of categories of schemes \'etale over $X$ and schemes \'etale over $X'$. By \cite[Proposition 5.4.]{Rydh}, this induces an equivalence between schemes \'etale and separated over respectively $X$ and $X'$. The only thing left to be shown is that if for an \'etale separated scheme $Y \rarr X$, the map $Y \times_X X' \rarr X'$ satisfies the existence part of the valuative criterion of properness, then so does $Y \rarr X$. But this property can be characterized in purely topological terms (see \cite[Lemma 01KE]{StacksProject}) and so the result follows from the fact that $h$ is a universal homeomorphism.
\end{proof}

\subsection{Noohi completion}\label{section-Noohization}
Let $\mathrm{HausdGps}$ denote the category of Hausdorff topological groups (recall that we assume all topological groups to be Hausdorff, unless stated otherwise) and $\mathrm{NoohiGps}$ to be the full subcategory of Noohi groups. Let $G$ be a topological group. Denote $\calC_G= G - \rmSets$ and let $F_G : \calC_G \rightarrow \rmSets$ be the forgetful functor. Observe that $(\calC_G, F_G)$ is a tame infinite Galois category. Thus, the group $\rmAut(F_G)$ is a Noohi group.
It is easy to see that a morphism $G \rightarrow H$ defines an induced morphism of groups $\text{Aut}(F_G) \rightarrow \text{Aut}(F_H)$ and check that it is continuous.
Let $\psi_N:\text{HausdGps} \rightarrow \text{NoohiGps}$ be the functor defined by $G \mapsto \text{Aut}(F_G)$. Denote also the inclusion $i_N : \text{NoohiGps} \rightarrow \text{HausdGps}$. 
\begin{defn}\label{Noohization-defn}
We call $\psi_N(G)$ the Noohi completion of $G$ and will denote it $G^{\rmNoohi}$.
\end{defn} 
\begin{example}
In \cite[Example 7.2.6]{BhattScholze}, it was explained that the category of Noohi groups admits coproducts. Let $G_1, G_2$ be two Noohi groups and let $G_1*^NG_2$ denote their coproduct as Noohi groups. Let $G_1 *^{\rmtop} G_2$ be their topological coproduct. It exists and it is a Hausdorff group (\cite{Graev}). Then $G_1*^NG_2 = (G_1 *^{\rmtop} G_2)^{\rmNoohi}$.
\end{example}

Let $\alpha_G : G \to \rmAut(F_G) = G^{\rmNoohi}$ denote the obvious morphism.

\begin{proposition}\label{alpha_G-equivalence}
For a topological group $G$, the functor $F_G$ induces an equivalence of categories
\begin{displaymath}
\widetilde{F}_G : G - \rmSets \stackrel{\sim}{\rarr} G^{\rmNoohi} - \rmSets
\end{displaymath}
Moreover, $\alpha_G^* \circ \widetilde{F}_G \simeq \id$, and thus $\alpha^*$ is an equivalence of categories, too.
\end{proposition}
\begin{proof}
The first part follows directly from \cite[Theorem 7.2.5]{BhattScholze}. The natural isomorphism $\alpha_G^* \circ \widetilde{F}_G \simeq \id$ is clear from the definitions. It follows that $\alpha_G^*$ is an equivalence.
\end{proof}

The following lemma is in contrast with \cite[Remark 2.13]{Noohi}, but agrees with \cite[Proposition 4.1.1.]{Lepage}.
\begin{lemma}\label{dense-in-Noohization}
For any topological group $G$, the image of $\alpha_G: G \rarr G^{\rmNoohi}$ is dense.
\end{lemma}
\begin{proof} Let $U \subset G^{\rmNoohi}$ be open. As $G^{\rmNoohi}$ is Noohi, there exists $q \in G^{\rmNoohi}$ and an open subgroup $V <^\circ G^{\rmNoohi}$ such that $qV \subset U$. The quotient $G^{\rmNoohi}/V$ gives a $G^{\rmNoohi}$-set. It is connected in the category  $G^{\rmNoohi}-\rmSets$ and, by Prop.\ \ref{alpha_G-equivalence}, $\alpha_G^*(G^{\rmNoohi}/V)$ is connected. Thus, the action of $G$ on $G^{\rmNoohi}/V$ is transitive and so there exists $g\in G$ such that $\alpha_G(g)\cdot [V]=[qV]$, i.e.\ $\alpha_G(g) \in qV$. Thus, the image of $\alpha_G$ is dense.
\end{proof}
\begin{obs*}
Let $f:H \rarr G$ be a map of topological groups. Directly from the definitions, one sees that the following diagram commutes:
\begin{center}
\begin{tikzpicture}
\matrix(a)[matrix of math nodes,
row sep=2em, column sep=2em,
text height=1ex, text depth=0.25ex]
{H & G\\ H^{\rmNoohi} & G^{\rmNoohi}\\};

\path[->,font=\scriptsize] (a-1-1) edge node[left] {$\alpha_H$} (a-2-1);
\path[->,font=\scriptsize] (a-1-2) edge node[left] {$\alpha_G$} (a-2-2);
\path[->,font=\scriptsize] (a-1-1) edge node[above] {$f$} (a-1-2);
\path[->,font=\scriptsize] (a-2-1) edge node[above] {$f^{\rmNoohi}$} (a-2-2);

\end{tikzpicture}\\
\end{center}
\end{obs*}

\begin{lemma}[universal property of Noohi completion]
Let $f:H \rightarrow G$ be a continuous morphism from a topological group to a Noohi group. Then there exists a unique map $f': H^{\rmNoohi} \rightarrow G$ such that $f'\circ \alpha_H=f$.
\end{lemma}
\begin{proof}
By the definition of a Noohi group, $\alpha_G$ is an isomorphism. Defining $f':=\alpha_G^{-1}\circ f^{\rmNoohi}$ gives the existence. The uniqueness follows from $\alpha_H$ having dense image. Alternatively, one can combine Prop.\ \ref{alpha_G-equivalence} with \cite[Theorem 7.2.5(2)]{BhattScholze}.
\end{proof}

\begin{corollary}
The functor $\psi_N$ is a left adjoint of $i_N$.
\end{corollary}
\begin{rmk}
 There are few places, where we write $G^{\rmNoohi}$ for a non-Hausdorff group $G$. This is mostly to avoid a large overline sign over a subgroup described by generators. In these cases, we mean
 \[
  G^{\rmNoohi} := (G^{\mathrm{Hausd}})^{\mathrm{Noohi}}
 \]
 where $G^{\mathrm{Hausd}}$ is the maximal Hausdorff quotient. As $(-)^{\mathrm{Hausd}}$ is a left adjoint as well, this usually does not cause problems. This also provides a left adjoint to the forgetful functor $\mathrm{NoohiGps} \to \mathrm{TopGps}$ to all topological groups.
\end{rmk}

We now move towards a more explicit description of the Noohi completion.
\begin{lemma}\label{lem:weakened-top}
Let $(G,\tau)$ be a topological group. Denote by $\mathcal{B}$ the collection of sets of the form
\begin{displaymath}
x_1\Gamma_1 y_1 \cap x_2\Gamma_2 y_2 \cap \ldots \cap x_m\Gamma_m y_m
\end{displaymath}
where $m \in \bbN$, $x_i,y_i \in G$ and $\Gamma_i < G$ are open subgroups of $G$. Then $\mathcal{B}$ is a basis of a group topology $\tau'$ on $G$ that is weaker than $\tau$ and open subgroups of $(G,\tau)$ form a basis of open neighbourhoods of $1_G$ in $(G,\tau')$.

Moreover, the natural map $i':(G,\tau) \rightarrow (G,\tau')$ induces an equivalence of categories $(G,\tau')-\rmSets \rightarrow (G,\tau)-\rmSets$. If $\{1_G\} \subset (G,\tau)$ is thickly closed, i.e.\ $\cap_{U<^{\circ} G} U = \{1_G\}$ (see Defn. \ref{definition-thick-closure}), then $(G,\tau')$ is Hausdorff and $(G,\tau)^{\rmNoohi} \stackrel{\sim}{\rightarrow} (G,\tau')^{\rmNoohi}$ is an isomorphism.
\end{lemma}
\begin{proof} The first statement follows from \cite[Prop.\ III.1.1]{BourbakiGT} by taking the filter of subsets of $G$ containing an open subgroup. It is also proven in \cite[Lm.\ 1.13]{ElenaPaper} (the proposition is stated there in a particular case, but the proof works for any topological group). The second statement follows from the fact that for a discrete set $S$, any continuous morphism $(G,\tau)\rightarrow \Aut (S)$ factorizes through $i':(G,\tau)\rightarrow (G,\tau')$. 
\end{proof}

\begin{fact}(\cite[Prop.\ 7.1.5]{BhattScholze})\label{raikovandnoohi}
Let $G$ be a topological group such that its open subgroups form a basis of open neighbourhoods of $1_G$. Then $G^{\rmNoohi} \simeq \widehat{G}$, where $\widehat{G}$ denotes the Ra{\u \i}kov completion of $G$.
\end{fact}

\begin{proposition}\label{Noohization-description}
Let $(G,\tau)$ be a topological group. Assume that $\{1_G\} \subset (G,\tau)$ is thickly closed (see Defn. \ref{definition-thick-closure}). Then there is a natural isomorphism of groups 
\begin{displaymath}
G^{\rmNoohi} \simeq \widehat{(G,\tau')},
\end{displaymath}
where $\tau'$ denotes the topology described in the previous lemma and $\widehat{\ldots}$ denotes the Ra{\u \i}kov completion.
\end{proposition}
\begin{proof}
We combine Fact \ref{raikovandnoohi} with the last lemma and get $(G,\tau)^{\rmNoohi} \simeq (G,\tau')^{\rmNoohi} \simeq \widehat{(G,\tau')}$.
\end{proof}

\begin{obs}\label{setsofquotient}
Let $G$ be a topological group and $H$ a normal subgroup. Then the full subcategory of $G-\rmSets$ of objects on which $H$ acts trivially is equal to the full subcategory of $G-\rmSets$ on which its closure $\bar{H}$ acts trivially and it is equivalent to the category of $G/\bar{H} - \rmSets$. So, it is an infinite Galois category with the fundamental group equal to $(G/\bar{H})^{\rmNoohi}$.
\end{obs}

\begin{lemma}\label{noohiandtopology}
Let $X$ be a connected, locally path-connected, semilocally simply-connected topological space and $x \in X$ a point. Let $F_x$ be the functor taking a covering space $Y \rarr X$ to the fibre $Y_x$ over the point $x \in X$. Then $(\mathrm{TopCov}(X),F_x)$ is a tame infinite Galois category and $\pi_1(\mathrm{TopCov}(X),F_x)=\pi_1^{\rmtop}(X,x)$, where we consider $\pi_1^{\rmtop}(X,x)$ with the discrete topology. Here, $\mathrm{TopCov}(X)$ denotes denotes the category of covering spaces of $X$.
\end{lemma}
\begin{proof}
We first claim that there is an isomorphism: $(\text{TopCov}(X),F_x) \simeq (\pi_1^{\rmtop}(X,x)-\rmSets,F_{\pi_1^{\rmtop}(X,x)})$. This is in fact a classical result in algebraic topology, which can be recovered from \cite[Ch. 13]{Fulton} or \cite[Ch. 1]{Hatcher} and is stated explicitly in \cite[Cor.\ 4.1]{Cakar}. This finishes the proof, as discrete groups are Noohi.
\end{proof}

\subsection{Dictionary between Noohi groups and $G-\rmSets$}\label{section-Dictionary}
\begin{definition}\label{definition-thick-closure}
Let $H \subset G$ be a subgroup of a topological group $G$. Then we define a ``thick closure'' $\bbH$ of $H$ in $G$ to be the intersection of all open subgroups of $G$ containing $H$, i.e.\ $\bbH:=\bigcap_{H \subset U <^\circ G}$U.
If a subgroup satisfies $H = \bbH$ we will call it thickly closed in $G$.
\end{definition}

In a topological group open subgroups are also closed, so a thickly closed subgroup is also an intersection of closed subgroups, so it is closed in $G$. Observe also that an arbitrary intersection of thickly closed subgroups is thickly closed. This justifies, for example, the existence of the smallest normal thickly closed subgroup containing a given group. In fact, we can formulate a more precise observation.
\begin{obs}\label{normalqpNc-closed}\label{thickly-closed}
Let $H<G$ be a subgroup of a topological group $G$. Then the smallest normal thickly closed subgroup of $G$ containing $H$ is equal to $\overline{\overline{(H^{nc})}}$, where $H^{nc}$ is the normal closure of $H$ in $G$.
\end{obs}

\begin{obs}
  Let $G$ be a topological group such that the open subgroups form a local base at $1_G$. Let $W \subset G$ be a subset. Then the topological closure of $W$ can be written as $\overline{W} = \cap_{V <^\circ G} WV$.
  \end{obs}
  
  The following lemma can be found on p.79 of \cite{Lepage}.
  \begin{lemma}\label{lepage-lemma}
  Let $G$ be a topological group such that the open subgroups form a basis of neighbourhoods of $1_G$. Let $H \lhd G$ be a normal subgroup. Then 
  \begin{displaymath}
  \overline{H} = \bbH
  \end{displaymath}
  i.e.\ the usual topological closure and the thick closure coincide.
  \end{lemma}
  \begin{proof}
  We compute that $\overline{H} = \cap_{V <^\circ G} HV \stackrel{(*)}{\supset} \cap_{H < U <^\circ G} U = \bbH \supset \overline{H}$. The inclusion $(*)$ follows from the fact that $HV$ is an (open) subgroup of $G$ as $H$ is normal.
\end{proof}

Let us make an easy observation, that will be useful to keep in mind while reading the proof of the technical proposition below.
\begin{obs}
Let $U<G$ be an open subgroup of a topological group and let $g_0 \in G$. Then the mapping $G/g_0Ug_0^{-1} \rarr G/U$ given by $[gg_0Ug_0^{-1}] \mapsto [gg_0U]$ is an isomorphism of $G$-sets.

Given open subgroups $U,V <G$ and some surjective map of $G$-sets $\phi:G/V \epirarr G/U$ we can assume that it is the standard quotient map (i.e.\ $V \subset U$) up to replacing $U$ by a conjugate open subgroup (more precisely by $g_0Ug_0^{-1}$, where $g_0$ is such that $\phi([V])=[g_0U]$).
\end{obs}

\begin{rmk}
A map $Y' \rarr Y$ in an infinite Galois category $(\calC,F)$ is an epimorphism/monomorphism if and only if the map $F(Y') \rarr F(Y)$ is surjective/injective. Similarly $Y$ is an initial object if and only if $F(Y)=\emptyset$ and so on. The proofs of those facts are the same as the proofs in \cite[Tag 0BN0]{StacksProject}. This justifies using words ``injective'' or ``surjective'' when speaking about maps in $(\calC,F)$.
\end{rmk}

Recall the following fact.
\begin{obs*}
Let $f: G' \rarr G$ be a surjective map of topological groups. Then the induced morphism $G'/\ker(f) \rarr G$ is an isomorphism if and only if $f$ is open. In such case, we say that $f$ is a quotient map. In the  language of \cite[III.2.8]{BourbakiGT} we would call $f$ \emph{strict} and surjective.
\end{obs*}

\begin{defn}
We will say that an object of a tame infinite Galois category is \emph{completely decomposed} if it is a (possibly infinite) disjoint union of final objects.
\end{defn}

\begin{proposition}\label{dictionary}
  Let $G'' \stackrel{h'}\rightarrow G' \stackrel{h}\rightarrow G$ be maps between Noohi groups and $\calC_{G''} \stackrel{H'}\leftarrow \calC_{G'} \stackrel{H}\leftarrow \calC_G$ the corresponding maps of the infinite Galois categories. Then the following hold:
  
  \begin{enumerate}[label={(\arabic*)}]
  \item The map $h': G'' \rightarrow G'$ is a topological embedding if and only if for every connected object $X$ in $\calC_{G''}$, there exist connected objects $X' \in \calC_{G''}$ and $Y \in \calC_{G'}$ and maps $X' \twoheadrightarrow X$ and $X' \hookrightarrow H'(Y)$.
  
  \item \label{denseimageequivalentconditions} The following are equivalent
  \begin{enumerate}
  \item The morphism $h:G' \rarr G$ has dense image.
  \item The functor $H$ maps connected objects to connected objects.
  \item The functor $H$ is fully faithful.
  \end{enumerate}

  \item \label{dictionary-normal} The thick closure of $\Image (h') \subset G'$ is normal if and only if for every connected object $Y$ of $\calC_{G'}$ such that $H'(Y)$ contains a final object of $\calC_{G''}$, $H'(Y)$ is completely decomposed.
  
  \item $h'(G'') \subset \Ker(h)$ if and only if the composition $H' \circ H$ maps any object to a completely decomposed object.
  
  \item \label{dictionary-kernel} Assume that $h'(G'') \subset \ker(h)$ and that $h:G' \rarr G$ has dense image. Then the following conditions are equivalent:
  
  \begin{enumerate}
  \item the induced map $(G'/\ker(h))^{\rmNoohi} \rarr G$ is an isomorphism and the smallest normal thickly closed subgroup containing $\Image(h')$ is equal to $\ker(h)$,
  \item for any connected $Y \in \calC_{G'}$ such that $H'(Y)$ is completely decomposed, $Y$ is in the essential image of $H$,
  \item the induced map $(G'/\ker(h))^{\rmNoohi} \rarr G$ is an isomorphism and for any connected $Y \in \calC_{G'}$ such that $H'(Y)$ is completely decomposed, there exists $Z \in \calC_G$ and an epimorphism $H(Z) \twoheadrightarrow Y$. 
  \end{enumerate}
  \end{enumerate}
  \end{proposition}
  \begin{proof}
  \begin{enumerate}[label={(\arabic*)}]
  \item The proof is virtually the same as for usual Galois categories, but there every injective map is automatically a topological embedding (as profinite groups are compact). Assume that $G'' \rightarrow G'$ is a topological embedding. Let $X \in \calC_{G''}$ be connected and write $X \simeq G''/U$ for an open subgroup $U<G''$. Then there exists an open subset $\widetilde{V} \subset G'$ such that $\widetilde{V} \cap G''=U$ (as $G'' \rightarrow G'$ is a topological embedding) and an open subgroup $V < G'$ such that $V \subset \widetilde{V}$ (as $G'$ is Noohi). Denote $W=V \cap G''$. Then $X':=G''/W \twoheadrightarrow X$ and $X' \hookrightarrow H'(G'/V)$, so we conclude by setting $Y:=G'/V$. For the other implication: we want to prove that $G'' \rightarrow G'$ is a topological embedding under the assumption from the statement. It is enough to check that the set of preimages ${h'}^{-1}(\calB)$  of some basis $\calB$ of opens of $e_{G'}$ forms a basis of opens of $e_{G''}$. Indeed, assume that 
  this is the case. Firstly, observe that it implies that $h'$ is injective, as both $G''$ and $G'$ are Hausdorff (and in particular $T_0$). If $U$ is an open subset of $G''$, then we can write $U=\bigcup g_\alpha'' U_\alpha$ for some $g_\alpha'' \in G''$ and $U_\alpha \in {h'}^{-1}(\calB)$. We can write $U_\alpha=h'^{-1}(V_\alpha)$ for some $V_\alpha \in \calB$. Then $V=\bigcup h'(g_\alpha'') V_\alpha$ satisfies ${h'}^{-1}(V)=U$ because ${h'}^{-1}(h'(g_\alpha'') V_\alpha)=g_\alpha'' U_\alpha$ (by injectivity of $h'$). So this will prove that the topology on $G''$ is induced from $G'$ via $h'$. Let $\calB=\{U<G|U \textrm{ is open}\}$. This is a basis of opens of $e_{G'}$ (as $G'$ is Noohi). We want to check that $h'^{-1}(\calB)$ is a basis of opens of $e_{G''}$. As open subgroups of $G''$ form a basis of opens of $e_{G''}$ it is enough to show that for any open subgroup $U<G''$ there exists an open subgroup $V<G'$ such that $h'^{-1}(V) \subset U$. From the assumption we know that there exist open subgroups $\widetilde{U}<G''$ and $V<G'$ such that $G''/\widetilde{U} \epirarr G''/U$ and $G''/\widetilde{U} \monorarr  G'/V$. The surjectivity of the first map means that we can assume (up to replacing $\widetilde{U}$ by a conjugate) $\widetilde{U}\subset U$. The injectivity of the second means that we can assume (up to replacing $V$ by a conjugate) that $h'^{-1}(V) \subset \widetilde{U}$. Indeed, the injectivity implies that if $h'(g'')V=V$, then $g''\widetilde{U}=\widetilde{U}$ which translates immediately to $h'^{-1}(V) \subset \widetilde{U}$. So we have also $h'^{-1}(V) \subset U$, which is what we wanted to prove.
  
  \item The equivalence between (a) and (b) follows from the observation that a map between Noohi groups $G' \rightarrow G$ has a dense image if and only if for any open subgroup $U$ of $G$, the induced map on sets $G' \rightarrow G/U$ is surjective. Here, we only use that open subgroups form a basis of open neighbourhoods of $1_G \in G$.

  Now, the functor $H$ is automatically faithful and conservative (because $F_{G'}\circ H=F_G$ is faithful and conservative). Assume that (b) holds. Let $S,T \in G-\rmSets$ and let $g \in \Hom_{G'-\rmSets}(H(S),H(T))$. We have to show that $g$ comes from $g_0\in \Hom_{G-\rmSets}(S,T)$.  We can and do assume $S$, $T$ connected for that. Let $\Gamma_g \subset H(S)\times H(T)$ be the graph of $g$. It is a connected subobject. As $H(S)\times H(T)=H(S\times T)$, the assumption (b) implies that each connected component of $H(S)\times H(T)$ is the pullback of a connected component $\Gamma_0$ of $S\times T$. Thus, $\Gamma_g$ is the pullback of some $\Gamma_0 \subset S \times T$. By conservativity of $H$, the projection $p_{\Gamma_0}:\Gamma_0 \rarr S$ is an isomorphism, as this is true for $p_{\Gamma_g} : \Gamma_g \rarr H(S)$. The composition $q_{\Gamma_0}\circ p_{\Gamma_0}^{-1}: S \rarr T$ maps via $H$ to $g$.
  
  Conversely, assume (c) holds. Let $S \in G-\rmSets$ be connected. We want to show that $H(S)$ is connected. Suppose on the contrary that $H(S) = A \sqcup B$ with $A,B \in G'-\rmSets$. Let $T = \bullet \sqcup \bullet \in G-\rmSets$ be a two-element set with a trivial $G$-action. Then $\Hom_{G-\rmSets}(S,T)$ has precisely two elements, while $\Hom_{G'-\rmSets}(H(S), H(T))=\Hom_{G'-\rmSets}(A\sqcup B, \bullet \sqcup \bullet)$ has at least four.

  \item Assume first that the thick closure of $\rmim(h')$ is normal. Let $Y=G'/U$ be an element of $\calC_{G'}$ whose pull-back to $G'' - \rmSets$ contains the final object. This means that $G''$ fixes one of the classes, let's say $[g'U]$. This is equivalent to $g'^{-1}h'(G'')g'$ fixing $[U]$, i.e.\ $g'^{-1}h'(G'')g' \subset U$. But this implies immediately that $\overline{\overline{(g'^{-1}h'(G'')g')}} \subset U$. Let $\widetilde{g} \in G'$ be any element. We have $\overline{\overline{(g'^{-1}h'(G'')g')}} = g'^{-1}\overline{\overline{h'(G'')}}g'=\overline{\overline{h'(G'')}}=\widetilde{g}^{-1}\overline{\overline{h'(G'')}}\widetilde{g}$ from the assumption that $\overline{\overline{h'(G'')}}$ is normal. So $\widetilde{g}^{-1}h'(G'') \widetilde{g} \subset \widetilde{g}^{-1}\overline{\overline{h'(G'')}}\widetilde{g}\subset U$ and we conclude that $h'(G'')$ fixes an arbitrary class $[\widetilde{g}U]$. This shows that $G'/U$ pulls back to a completely decomposed object.
  
  The other way round: assume that for every connected object $Y$ of $\calC_{G'}$ such that $H'(Y)$ contains a final object, $H'(Y)$ is completely decomposed. Let $U$ be an open subgroup of $G'$ containing $h'(G'')$. Then $G''$ fixes $[U] \in G/U$ and so, by assumption, fixes every $[g'U] \in G/U$. This implies that for any $g' \in G'$ $g'^{-1}h'(G'')g' \subset U$ which easily implies that also $h'(G'')^{nc} \subset U$. As this is true for any $U$ containing $h'(G'')$ we get that $\overline{\overline{h'(G'')}}=\overline{\overline{(h'(G'')^{nc})}}$ and the last group is the smallest normal thickly closed subgroup of $G'$ containing $h'(G'')$ (Observation \ref{normalqpNc-closed}).
  \item The same as for usual Galois categories, we use that $\cap_{U <^\circ G} U = 1_G$.
  
  \item \begin{itemize}
  \item (b) $\Rightarrow$ (c): Assume (b). We only need to show, that $(G'/\ker(h))^{\rmNoohi} \rarr G$ is an isomorphism. This is equivalent to showing that $H$ induces an equivalence $G'/\ker(h)-\rmSets \simeq G-\rmSets$. As $G'/\ker(h) - \rmSets \simeq \{S \in G'-\rmSets|\ker(h) \text{ acts trivially on } S\} \subset \{S \in G-\rmSets | G'' \text{ acts trivially on } S \}$, the assumption of (b) implies that the functor $G-\rmSets \rarr G'/\ker(h) - \rmSets$ is essentially surjective. By the global assumption that $G' \rarr G$ has dense image, it is fully faithful (see \ref{denseimageequivalentconditions}).
  \item (c) $\Rightarrow$ (b): Assume (c). Let $Y \in \calC_{G'}$ be connected and such that $H'(Y)$ is completely decomposed. We have $Z \in \calC_G$ and an epimorphism $H(Z) \twoheadrightarrow Y$. As $\ker(h)$ acts trivially on $H(Z)$, we conclude that it also acts trivially on $Y$. Thus, by abuse of notation, $Y \in G'/\ker(h)-\rmSets$. But $G'/\ker(h)-\rmSets \simeq (G'/ \ker(h))^{\rmNoohi} - \rmSets \simeq G -\rmSets$ from the assumption. Thus, we see that $Y$ is in the essential image of $H$.

  \item (b) $\Rightarrow$ (a): Assume (b). We give two proofs of this fact.

  First proof: We have proven above that (b) $\Rightarrow$ $(G'/\ker(h))^{\rmNoohi} \simeq G$. Let $N$ be the smallest normal thickly closed subgroup of $G'$ containing $h'(G'')$. Observe that $N \subset \ker h$ (as $\ker(h)$ is thickly closed). Let $U$ be an open subgroup containing $N$. We want to show that $U$ contains $\ker h$. This will finish the proof as both $N$ and $\ker h$ are thickly closed. Write $Y=G'/U$. Observe that $G'/U$ pulls back to a completely decomposed $G''$-set if and only if for any $g' \in G'$ there is $g' h'(G'') {g'}^{-1} \subset U$. Indeed, $h'(G'')$ fixes $[g'U] \in G'/U$ if and only if $g'h'(G''){g'}^{-1}$ fixes $[U]$. So $N \subset U$ implies that $Y$ pulls back to a completely decomposed $G''$-set and, by assumption, $Y$ is isomorphic to a pull-back of some $G$-set and so $\ker(h)$ acts trivially on $Y$. This implies that $\ker h \subset U$, which finishes the proof.

  Alternative proof: We already know that (b) $\Rightarrow$ $(G'/\ker(h))^{\rmNoohi} \simeq G$. Let $N \subset \ker(h)$ be as in the first proof above. Consider the map $G/N \epirarr G/\ker(h)$. The assumption (b) and full faithfulness of $H$ (by the global assumption and using \ref{denseimageequivalentconditions}) imply that $(G'/N)^{\rmNoohi} \rarr G$ is an isomorphism. Thus, $(G'/N)^{\rmNoohi} \simeq (G'/\ker(h))^{\rmNoohi}$. Using Prop.\ \ref{Noohization-description}, we check that the canonical maps $G'/N \rarr (G'/N)^{\rmNoohi}$ and $G'/\ker(h) \rarr (G'/\ker(h))^{\rmNoohi}$ are injective. Thus, $G'/N \epirarr G'/\ker(h)$ is injective and so $N = \ker(h)$.
  \item (a) $\Rightarrow$ (b): Assume (a). Let $Y=G'/U$ be a connected $G'$-set that pulls back via $h'$ to a completely decomposed object. As we have seen while proving ``(b) $\Rightarrow$ (a)'', this implies that for any $g' \in G'$ $g' h'(G''){g'}^{-1} \subset U$, so $H^{nc} \subset U$ and so also $\overline{\overline{(H^{nc})}} \subset U$. But, by Observation \ref{normalqpNc-closed}, there is $N=\overline{\overline{(H^{nc})}}$. By assumption, we have $N=\ker h$ and so we conclude that $\ker h \subset U$. But then, by assumption $(G'/\ker(h))^{\rmNoohi} \simeq G$, $Y$ is in the essential image of $H$.
  \end{itemize}
  \end{enumerate}
  \end{proof}
  
  To distinguish between exactness in the usual sense (i.e.\ on the level of abstract groups) and notions of exactness appearing in Prop.\ \ref{dictionary}, we introduce a new notion. It will be mainly used in the context of Noohi groups.
\begin{defn}\label{definition-weakly-exact}
Let $G'' \stackrel{h'}{\rarr} G' \stackrel{h}{\rarr} G \rarr 1$ be a sequence of topological groups such that $\rmim(h') \subset \ker(h)$. Then we will say that the sequence is
\begin{enumerate}[label={(\arabic*)}]
    \item \emph{nearly exact on the right} if $h$ has dense image,
    \item \emph{nearly exact in the middle} if $\overline{\overline{\rmim(h')}}= \ker(h)$, i.e.\ the thick closure of the image of $h'$ in $G'$ is equal to the kernel of $h$,
    \item \emph{nearly exact} if it is both nearly exact on the right and nearly exact in the middle.
\end{enumerate}
\end{defn}

We end this subsection with a lemma on topological groups and their Noohi completions that will be used later in the proof of the main theorem.
\begin{lemma}\label{lem:V^Noohi-description}
  Let $G$ be a topological group and $\tilde{G}$ be a subgroup of $G^{\rmNoohi}$ such that the canonical map $G \to G^{\rmNoohi}$ factorizes through $\tilde{G}$ 
  \[
    G \to \tilde{G} \subset G^{\rmNoohi}
  \]
Let $V_0 < \tilde{G}$ be a subgroup. Let $S = (\tilde{G}/V_0,\rmdiscr)$ be the discrete set that comes naturally with an {\bf abstract} action by $\tilde{G}$.

If the induced abstract $G$-action on $S$ is continuous, then $V_0$ is open in $\tilde{G}$.

Moreover, in such case, denoting  $V = \mathrm{Stab}_{G^{\rmNoohi}}([V_0] \in \tilde{G}/V_0)$, there is
\[
  V = V_0^{\rmNoohi} = \overline{V_0}^{G^{\rmNoohi}} \quad \textrm{ and } \quad V_0 = V \cap \tilde{G}
\]
\end{lemma}
\begin{proof}
  By the universal property, the $G$-action on $S$ extends to $G^{\rmNoohi}$ and this action is transitive. Then $V_0$ is the preimage of the stabilizer $V = \mathrm{Stab}_{G^{\rmNoohi}}([V_0] \in \tilde{G}/V_0)$, which is open.

The group $V$ is open in a Noohi group, thus Noohi (see \cite[Lemma 7.1.8.]{BhattScholze}). By the universal property, there is a factorization $V_0^{\rmNoohi} \to V$. But as $V_0$ is a subgroup of a Noohi group, its open subgroups form a basis of $1_{V_0}$. Thus, the Noohi completion of $V_0$ is just the Ra{\u \i}kov completion. But as the canonical map from a group to its Ra{\u \i}kov completion is a topological embedding, \cite[Cor.\ 3.6.18.]{AT} implies that $V_0^{\rmNoohi} \to V$ is a topological embedding. By a characterization of Ra{\u \i}kov completeness (see \cite[Prop.\ 6.2.7.]{Dikranjan}), it follows that $V_0^{\rmNoohi}$ is closed in $V$. But as $\tilde{G}$ contains the image of $G$, it is dense in $G^{\rmNoohi}$, and from the definition of $V$ it follows that $V_0$ has to be dense in $V$. Putting this together, we get that $V_0^{\rmNoohi} = V = \overline{V_0}^{G^{\rmNoohi}}$.
\end{proof}

\subsection{A remark on valuative criteria}
We will sometimes shorten ``the valuative criterion of properness'' to ``VCoP''. It is useful to keep in mind the precise statements of different parts of the valuative criterion, see \cite[Lemma 01KE]{StacksProject}, \cite[Section 01KY]{StacksProject} and \cite[Lemma 01KC]{StacksProject}. Let us prove a lemma (which is implicit in \cite{BhattScholze}), that VCoP chan be checked fpqc-locally.

\begin{lemma}\label{VCoP-is-local}
Let $g:X \rarr S$ be a map of schemes. The properties: 
\begin{enumerate}[label=(\alph*)]
\item $g$ is \'etale
\item $g$ is separated
\item $g$ satisfies the existence part of VCoP \label{VCoPproperty}
\end{enumerate}
can be checked fpqc-locally on $S$.

Moreover, the property \ref{VCoPproperty} can be also checked after a surjective proper base-change.
\end{lemma}
\begin{proof}
The cases of \'etale and separated morphisms are proven in \cite[Section 02YJ]{StacksProject}. For the last part: satisfying the existence part of VCoP is equivalent to specializations lifting along any base-change of $g$ (\cite[Lemma 01KE]{StacksProject}). It is easy to see that this property can be checked Zariski locally. Thus, if $S' \rarr S$ is an fpqc cover such that the base-change $g': X' \rarr S'$ satisfies specialization lifting for any base-change, we can assume that $S, S'$ are affine with $S' \rarr S$ faithfully flat. Let $T \rarr S$ be any morphism. Consider the diagram:
\begin{center}
\begin{tikzpicture}
\matrix(a)[matrix of math nodes,
row sep=2em, column sep=2.5em,
text height=2ex, text depth=0.25ex]
{X' & S' \times_S X \times_S T & T \times_S X\\ S'  & S'\times_S T & T \\};

\path[->] (a-1-2) edge (a-1-1);
\path[->] (a-1-2) edge (a-1-3);

\path[->] (a-2-2) edge (a-2-1);
\path[->] (a-2-2) edge (a-2-3);

\path[->] (a-1-1) edge (a-2-1);
\path[->] (a-1-2) edge (a-2-2);
\path[->] (a-1-3) edge (a-2-3);

\end{tikzpicture}
\end{center}

Let $\xi'\in T \times_S X$, let $\xi$ be its image in $T$ and let $t \in T$ be such that $\xi \rightsquigarrow t$. We need to find $t' \in T \times_S X$ over $t$ such that $\xi' \rightsquigarrow t'$. Let $Z= \overline{\{\xi'\}} \subset T \times_S X$ be the closure of $\{\xi'\}$. We need to show that the set-theoretic image $W \subset T$ of $Z$ in $T$ contains $t$. It is enough to show, that $W$ is stable under specialization or, equivalently, that $T \setminus W$ is stable under generalization. But, from flatness (\cite[Lemma 03HV]{StacksProject}), generalizations lift along $S' \times_S T \rarr T$. Thus, it is enough to show that the preimage of $T \setminus W$ in $S' \times_S T$ is stable under generalizations or, equivalently (using the surjectivity of $S' \times_S T \rarr T$), that the preimage of $W$ in $S' \times_S T$ is closed under specializations. But an easy diagram chasing (using the fact that the right square of the diagram above is cartesian) shows that the preimage of $W$ in $S' \times_S T$ is the image of a closed subset of $S' \times_S X \times T$. We conclude, because specializations lift along $S' \times_S X \times_S T \rarr S' \times_S T$ by assumption.

The last part of the statement is proven in an analogous way.
\end{proof}

\begin{lemma}\label{separated}
Let $f:Y \rarr X$ be a geometric covering of a locally topologically noetherian scheme. Then $f$ is separated.
\end{lemma}
\begin{proof}
By \cite[Remark 7.3.3]{BhattScholze}, $f$ is quasi-separated. A quasi-separated morphism satisfying VCoP is separated (see \cite[Tag 01KY]{StacksProject}).
\end{proof}

\section{Seifert-van Kampen theorem for $\pipet$ and its applications}

\subsection{Abstract Seifert–van Kampen theorem for infinite Galois categories}
We aim at recovering a general version of van Kampen theorem, proven in \cite{Stix}, in the case of the pro-\'etale fundamental group. Most of the definitions and proofs are virtually the same as in \cite{Stix}, after replacing ``Galois category'' with ``(tame) infinite Galois category'' and ``profinite'' with ``Noohi'', but still some additional technical difficulties appear here and there. We make the necessary changes in the definitions and deal with those difficulties below.

Denote by $\Delta_{\leq 2}$ a category whose objects are $[0]=\{0\}, [1]=\{0,1\}, [2]=\{0,1,2\}$ and has strictly increasing maps as morphisms. There are face maps $\partial_i:[n-1] \rarr [n]$ for $n=1,2$ and $0 \leq i \leq n$ which omit the value $i$ and vertices $v_i:[0]\rarr [2]$ with image $i$.

The category of \emph{$2$-complexes in a category $\scrC$} is the category of contravariant functors $T_\bullet: \Delta_{\leq 2} \rarr \scrC$. We denote $T_n=T_\bullet([n])$ and call it \emph{the $n$-simplices of $T_\bullet$}. $T(\partial_i)$ is called the \emph{$i$-th boundary map}.

By a \emph{$2$-complex $E$} we mean a $2$-complex in the category of sets. We often think of $E$ as a category: its objects are the elements of $E_n$ for $n=0,1,2$ and its morphisms are obtained by defining $\partial: s \rarr t$ where $s \in E_n$ and $t=E(\partial)(s)$. Let $\Delta_n=\{\sum_{i=0}^n \lambda_i e_i \in \bbR_{\geq 0}^{n+1}|\sum_i \lambda_i = 1\}$ denote the topological $n$-simplex. Then we define $|E|=\bigsqcup E_n\times \Delta_n/\sim$, where $\sim$ identifies $(s,d(x))$ with $(E(\partial)(s),x)$ for all $\partial : [m] \rarr [n]$ and its corresponding linear map $d : \Delta_m \rarr \Delta_n$ sending $e_i$ to $e_{\partial(i)}$, and $s \in E_n$ and $x \in \Delta_m$. We call $E$ \emph{connected} if $|E|$ is a connected topological space.

\begin{defn}
\emph{Noohi group data} $(\scrG,\alpha)$ on a $2$-complex $E$ consists of the following:
\begin{enumerate}
\item A mapping (not necessarily a functor!) $\scrG$ from the category $E$ to the category of Noohi groups: to a complex $s \in E_n$ is attributed a Noohi group $\scrG(s)$ and to a map $\partial : s \rarr t$ is attached a continuous morphism $\scrG(\partial):\scrG(s) \rarr \scrG(t)$.
\item For every triple $v \in E_0$, $e \in E_1$, $f \in E_2$ and boundary maps $\partial',\partial$ such that $\partial'(f)=e$, $\partial(e)=v$, an element $\alpha_{vef} \in \scrG(v)$ (its existence is a part of the definition) such that the following diagram commutes:
\begin{center}
\begin{tikzpicture}
\matrix(a)[matrix of math nodes,
row sep=3em, column sep=2em,
text height=1.5ex, text depth=0.25ex]
{\scrG(f) & & \scrG(e) \\ \scrG(v) & &  \scrG(v) \\};
\path[->,font=\scriptsize] (a-1-1) edge node[above] {$\scrG(\partial')$} (a-1-3);
\path[->,font=\scriptsize] (a-1-1) edge node[left] {$\scrG(\partial\partial')$} (a-2-1);
\path[->,font=\scriptsize] (a-1-3) edge node[right] {$\scrG(\partial)$} (a-2-3);
\path[->,font=\scriptsize] (a-2-1) edge node[above] {$\alpha_{vef}(\cdot)\alpha_{vef}^{-1}$} (a-2-3);
\end{tikzpicture}
\end{center}
\end{enumerate}
\end{defn}

\begin{defn}
Let $(\scrG,\alpha)$ be Noohi group data on the $2$-complex $E$. A \emph{$(\scrG,\alpha)$-system} $M$ on $E$ consists of the following:
\begin{enumerate}[label={(\arabic*)}]
\item For every simplex $s \in E$ a $\scrG(s)$-set $M_s$.
\item For every boundary map $\partial : s \rarr t$ a map of $\scrG(s)$-sets $m_\partial : M_s \rarr \scrG(\partial)^*(M_t)$, such that:
\item for every triple $v \in E_0$, $e \in E_1$, $f \in E_2$ and boundary maps $\partial',\partial$ such that $\partial'(f)=e$, $\partial(e)=v$ the following diagram commutes
\begin{center}
\begin{tikzpicture}
\matrix(a)[matrix of math nodes,
row sep=3em, column sep=2em,
text height=1.5ex, text depth=0.25ex]
{M_f & M_e \\ M_v & M_v \\};
\path[->,font=\scriptsize] (a-1-1) edge node[above] {$m_{\partial'}$} (a-1-2);
\path[->,font=\scriptsize] (a-1-1) edge node[left] {$m_{\partial\partial'}$} (a-2-1);
\path[->,font=\scriptsize] (a-1-2) edge node[right] {$m_{\partial}$} (a-2-2);
\path[->,font=\scriptsize] (a-2-1) edge node[above] {$\alpha_{vef} \cdot$} (a-2-2);
\end{tikzpicture}
\end{center}
\end{enumerate}
\end{defn}
\begin{defn}
A $(\scrG,\alpha)$-system is called \emph{locally constant} if all the maps $m_\partial$ are bijections.
\end{defn}

Observe that $\alpha \cdot : m \mapsto \alpha m$ is a $\scrG(v)$-equivariant map $M_v \rarr (\alpha()\alpha^{-1})^*M_v$. Observe that there is an obvious notion of a morphism of $(\scrG,\alpha)$-systems: a collection of $\scrG(s)$-equivariant maps that commute with the $m$'s. Let us denote by $\rmlcs(E,(\scrG, \alpha))$ the category of locally constant $(\scrG,\alpha)$-systems.

Let $M \in \rmlcs(\scrG,\alpha)$ for Noohi group data $(\scrG,\alpha)$ on some $2$-complex $E$. We define oriented graphs $E_{\leq 1}$ and $M_{\leq 1}$ (which will be an oriented graph \emph{over} $E_{\leq 1}$) as in \cite{Stix}, but our graphs $M_{\leq 1}$ are possibly infinite. For $E_{\leq 1}$ the vertices are $E_0$ and edges $E_1$ such that $\partial_0$ (resp. $\partial_1$) map an edge to its target (resp. origin). For $M_{\leq 1}$ the vertices are $\bigsqcup_{v\in E_0} M_v$ and edges are $\bigsqcup_{e\in E_1} M_e$ serves as the set of edges. The target/origin maps are induced by the $m_\partial$ and the map $M_{\leq 1} \rarr E_{\leq 1}$ is the obvious one.

There is an obvious topological realization functor for graphs $|\cdot |$. By applying this functor to the above construction we get a \underline{topological covering} (because $M$ is locally constant) $|M_{\leq 1}| \rightarrow |E_{\leq 1}|$. This gives a functor
\begin{displaymath}
|\cdot_{\leq 1}|: \rmlcs(E,(\scrG,\alpha)) \rightarrow  \mathrm{TopCov}(|E_{\leq 1}).
\end{displaymath}
Choosing a maximal subtree $T$ of $|E_{\leq 1}|$ gives a fibre functor $F_T:\mathrm{TopCov}(|E_{\leq 1}|) \rightarrow \rmSets$ by $(p:Y \rarr |E_{\leq 1}|) \mapsto \pi_0(p^{-1}(|T|))$. The pair $(\mathrm{TopCov}(|E_{\leq 1}|),F_T)$ is an infinite Galois category and the resulting fundamental group $\pi_1($Cov$(|E_{\leq 1}|),F_T))$ is isomorphic to $\pi_1^\rmtop(|E_{\leq 1}|)$ (see Lemma \ref{noohiandtopology}) which is in turn isomorphic to $\rmFr(E_1)/\langle\langle \{\vec{e}| e \in T\}^{\rmFr(E_1)} \rangle\rangle = \rmFr(\vec{e}| e \in E_1 \setminus T)$, where $\rmFr(\ldots)$ denotes a free group on the given set of generators and $\langle\langle \{\vec{e}| e \in T\}^{Fr(E_1)} \rangle\rangle$ denotes the normal closure in $\rmFr(E_1)$ of the subgroup generated by $\{\vec{e} \in T\}$. Here, $\vec{e}$ acts on $F_T(M)$ via
\begin{displaymath}
\pi_0(p^{-1}(|T|)) \cong \pi_0(p^{-1}(\partial_0(e))) \cong \pi_0(p^{-1}(|e|)) \cong \pi_0(p^{-1}(\partial_1(e)) \cong \pi_0(p^{-1}(|T|))
\end{displaymath}

As in \cite{Stix}, for every $s \in E_0$ and $M \in $ lcs$(E,(\scrG,\alpha))$ we have that $F_T(M)$ can be seen canonically as a $\scrG(s)$-module by $M_s = \pi_0(p^{-1}(s))\cong\pi_0(p^{-1}(T))$. Denote $\pi_1(E_{\leq 1},T)=\rmFr (E_1)/\langle\langle \{\vec{e}|e \in T \}^{\rmFr(E_1)}\rangle\rangle$. Putting the above together we get a functor
\begin{displaymath}
Q: \text{lcs}(E,(\scrG,\alpha)) \rightarrow (*^N_{v \in E_0}\scrG(v)*^N \pi_1(E_{\leq 1}, T))-\text{sets}
\end{displaymath}
\begin{rmk}
In the setting of usual ("finite") Galois categories, it is usually enough to say that a particular morphism between two Galois categories is exact, because of the following fact (\cite[Tag 0BMV]{StacksProject}): Let $G$ be a topological group. Let $F : \mathrm{Finite-}G\mathrm{-Sets} \rarr \rmSets$ be an exact functor with $F(X)$ finite for all $X$. Then $F$ is isomorphic to the forgetful functor. 

As we do not know if an analogous fact is true for infinite Galois categories, given two infinite Galois categories $(\calC,F)$, $(\calC',F')$ and a morphism $\phi:\calC \rarr \calC'$, we are usually more interested in checking whether $F \simeq F' \circ \phi$. If $\phi$ satisfies this condition, it also commutes with finite limits and arbitrary colimits. Indeed, we have a map $\mathrm{colim} \phi(X_i) \rarr \phi(\mathrm{colim}X_i)$ that becomes an isomorphism after applying $F'$ (as $F'$ and $F=F' \circ \phi$ commute with colimits) and we conclude by conservativity of $F'$. Similarly for finite limits.
\end{rmk}

\begin{proposition}
Let $(E,(\scrG, \alpha))$ be a connected 2-complex with Noohi group data. Define a functor $F: \rmlcs(E,(\scrG,\alpha)) \rightarrow $ Sets in the following way: pick any simplex $s$ and define $F$ by $M \mapsto M_s$. Then ($\rmlcs(E,(\scrG,\alpha)),F)$ is a tame infinite Galois category.

Moreover, the obtained functor
\begin{displaymath}
Q: \text{lcs}(E,(\scrG,\alpha)) \rightarrow (*^N_{v \in E_0}\scrG(v)*^N \pi_1(E_{\leq 1}, T))-\text{sets}
\end{displaymath}
satisfies $F \simeq F_{\mathrm{forget}} \circ Q$ and maps connected objects to connected objects.
\end{proposition}

\begin{proof}
We first check conditions (1), (2) and (4) of \cite[Def. 7.2.1]{BhattScholze}. Then we show that $Q$ maps connected objects to connected objects and we use the proof of this last fact to show the condition (3).

Colimits and finite limits: they exist simplexwise and taking limits and colimits is functorial so we get a system as candidate for a colimit/finite limit. This will be a locally constant system, as the colimit/finite limit of bijections between some $G$-sets is a bijection.

Each $M$ is a disjoint union of connected objects: let us call $N \in \rmlcs(\scrG,\alpha)$ a \emph{subsystem} of $M$ if there exists a morphism $N \rarr M$ such that for any simplex $s$ the map $N_s \rarr M_s$ is injective (we then identify, for any simplex $s$, $N_s$ with a subset of $M_s$). We can intersect such subsystems in an obvious way and observe that it gives another subsystem. So for any element $a \in M_v$ there exists the smallest subsystem $N$ of $M$ such that $a \in N_v$. We see readily that for any vertices $v,v'$ and $a \in M_v, a' \in M_{v'}$ the smallest subsystems $N$ and $N'$ containing one of them are either equal or disjoint (in the sense that, for each simplex $s$, $N_s$ and $N_s'$ are disjoint as subsets of $M_s$). It is easy to see that in this way we have obtained a decomposition of $M$ into a disjoint union of connected objects.

$F$ is faithful, conservative and commutes with colimits and finite limits: observe that $\phi_s:$ lcs$(E,(\scrG,\alpha)) \ni M \mapsto M_s \in \scrG(s)-\rmSets$ is faithful, conservative and commutes with colimits and finite limits and $F=F_s \circ \phi_s$, where $F_s$ is the usual forgetful functor on $\scrG(s)-\rmSets$.

It is obvious that $F \simeq F_{\mathrm{forget}} \circ Q$. We are now going to show that $Q$ preserves connected objects. Take a connected object $M \in \text{lcs}(E,(\scrG,\alpha))$ and suppose that $N$ is a non-empty subset of $F_T(M)$ stable under the action of $\pi_1(E_{\leq 1},T)$ and $\scrG(v)$ for $v \in E_0$. Stability under the action of $\pi_1(E_{\leq 1},T)$ shows that $N$ can be extended to a subgraph $N_{\leq 1} \subset M_{\leq 1}$: for an edge $e$ of $M_{\leq 1}$ we declare it to be an edge of $N_{\leq 1}$ if one of its ends touches a connected component of $p^{-1}(|T|)$ corresponding to an element of $N$. This is well defined, as in this case both ends touch such a component - this is because the action of $m_{\partial_1}m_{\partial_0}^{-1}$ equals the action of $\stackrel{\rightarrow}{e} \in \pi_1(E_{\leq 1},T)$.

Now we want to show that it extends to $2$-simplexes. 
This is a local question and we can restrict to simplices in the boundary of a given face $f \in E_2$. Define $N_f$ as a preimage of $N_s$ via any $\partial$ such that $\partial(f)=s$. We see that if the choice is independent of $s$, then we have extended $N$ to a locally constant system. To see the independence it is enough to prove that if $(vef)$ is a barycentric subdivision (i.e.\ we have $\partial$ and $\partial'$ such that $\partial'(f)=e$ and $\partial(e)=v$), then $m_{\partial\partial'}^{-1}(N_v)=m_{\partial'}^{-1}(N_e)$. But from the $\scrG(v)$-invariance we have $N_v=\alpha_{vef}^{-1}(N_v)$ and so
\begin{displaymath}
m_{\partial\partial'}^{-1}(N_v)=m_{\partial\partial'}^{-1}(\alpha_{vef}^{-1}(N_v))=m_{\partial'}^{-1}m_\partial^{-1}(N_v)=m_{\partial'}(N_e)
\end{displaymath}
and thus $N$ can be seen as an element of lcs$(E,(\scrG,\alpha))$ which is a subobject of $M$, which contradicts connectedness of $M$.

To see that $\rmlcs(E,(\scrG,\alpha))$ is generated under colimits by a set of connected objects, observe that in the above proof of the fact that $Q$ preserves connected objects, we have in fact shown the following statement. 
\begin{fact}
Let $M \in \rmlcs(\scrG,\alpha)$ and let $Z$ be a connected component of $Q(M)$. Then there exists a subsystem $W \subset M$ such that $Q(W)=Z$. 
\end{fact}
We want to show that there exists a \underline{set} of connected objects in $\rmlcs(\scrG,\alpha)$ such that any connected object of $\rmlcs(\scrG,\alpha)$ is isomorphic to an element in that set. As an analogous fact is true in $(*^N_{v \in E_0}\scrG(v)*^N \pi_1(E_{\leq 1}, T))-\mathrm{sets}$, it is easy to see that it is enough to check that, for any $X,Y$, if $QX \simeq QY$, then $X \simeq Y$. Let $X, Y \in \rmlcs(\scrG,\alpha)$ be connected. Assume that $QX \simeq QY$. Looking at the graph of this isomorphism, we find a connected subobject $Z \subset QX \times QY$ that maps isomorphically on $QX$ and $QY$ via the respective projections. By the above fact, we know that there exists $W \subset X \times Y$ such that $QW=Z$. Because $F \simeq F_{\mathrm{forget}} \circ Q$ and $F$ is conservative, we see that the projections $W \rarr X$ and $W \rarr Y$ must be isomorphisms. This shows $X \simeq Y$ as desired.

The only claim left is that $\rmlcs(E(\scrG,\alpha))$ is tame, but this follows from tameness of $(*^N_{v \in E_0}\scrG(v)*^N \pi_1(E_{\leq 1}, T))-\rmSets$, the equality $F\simeq F_{\mathrm{forget}}\circ Q$ and the fact that $Q$ maps connected objects to connected objects. 
\end{proof}
 
Let us denote by $\pi_1(E,\scrG,s)$ the fundamental group of the infinite Galois category $(\rmlcs(E,\scrG),F_s)$. The proposition above tells us that there is a continuous map of Noohi groups with dense image $*^N_{v \in E_0}\scrG(v)*^N \pi_1(E_{\leq 1}, T)) \rarr \pi_1(E,\scrG,s)$. We now proceed to describe the kernel.

Recall that $\pi_1(E_{\leq 1},T)=\rmFr (E_1)/\langle\langle \{\vec{e}|e \in T \}^{\rmFr(E_1)}\rangle\rangle$.
\begin{theorem}(abstract Seifert-Van Kampen theorem for infinite Galois categories)
Let $E$ be a connected $2$-complex with group data $(\scrG,\alpha)$. With notations as above, the functor $Q$ induces an isomorphism of Noohi groups
\begin{displaymath}
(*^N_{v \in E_0} \scrG(v) *^N \pi_1(E_{\leq 1},T)/\bar{H})^{\rmNoohi} \rarr \pi_1(E,\scrG,s)
\end{displaymath}
where $\bar{H}$ is the closure of the group
\begin{displaymath}
H=\left\langle\left\langle\begin{array}{c|c}
\scrG(\partial_1)(g)\vec{e} =\vec{e} \scrG(\partial_0)(g) & e \in E_1 \textrm{ , } g \in \scrG(e) \\ 
\overrightarrow{(\partial_2f)}\alpha^{(f)}_{102}(\alpha^{(f)}_{120})^{-1}\overrightarrow{(\partial_0f)}\alpha^{(f)}_{210}(\alpha^{(f)}_{201})^{-1}\Big(\overrightarrow{(\partial_1f)}\Big)^{-1} \alpha^{(f)}_{021} (\alpha^{(f)}_{012})^{-1} & f \in E_2
\end{array}\right\rangle\right\rangle
\end{displaymath}
where $\langle\langle - \rangle\rangle$ denotes the normal closure of the subgroup generated by the indicated elements and $\alpha$'s come from the definition of a $(\scrG,\alpha)$-system for each given $f$.
\end{theorem}
\begin{proof}
The same proof as the proof of \cite[Thm.\ 3.2 (2)]{Stix} shows that $Q$ induces an equivalence of categories between the infinite Galois categories $(\rmlcs(E,\scrG), F_s)$ and the full subcategory of objects of $*^N_{v \in E_0} \scrG(v) *^N \pi_1(E_{\leq 1},T)-\rmSets$ on which $H$ acts trivially. We conclude by Observation \ref{setsofquotient}. 
\end{proof}
\begin{rmk}\label{topvsnoohi}
It is important to note that we can replace free Noohi products by free topological products in the statement above, as we take the Noohi completion of the quotient anyway. More precisely, the canonical map
\begin{displaymath}
(*^{\rmtop}_{v \in E_0} \scrG(v) *^{\rmtop} \pi_1(E_{\leq 1},T)/\bar{H_1})^{\rmNoohi} \rarr (*^N_{v \in E_0} \scrG(v) *^N \pi_1(E_{\leq 1},T)/\bar{H})^{\rmNoohi}
\end{displaymath}
is an isomorphism, where $H_1$ is the normal closure in $*^{\rmtop}_{v \in E_0} \scrG(v) *^{\rmtop} \pi_1(E_{\leq 1},T)$ of a group having the same generators as $H$. This is because the categories of $G-\rmSets$ are the same for those two Noohi groups.
\end{rmk}
\begin{fact}
The topological free product $*^{\rmtop}_i G_i$ of topological groups has as an underlying space the free product of abstract groups $*_i G_i$. This follows from the original construction of Graev \cite{Graev}. 
\end{fact}
\subsection{Application to the pro-\'etale fundamental group}
\subsubsection*{Descent data}
Let $T_\bullet$ be a $2$-complex in a category $\scrC$ and let $\scrF \rightarrow \scrC$ be a category fibred over $\scrC$, with $\scrF(S)$ as a category of sections above the object $S$. 
\begin{defn}
The category $\rmDD (T_\bullet,\scrF)$ of \emph{descent data} for $\scrF / \scrC$ relative $T_\bullet$ has as objects pairs $(X',\phi)$ where $X' \in \scrF(T_0)$ and $\phi$ is an isomorphism $\partial_0^*X' \stackrel{\sim}{\rightarrow}\partial_1^*X'$ in $\scrF(T_1)$ such that the \emph{cocycle condition} holds, i.e., the following commutes in $\scrF(T_2)$:
\begin{center}
\begin{tikzpicture}
\matrix(a)[matrix of math nodes,
row sep=3em, column sep=1.5em,
text height=1.5ex, text depth=0.25ex]
{v_2^*X' &  &v_1^* X'\\  & v_0^* X' &  \\};
\path[->,font=\scriptsize] (a-1-1) edge node[above] {$\partial_0^* \phi$} (a-1-3);
\path[->,font=\scriptsize] (a-1-1) edge node[left] {$\partial_1^* \phi$} (a-2-2);
\path[->,font=\scriptsize] (a-1-3) edge node[right] {$\partial_2^* \phi$} (a-2-2);
\end{tikzpicture}
\end{center}
Morphisms $F: (X',\phi) \rightarrow (Y',\psi)$ in $\rmDD(T_\bullet,\scrF)$ are morphisms $F: X' \rightarrow Y'$ in $\scrF(T_0)$ such that its two pullbacks $\partial_0^*f$ and $\partial_1^* f$ are compatible with $\phi$, $\psi$, i.e., $\partial_1^* f \circ \phi = \psi \circ \partial_0^* f$.
\end{defn}
Let $h: S' \rightarrow S$ be a map of schemes. There is an associated $2$-complex of schemes
\begin{displaymath}
S_\bullet(h): S' \leftleftarrows S' \times_S S' \tripleleftarrow S' \times_S S'\times_S S'
\end{displaymath}
The value of $\partial_i$ is the projection under omission of the $i^{\mathrm{th}}$ component. We abbreviate $\rmDD(S_\bullet(h),\scrF)$ by $\rmDD(h,\scrF)$. Observe that $h^*$ gives a functor $h^*:\scrF(S) \rightarrow \rmDD(h,\scrF)$.
\begin{defn}
In the above context $h : S' \rightarrow S$ is called an \emph{effective descent} morphism for $\scrF$ if $h^*$ is an equivalence of categories.
\end{defn}
\begin{proposition}[{\cite[Prop.\ 1.16]{ElenaPaper}}]\label{properdescent}
Let $g: S' \rightarrow S$ be a proper, surjective morphism of finite presentation, then $g$ is a morphism of effective descent for geometric coverings.
\end{proposition}
\begin{proof}
This was proven by Lavanda and relies on the results of \cite{Rydh}. More precisely, this follows from Prop.\ 5.4 and Thm.\ 5.19 of \cite{Rydh}, then checking that the obtained algebraic space is a scheme (using \'etaleness and separatedness, see \cite[Tag 0417]{StacksProject}) and that it still satisfies the valuative criterion (see Lemma \ref{VCoP-is-local}).
\end{proof}

\subsubsection*{Discretisation of descent data}
We would like to apply the procedure described in \cite[\S4.3]{Stix} but to the pro-\'etale fundamental group. However, in the classical setting of Galois categories, given a category $\calC$ and functors $F,F': \calC \rightarrow \rmSets$ such that $(\calC,F)$ and $(\calC,F')$ are Galois categories (i.e.\ $F,F'$ are fibre functors), there exists an isomorphism (not unique) between $F$ and $F'$. Choosing such an isomorphism is  called ``choosing a path'' between $F$ and $F'$. However, it is not clear whether an analogous statement is true for tame infinite Galois categories as the proof does not carry over to this case (see the proof of \cite[Lemma 0BN5]{StacksProject} or in \cite{SGA1} - these proofs are essentially the same and rely on the pro-representability result of Grothendieck \cite[Prop.\ A.3.3.1]{Grothendieckdescenteii}).
\begin{question}
Let $\calC$ be a category and $F,F': \calC \rightarrow \rmSets$ be two functors such that $(\calC,F)$ and $(\calC,F')$ are tame infinite Galois categories. Is it true that $F$ and $F'$ are isomorphic? 
\end{question}
As we do not know the answer to this question, we have to make an additional assumption when trying to discretise the descent data. Fortunately, it will always be satisfied in the geometric setting, which is our main case of interest.
\begin{defn}\label{definition-compatible-functor}
Let $(\calC,F)$, $(\calC',F')$ be two infinite Galois categories and let $\phi: \calC \rarr \calC'$ be a functor. We say that $\phi$ is \emph{compatible} if there exists an isomorphism of functors $F \simeq F' \circ \phi$. 
\end{defn}

Let $\scrF \rarr \scrC$ be fibred in tame infinite Galois categories. More precisely, we have a notion of connected objects in $\scrC$ and any $T \in \scrC$ is a coproduct of connected components. Over connected objects $\scrF$ takes values in tame infinite Galois categories (i.e.\ over a connected $Y\in \scrC$ there exists a functor $F_Y : \scrF(Y) \rarr \rmSets$ such that $(\scrF(Y),F_Y)$ is a tame infinite Galois category but we do not fix the functor).

\begin{defn}
Let $T_\bullet$ be a $2$-complex in $\scrC$. Let $E=\pi_0(T_\bullet)$ be its $2$-complex of connected components: the $2$-complex in $\rmSets$ built by degree-wise application of the connected component functor. We will say that $T_\bullet$ is a \emph{compatible} $2$-complex if one can fix fibre functors $F_s$ of $\scrF(s)$ for each simplex $s \in E$ such that $(\scrF(s),F_s)$ is tame and for any boundary map $\partial : s \rarr s'$ there exists an isomorphism of fibre functors $F_s\circ T(\partial)^* \stackrel{\sim}{\rarr} F_{s'}$. 
\end{defn}
The $2$-complexes that will appear in the (geometric) applications below will always be compatible. From now on, we will assume all $2$-complexes to be compatible, even if not stated explicitly.
Let $T_\bullet$ be a compatible $2$-complex in $\scrC$. Fix fibre functors $F_s$ and isomorphisms between them as in the definition of a compatible $2$-complex. For any $\partial$, denote the fixed isomorphism by $\vec{\partial}$.
For a $2$-simplex $(vef)$ of the barycentric subdivision with  $\partial' : f \rarr e$ and
$\partial : e \rarr v$ we define
\begin{displaymath}
\alpha_{vef} = \vec{\partial'} \vec{\partial} \Big(\overrightarrow{\partial \partial'}\Big)^{-1}
\end{displaymath}
or, more precisely,
\begin{displaymath}
\alpha_{vef} = T(\partial)(\vec{\partial'}) \vec{\partial} (\overrightarrow{\partial \partial'})^{-1} \in \Aut(F_v)=\pi_1(\scrF(s),F_s).
\end{displaymath}
We define Noohi group data $(\scrG,\alpha)$ on $E$ in the following way: $\scrG(s) = \pi_1(\scrF(s),F_s)$ for any simplex $s \in E$ and to $\partial : s \rarr s'$ is associated
$\scrG(\partial) : \pi_1(\scrF(s),F_s) \stackrel{T(\partial)^*}{\rarr} \pi_1(\scrF(s'),F_{s}\circ T(\partial)^*) \stackrel{\vec{\partial}()\vec{\partial}^{-1}}{\longrightarrow} \pi_1(\scrF(s'),F_{s'})$. We define elements $\alpha$ as described above and we easily check that this gives Noohi group data.
\begin{proposition}
The choice of functors $F_s$ and the choice of $\vec{\partial}$ as above fix a functor
\begin{displaymath}
\rmdiscr : \rmDD(T_\bullet , \scrF) \rarr \rmlcs(E,(\scrG,\alpha))
\end{displaymath}
which is an equivalence of categories.
\end{proposition}
\begin{proof}
Given a descent datum $(X',\phi)$ relative $T_\bullet$ we have to attach a locally constant $(\scrG,\alpha)$-system on $E$ in a functorial way. For $v \in E_0, e \in E_1$ and $f \in E_2$, the definition of suitable $\scrG(v)$ (or $\scrG(e)$ or $\scrG(f)$) sets and maps $m_\partial$ between them can be given by the same formulas as in \cite[Prop.\ 4.4]{Stix} and also the same computations as in \cite[Prop.\ 4.4]{Stix} show that we obtain an element of $ \rmlcs(E,(\scrG,\alpha))$. Again, the reasoning of \cite[Prop.\ 4.4]{Stix} gives a functor in the opposite direction: given $M \in \rmlcs(E,(\scrG,\alpha))$ we define $X' \in \scrF(T_0) = \prod_{v\in E_0} \scrF(v)$ as $X'_{|v}$ corresponding to $M_v$ for all $v \in E_0$. Maps from edges to vertices define a map $\phi:T(\partial_0)^*X' \rarr T(\partial_1)^*X'$ and to check the cocycle condition one reverses the argument of the proof that $\rmdiscr$ gives a locally constant system.
\end{proof}
To apply the last proposition we need to know that the compatibility condition holds in the setting we are interested in.
\begin{lemma}[{\cite[Lm.\ 7.4.1]{BhattScholze}}]\label{geomfunctorscompatibility}
Let $f:X' \rarr X$ be a morphism of two connected locally topologically noetherian schemes and let $\bx',\bx$ be geometric points on $X', X$, correspondingly. Then the functor $f^* : \rmCov_X \rarr \rmCov_{X'}$ is a compatible functor between infinite Galois categories $(\rmCov_X,F_{\bx})$ and $(\rmCov_{X'},F_{\bx'})$, i.e.\ the functors $F_{\bx}$ and $F_{\bx'}\circ f^*$ are isomorphic.
\end{lemma}
\begin{proof}
Looking at the image of $\bar{x}'$ (as a geometric point) on $X$, we reduce to the case when both $\bar{x}'$ and $\bar{x}$ lie on the same scheme $X$. In that case we proceed as in the second part of the proof of \cite[Lm.\ 7.4.1]{BhattScholze}.
\end{proof}

\begin{corollary}[{\cite[Lm.\ 7.4.1]{BhattScholze}}]\label{base-points}
  Let $X$ be a connected topologically noetherian scheme. Let $\bx_1$, $\bx_2$ be two geometric points on $X$. Then there is an isomorphism $\pipet(X,\bx_1) \simeq \pipet(X,\bx_2)$. It is unique (only) up to an inner automorphism.
\end{corollary}

The above results combine to recover the analogue of \cite[Cor.\ 5.3]{Stix} in the pro-\'etale setting.
\begin{corollary}\label{geomvK}
Let $h : S' \rarr S$ be an effective descent morphism for geometric coverings. Assume that $S$ is connected and $S, S', S'\times_S S',S'\times_S S'\times_S S'$ are locally topologically noetherian. Let $S' = \bigsqcup_v S'_v$
be the decomposition into connected components. Let $\bar{s}$ be a geometric point of $S$, let $\bar{s}(t)$ be a geometric point of the simplex $t \in \pi_0(S_\bullet(h))$, and let $T$ be a maximal tree in the graph $\Gamma = \pi_0(S_\bullet(h))_{\leq 1}$ . For every boundary map $\partial : t \rarr t'$ let $\gamma_{t',t} :
\bar{s}(t') \rarr S_\bullet(h)(\partial)\bar{s}(t)$ be a fixed path (i.e.\ an isomorphism of fibre functors as in Lm.\ \ref{geomfunctorscompatibility}). Then canonically with respect to all these choices
\begin{displaymath}
\pipet(S,\bar{s}) \cong \Big(\big(*^N_{v \in E_0} \pipet(S_v',\bs(v)) *^N \pi_1(\Gamma,T)\big)/\overline{H}\Big)^{\rmNoohi}
\end{displaymath}
where $H$ is the normal subgroup generated by the cocycle and edge relations
\begin{eqnarray}
\pipet(\partial_1)(g)\vec{e} =\vec{e} \pipet(\partial_0)(g)\\ 
\overrightarrow{(\partial_2f)}\alpha^{(f)}_{102}(\alpha^{(f)}_{120})^{-1}\overrightarrow{(\partial_0f)}\alpha^{(f)}_{210}(\alpha^{(f)}_{201})^{-1}\Big(\overrightarrow{(\partial_1f)}\Big)^{-1} \alpha^{(f)}_{021} (\alpha^{(f)}_{012})^{-1}=1
\end{eqnarray}
for all parameter values $e \in S_1(h)$, $g \in \pipet(e,\bar{s}(e))$, and $f \in S_2 (h)$. The map $\pipet(\partial_i)$ uses the fixed path $\gamma_{\partial_i(e),e}$ and $\alpha^{(f)}_{ijk}$ is defined using $v \in S_0(h)$ and $e \in S_1(h)$ determined by $v_i(f) = v$, $\{\partial_0(e),\partial_1(e)\} = \{v_i(f),v_j(f)\}$ as
\begin{displaymath}
\alpha^{(f)}_{ijk} = \gamma_{v,e}\gamma_{e,f}\gamma_{v,f}^{-1} \in \pipet(v,\bar{s}(v)).
\end{displaymath}
\end{corollary}
\begin{rmk}
Similarly as in Rmk.\ \ref{topvsnoohi}, we could replace $*^{N}$ by $*^{\rmtop}$ in the above, as we take the Noohi completion of the whole quotient anyway.
\end{rmk}
\begin{rmk}\label{remark-relations-in-vK-for-normalization}
We will often use Cor.\ \ref{geomvK} for $h$ - the normalization map (or similar situations), where the connected components $S_v'$ are normal. In this case $\pipet(S_v',\bs(v)) = \piet(S_v',\bs_v)$. This implies that $\pipet(\partial_1)$ factorizes through the profinite completion of $\pipet(e,\bs(e))$, which can be identified with $\piet(e,\bs(e))$. Moreover, the map $\pipet(e,\bs(e)) \rarr \piet(e,\bs(e))$ has dense image and, in the end, we take the closure $\overline{H}$ of $H$. The upshot of this discussion is that in the definition of generators of $H$ we might consider $g \in \piet(e,\bs(e))$ instead of $g \in \pipet(e,\bs(e))$ and $\piet(\partial_i)$ instead of $\pipet(\partial_i)$, $i \in \{0,1\}$, i.e.
\begin{displaymath}
\pipet(S,\bar{s}) \cong \Big(\big(*^{\rmtop}_{v \in E_0} \piet(S_v',\bs(v)) *^{\rmtop} \pi_1(\Gamma,T)\big)/\overline{H}\Big)^{\rmNoohi}
\end{displaymath}
where $H$ is the normal subgroup generated by
\begin{displaymath}
\piet(\partial_1)(g)\vec{e}\piet(\partial_0)(g)^{-1}\vec{e}^{-1} \textrm{ for all } e \in S_1(h), g \in \piet(e, \bs(e)) \tag{$R_1$}
\end{displaymath}
and
\begin{displaymath}
\overrightarrow{(\partial_2f)}\alpha^{(f)}_{102}(\alpha^{(f)}_{120})^{-1}\overrightarrow{(\partial_0f)}\alpha^{(f)}_{210}(\alpha^{(f)}_{201})^{-1}\Big(\overrightarrow{(\partial_1f)}\Big)^{-1} \alpha^{(f)}_{021} (\alpha^{(f)}_{012})^{-1} \textrm{ for all } f \in S_2(h). \tag{$R_2$}
\end{displaymath}

\end{rmk}

Let us move on to some applications.
\subsubsection*{Ordered descent data}
Let $\scrF$ be a category fibred over $\scrC$ with a fixed splitting cleavage (i.e.\ the associated pseudo-functor is a functor). Assume that $\scrC$ is some subcategory of the category of locally topologically noetherian schemes with the property that finite fibre products in $\scrC$ are the same as the finite fibre products as schemes. Let $h=\bigsqcup_{i \in I}h_i: S'=\bigsqcup_{i} S_{i \in I}' \rarr S$ be a morphism of schemes and let $<$ be a total order on the set of indices $I$. Let $S_\bullet^< (h) \subset S_\bullet(h)$ be the open and closed sub-2-complex of schemes in $\scrC$ of ordered partial products 
\begin{displaymath}
S_0^<(h) = S' \quad , \quad S_1^<(h) = \bigsqcup_{i < j} S_i' \times_S S_j' \quad , \quad S_2^<(h) = \bigsqcup_{i < j < k} S_i' \times_S S_j' \times_S S_k'
\end{displaymath}
\begin{proposition}\label{ordereddescentprop}
Let $h=\bigsqcup_{i \in I}h_i: S'=\bigsqcup_{i} S_{i \in I}' \rarr S$ be a morphism of schemes such that, for every $i, j \in I$, the maps induced by the diagonal morphisms $\Delta_i^*:\scrF(S_i' \times_S S_i') \rarr \scrF(S_i')$ and $(\Delta_i \times \id_{S'_j})^*:\scrF(S_i' \times_S S_j' \times_S S_i') \rarr \scrF(S_i'\times S_j')$ are fully faithful. Then the natural open and closed immersion $S_\bullet^<(h) \hookrightarrow S_\bullet(h)$ induces an equivalence of categories
\begin{displaymath}
\rmDD(h,\scrF) \stackrel{\cong}{\rarr} \rmDD(S_\bullet^<(h),\scrF)
\end{displaymath}
\end{proposition}
\begin{proof}
Let $Y \in \scrF(S'_i)$ and consider $\partial_0^*Y, \partial_1^* Y \in \scrF(S_i' \times_S S_i')$ obtained via maps induced by the projections $\scrF(S_i') \rarr \scrF(S_i' \times_S S_i')$. We first claim that there is \underline{exactly one} isomorphism $\partial_0|_{S_i\times_S S_i}^*Y \rarr \partial_1|_{S_i\times_S S_i}^* Y$ as in the definition of descent data. Observe that $\Delta_i^* \partial_0^*Y = Y$, $\Delta_i^* \partial_1^*Y = Y$ and from the assumption any isomorphism $\phi : \partial_0^*Y|_{S_i} \rarr \partial_1^* Y|_{S_i}$ corresponds to some isomorphism in $\psi \in \Hom_{S_i}(Y|_{S_i},Y|_{S_i})$. Pulling back the cocycle condition via the diagonal $\Delta_{2,i}^*:\scrF(S_i' \times_S S_i' \times_S S_i') \rarr \scrF(S_i')$ we get $\psi = \id_{Y|_{S_i}}$, so there is at most one map $\phi$ as above (we use here and below that we work with a splitting cleavage and so, by definition, the pullback functors preserve compositions of maps). Moreover, our assumptions imply that $\Delta_{2,i}^*$ is fully faithful as well, which shows that $\phi: \partial_0^*Y|_{S_i} \rarr \partial_1^* Y|_{S_i}$ corresponding to $\id_{Y|_{S_i}}$ will satisfy the condition. A similar reasoning shows that if we have $\phi_{ij}$ specified for $i<j$, then $\phi_{ji}$ is uniquely determined and the if $\phi_{ij}$'s satisfy the cocycle condition on $S_{ijk}$ for $i<j<k$, then $\phi_{ij}$'s together with $\phi_{ji}$'s obtained will satisfy the cocycle condition on any $S_{\alpha \beta \gamma}$, $\alpha, \beta, \gamma \in \{i,j,k\}$
\end{proof}

\begin{obs}\label{ordereddescentobs}
If the map of schemes $S'_i \rarr S$ is injective, i.e.\ if the diagonal map $S_i' \rarr S_i' \times_S S_i'$ is an isomorphism, then (still assuming splitting of the cleavage) the assumptions of the proposition are satisfied.
\end{obs} 

\subsubsection*{Two examples}
\begin{example}\label{exnodal}
  Let $k$ be a field and $C$ be $\bbP^1_k$ with two $k$-rational closed points $p_0$ and $p_1$ glued (see \cite{Schwede} for results on gluing schemes). Denote by $p$ the node (i.e.\ the image of $p_i$'s in $C$). We want to compute $\pipet(C)$. By the definition of $C$, we have a map $h:\tC = \bbP^1 \rarr C$ (which is also the normalization). It is finite, so it is an effective descent map for geometric coverings. Thus, we can use the van Kampen theorem. This goes as follows:
  \begin{itemize}
    \item Check that $\tC \times_C \tC \simeq \tC \sqcup p_{01} \sqcup p_{10}$ as schemes over $C$, where $p_{\alpha \beta}$ are equal to $\Spec (k)$ and map to the node of $C$ via the structural map. This can be done by checking that $\Hom_C(Y,\tC \sqcup p_{01} \sqcup p_{10})\simeq \Hom_C(Y,\tC)\times \Hom_C(Y,\tC)$;
    \item Similarly, check that $\tC \times_C \tC \times_C \tC \simeq \tC \sqcup p_{001} \sqcup p_{010} \sqcup p_{011} \sqcup p_{100} \sqcup p_{101} \sqcup p_{110}$, where the  projection $\tC \times_C \tC \times_C \tC \rarr \tC \times_C \tC$ omitting the first factor maps $p_{abc}$ to $p_{bc}$ and so on;
    \item We fix a geometric point $\bar{b}=\Spec(\bk)$ over the base scheme $\Spec(k)$ and fix geometric points $\bp_0$ and $\bp_1$ over $p_0$ and $p_1$ that map to $\bar{b}$. Then we fix geometric points on $\tC,p_{01},p_{10} \subset \tC \sqcup p_{01} \sqcup p_{10} \simeq \tC \times_C \tC$ in a compatible way and similarly for connected components of $\tC \times_C \tC \times_C \tC$ (i.e.\ let us say that $\bp_{\alpha \beta \gamma} \mapsto \bp_{\alpha}$ via $v_0$ and $\bp_{\alpha \beta} \mapsto \bp_{\alpha}$). We fix a path $\gamma$ from $\bp_0$ to $\bp_1$ that becomes trivial on $\Spec (k)$ via the structural map (this can be done by viewing $\bp_0$ and $\bp_1$ as geometric points on $\tC_{\bk}$, choosing the path on $\tC_{\bk}$ first and defining $\gamma$ to be its image). Let $\bp$ be the fixed geometric point on $C$ given by the image of $\bp_0$ (or, equivalently, $\bp_1$).
    \item We want to use Cor.\ \ref{geomvK} to compute $\pipet(C,\bp)$. We choose $\bp_0$ as the base point $\bs(\tC)$ for $\tC \in \pi_0(S_0(h))$, $\tC \in \pi_0(S_1(h))$ and $\tC \in \pi_0(S_2(h))$ . Then for any $t,t' \in \pi_0(S_\bullet(h))$ and the boundary map $\partial : t \rarr t'$, we use either the identity or $\gamma$ to define $\gamma_{t',t} : \bs(t') \rarr S_\bullet(h)(\partial)\bs(t)$ as all the points $\bp_{abc}$ map ultimately either to $\bp_0$ or $\bp_1$. 
  \end{itemize}

  With this setup, the $\alpha^{(f)}_{ijk}$'s (defined as in Cor.\ \ref{geomvK}) are trivial for any $f$ and so the relation (2) in this corollary reads $\overrightarrow{(\partial_2f)}\overrightarrow{(\partial_0f)}\Big(\overrightarrow{(\partial_1f)}\Big)^{-1}=1$. Applying this to different faces $f \in \pi_0(\tC \times_C \tC \times_C\tC)$ gives that the image of $\pi_1(\Gamma,T) \simeq \bbZ^{*3}$ in $\pipet(C,\bp)$ is generated by a single edge (in our case only one maximal tree can be chosen -- containing a single vertex). The choice of paths made guarantees $\pipet(\partial_0)(g)=\pipet(\partial_1)(g)$ in $\pipet(\tC,\bp_0)$ for any $g \in \pipet(p_{ab},\bp_{ab})=\Gal(k)$. So relation (1) in Cor.\ \ref{geomvK} implies that the image of $\pipet(\tC,\bp_0)\simeq \Gal(k)$ in $\pipet(C,\bp_0)$ commutes with the elements of the image of $\pi_1(\Gamma,T)$. Putting this together we get
  \begin{eqnarray*}
  \pipet(C,\bp) \simeq\\
  \simeq \Big((\pipet(\tC,\bp_0)*^{\rmtop}\pi_1(\Gamma,T))/\langle\langle \pipet(\partial_1)(g)\vec{e}=
  \vec{e}\pipet(\partial_2)(g), \overrightarrow{(\partial_2f)}\overrightarrow{(\partial_0f)}\Big(\overrightarrow{(\partial_1f)}\Big)^{-1}
  =1\rangle\rangle\Big)^{\rmNoohi}\simeq\\
  \simeq \Big(\Gal_k \times \bbZ\Big)^{\rmNoohi} = \Gal_k \times \bbZ
  \end{eqnarray*}
  
\end{example}

\begin{example}\label{exofgluingmany}
Let $X_1,\ldots,X_m$ be geometrically connected normal curves over a field $k$ and let $Y_{m+1}, \ldots,Y_n$ be nodal curves over $k$ as in Ex.\ \ref{exnodal}. Let $x_i : \Spec(k) \rarr X_i$ be rational points and let $y_j$ denote the node of $Y_j$. Let $X:=\cup_\bullet X_i\cup_\bullet Y_j$ be a scheme over $k$ obtained via gluing of $X_i$'s and $Y_j$'s along the rational points $x_i$ and $y_j$ (in the sense of \cite{Schwede}). The notation $\cupt$ denotes gluing along the obvious points. The point of gluing gives a rational point $x: \Spec(k) \rarr X$. We choose a geometric point $\bar{b}=\Spec(\bk)$ over the base $\Spec(k)$ and choose a geometric point $\bx$ over $x$ such that it maps to $\bar{b}$. The maps $X_i \rarr X$ and $Y_j \rarr X$ are closed immersions (this is basically \cite[Lm.\ 3.8]{Schwede}).  We also get geometric points $\bx_i$ and $\by_j$ over $x_i$ and $y_j$ that map to $\bar{b}$ as well. Denote $\bX_i = (X_{i})_{\bk}$. Let $\Gal_{k,i} = \piet(x_i,\bx_i)$. It is a copy of $\Gal_k$ in the sense that the induced map $\piet(x_i,\bx_i) \rarr \piet(\Spec(k),\bar{b})$ is an isomorphism. Let us denote by $\iota_i : \Gal_{k} \rarr \Gal_{k,i}$ the inverse of this isomorphism. The group $\piet(x_i,\bx_i)$ acts on $\piet(\bX_i,\bx_i)$ and allows to write $\piet(X_i,\bx_i) \simeq \piet(\bX_i,\bx_i) \rtimes \Gal_{k,i}$.
  
After some computations (as in the previous example), using Cor.\ \ref{geomvK} and Ex.\ \ref{exnodal}, one gets 
  \begingroup\makeatletter\def\f@size{10}\check@mathfonts
  \begin{displaymath}
    \pipet(X,\bx) \simeq \Big(*^N_{1 \leq i \leq m} (\piet(\bX_i,\bx_i) \rtimes  \Gal_{k,i}) *^N_{m+1 \leq j \leq n} (\bbZ \times \Gal_{k,j})/\langle\langle \iota_i(\sigma)=\iota_{i'}(\sigma)|\sigma \in \Gal_k, \text{ } i,i' = 1, \ldots, n \rangle\rangle \Big)^{\rmNoohi}
  \end{displaymath}
  \endgroup
  Let us describe the category of group-sets.
\begingroup\makeatletter\def\f@size{10}\check@mathfonts
\begin{equation*}
\begin{aligned}
\Big(&{}*^N_{1 \leq i \leq m} (\piet(\bX_i,\bx_i) \rtimes  \Gal_{k,i}) *^N_{m+1 \leq j \leq n} (\bbZ \times \Gal_{k,j})/\langle\langle \iota_i(\sigma)=\iota_{i'}(\sigma)|\sigma \in \Gal_k, \text{ } i,i' = 1, \ldots, n \rangle\rangle \Big)^{\rmNoohi}-\rmSets \simeq\\
&{}\simeq \Big\{ S \in \Big(*^N_{1 \leq i \leq m} (\piet(\bX_i,\bx_i) \rtimes  \Gal_{k,i}) *^N_{m+1 \leq j \leq n} (\bbZ \times \Gal_{k,j}) \Big)^{\rmNoohi} - \rmSets \Big| \forall_{i,i',\sigma}\forall_{s \in S} \iota_i(\sigma)\cdot s = \iota_{i'}(\sigma) \cdot s \Big\} \simeq\\
&{}\simeq \Big\{ S \in \Big(*^N_{1 \leq i \leq m} \piet(\bX_i,\bx_i) *^N\bbZ^{*n-m}*^N \Gal_k \Big)^{\rmNoohi} - \rmSets \Big| (\spadesuit) \Big\} \\
&\text{where the condition $(\spadesuit)$ reads } \forall_{\substack{\sigma \in \Gal_k\\ s \in S, 1 \leqslant i \leqslant m}} \forall_{\substack{\gamma \in \piet(\bX_i,\bx_i)\\ w \in \bbZ^{*n-m}}}\big(\sigma \cdot (\gamma \cdot s) = ^{\sigma}\gamma \cdot( \sigma \cdot s) \text{ and } \sigma \cdot (w \cdot s) = w \cdot (\sigma \cdot s)\big).
\end{aligned}     
\end{equation*}
\endgroup

We have used Obs.\ \ref{setsofquotient} and Lm.\ \ref{setsofsemi} below.
\end{example}

\begin{lemma}\label{setsofsemi}
Let $K$ and $Q$ be topological groups and assume we have a continuous action $K \times Q \rarr K$ respecting multiplication in $K$. Then $K \rtimes Q$ with the product topology (on $K \times Q$) is a topological group and there is an isomorphism
\begin{displaymath}
K*^{\rmtop}Q / \langle \langle qkq^{-1}= {}^{q}k \rangle \rangle \rarr K \rtimes Q
\end{displaymath}
\end{lemma}
\begin{proof}
That $K \rtimes Q$ becomes a topological group is easy from the continuity assumption of the action. The isomorphism is obtained as follows: from the universal property we have a continuous homomorphism $K*^{\rmtop}Q \rarr K \rtimes Q$ and the kernel of this map is the smallest normal subgroup containing the elements $qkq^{-1}( ^{q}k)^{-1}$ (this follows from the fact that the underlying abstract group of $K*^{\rmtop}Q$ is the abstract free product of the underlying abstract groups, similarly for $K \rtimes Q$ and that we know the kernel in this case). So we have a continuous map that is an isomorphism of abstract groups. We have to check that the inverse map $K \rtimes Q \ni kq \mapsto kq \in K*^{\rmtop}Q / \langle \langle qkq^{-1}= ^{q}k \rangle \rangle$ is continuous. It is enough to check that the map $K \times Q \ni (k,q) \mapsto kq \in K*^{\rmtop}Q$ (of topological spaces) is continuous, but this follows from the fact that the maps $K \rarr K*^{\rmtop}Q$ and $Q \rarr K*^{\rmtop}Q$ are continuous and that the multiplication map $(K*^{\rmtop}Q) \times (K*^{\rmtop}Q) \rarr K*^{\rmtop}Q$ is continuous.
\end{proof}

Let us also state a technical lemma concerning the ``functoriality'' of the van Kampen theorem. It is important that the diagram formed by the schemes $X_1,X_2, \tX, \tX_1$ in the statement is cartesian.
\begin{lemma}\label{vKfunctoriality}
Let $f:X_1 \rarr X_2$ be a morphism of connected schemes and $h:\tX \rarr X_2$ be a morphism of schemes. Denote by $h_1:\tX_1 \rarr X_1$ the base-change of $h$ via $f$. Assume that $h$ and $h_1$ are effective descent morphisms for geometric coverings and that local topological noetherianity assumptions are satisfied for the schemes involved as in the statement of Cor.\ \ref{geomvK}. Assume that for any connected component $W \in \pi_0(S_\bullet(h))$, the base-change $W_1$ of $W$ via $f$ is connected. Choose the geometric points on $W_1 \in \pi_0(S_\bullet(h_1))$ and paths between the obtained fibre functors as in Cor.\ \ref{geomvK} and choose the geometric points and paths on $W \in \pi_0(S_\bullet(h))$ as the images of those chosen for $\tX_1$.
Identify the graphs $\Gamma=\pi_0(S_\bullet(h))_{\leqslant 1}$ and $\Gamma_1=\pi_0(S_\bullet(h_1))_{\leqslant 1}$ (it is possible thanks to the assumption made) and choose a maximal tree $T$ in $\Gamma$. Using the above choices, use Cor.\ \ref{geomvK}. to write the fundamental groups $\pipet(X_1) \simeq \big((*_{W \in \pi_0(\tX)}^{\rmtop} \pipet(W_1))*^{\rmtop} \pi_1(\Gamma_1,T)/\langle R' \rangle\big)^{\rmNoohi}$ and  $\pipet(X_2) \simeq \big((*_{W \in \pi_0(\tX)}^{\rmtop} \pipet(W))*^{\rmtop} \pi_1(\Gamma,T)/\langle R \rangle \big)^{\rmNoohi}$. 

Then the map of fundamental groups $\pipet(f) : \pipet(X_1) \rarr \pipet(X_2)$ is the Noohi completion of the map
\begin{displaymath}
\Big((*_{W \in \pi_0(\tX)}^{\rmtop} \pipet(W_1))*^{\rmtop} \pi_1(\Gamma_1,T)\Big)/\langle R' \rangle \rarr  \Big(*_{W \in \pi_0(\tX)}^{\rmtop} \pipet(W))*^{\rmtop} \pi_1(\Gamma,T)\Big)/\langle R \rangle,
\end{displaymath}  which is induced by the maps $\pipet(W_1) \rarr \pipet(W)$ and the identity on $\pi_1(\Gamma_1,T)$ (which makes sense after identification of $\Gamma_1$ with $\Gamma$).
\end{lemma}
\begin{proof}
It is clear that on (the image of) $\pipet(W_1)$ (in $\pipet(X_1)$) the map is the one induced from $f_W:W_1 \rarr W$. The part about $\pi_1(\Gamma_1,T)$ follows from the fact that $\pi_1(\Gamma_1,T) < \pipet(X_1)$ acts in the same way as $\pi_1(\Gamma,T) < \pipet(X_2)$ on any geometric covering of $X_2$. This follows from the choice of points and paths on $W \in \pi_0(S_\bullet(h))$ as the images of the points and paths on the corresponding connected components $W_1 \in \pi_0(S_\bullet(h_1))$. The maps as in the statement give a morphism $\phi:(*_{W \in \pi_0(\tX)}^{\rmtop} \pipet(W_1))*^{\rmtop} \pi_1(\Gamma_1,T) \rarr  (*_{W \in \pi_0(\tX)}^{\rmtop} \pipet(W))*^{\rmtop} \pi_1(\Gamma,T)$ and it is easy to check that $\phi(R') \subset R$, which finishes the proof.
\end{proof}

\subsection{K\"unneth formula}
In this subsection we use the van Kampen formula to prove the K\"unneth formula for $\pipet$.

Let $X, Y$ be two connected schemes locally of finite type over an algebraically closed field $k$ and assume that $Y$ is proper. Let $\bx, \by$ be geometric points of $X$ and $Y$ respectively with values in the same algebraically closed field extension $K$ of $k$. With these assumptions, the classical statement says that the ``K\"unneth formula'' for $\piet$ holds, i.e.
\begin{fact}(\cite[Exp. X, Cor.\ 1.7]{SGA1})\label{usual-van-Kampen}
With the above assumptions, the map induced by the projections is an isomorphism
\begin{displaymath}
\piet(X \times_k Y,(\bx,\by)) \stackrel{\sim}{\rarr} \piet(X,\bx) \times \piet(Y,\by)
\end{displaymath}
\end{fact}
We want to establish analogous statement for $\pipet$.

\begin{proposition}\label{Kunneth-proetale}
Let $X, Y$ be two connected schemes locally of finite type over an algebraically closed field $k$ and assume that $Y$ is proper. Let $\bx, \by$ be geometric points of $X$ and $Y$ respectively with values in the same algebraically closed field extension $K$ of $k$. Then the map induced by the projections is an isomorphism
\begin{displaymath}
\pipet(X \times_k Y,(\bx,\by)) \stackrel{\sim}{\rarr} \pipet(X,\bx) \times \pipet(Y,\by)
\end{displaymath}
\end{proposition}
Choosing a path between $(\bx,\by)$ and some fixed $k$-point of $X \times_k Y$ (seen as a geometric point) and looking at the images of this path via projections onto $X$ and $Y$ reduces us (by Cor.\ \ref{base-points} and compatibility of the chosen paths), to the situation where we can assume that $\bx$ and $\by$ are $k$-points. We are going to assume this in the proof.   
Before we start, let us state and prove the surjectivity of the above map as a lemma. Properness is not needed for this.

\begin{lemma}\label{Kunneth-sujectivity}
Let $X, Y$ be two connected schemes over an algebraically closed field $k$ with $k$-points on them: $\bx$ on $X$ and $\by$ on $Y$. Then the map induced by the projections
\begin{displaymath}
\pipet(X \times_k Y,(\bx,\by)) \rarr \pipet(X,\bx) \times \pipet(Y,\by)
\end{displaymath}
is surjective.
\end{lemma}
\begin{proof}
Consider the map $(\id_X, \by): X = X\times_k \by  \rarr X \times_k Y$. It is easy to check that the map induced on fundamental groups $\pipet(X,\bx) \rarr \pipet(X \times_k Y,(\bx,\by)) \rarr \pipet(X,\bx) \times \pipet(Y,\by)$ is given by $(\id_{\pipet(X,\bx)}, 1_{\pipet(Y,\by)}) : \pipet(X,\bx) \rarr \pipet(X,\bx) \times \pipet(Y,\by)$. Analogous fact holds if we consider  $(\bx, \id_Y): Y  \rarr X \times_k Y$. As a result, the image $\rmim(\pipet(X \times_k Y,(\bx,\by)) \rarr \pipet(X,\bx) \times \pipet(Y,\by))$ contains the set $(\pipet(X,\bx) \times \{1_{\pipet(Y,\by)}\}) \cup  (\{1 _{\pipet(X,\bx)}\} \times \pipet(Y,\by))$. This finishes the proof, as this set generates $\pipet(X,\bx) \times \pipet(Y,\by)$.
\end{proof}

\begin{proof}(of Prop.\ \ref{Kunneth-proetale})
As $X,Y$ are locally of finite type over a field, the normalization maps are finite and we can apply Prop.\ \ref{properdescent}.
Let $\tX \rarr X$ be the normalization of $X$ and let $\tX = \sqcup_v \tX_v$ be its decomposition into connected components and let us fix a closed point $x_v \in \tX_v$ for each $v$. Similarly, let $\sqcup_u \widetilde{Y}_u = \widetilde{Y} \rarr Y$ be the decomposition into connected components of the normalization of $Y$ with closed points $y_u \in \widetilde{Y}_u$.

We first deal with a particular case.

\textbf{Claim}: the statement of Prop.\ \ref{Kunneth-proetale} holds under the additional assumption that
\begin{itemize}
  \item either, for any $v$, the projections induce isomorphisms
\begin{displaymath}
\pipet(\tX_v \times_k Y,(x_v,\by)) \stackrel{\sim}{\rarr} \pipet(\tX_v,x_v) \times \pipet(Y,\by).
\end{displaymath}
  \item or, for any $u$, the projections induce isomorphisms
\begin{displaymath}
\pipet(X \times_k \widetilde{Y}_u,(\bx,y_u)) \stackrel{\sim}{\rarr} \pipet(X,\bx) \times \pipet(\widetilde{Y}_u,y_u).
\end{displaymath}

\end{itemize}

\textbf{Proof of the claim.} Apply Cor.\ \ref{geomvK} to $h: \tX \rarr X$. We choose $\bx$ and $x_v$'s as geometric points $\bs(t)$ of the corresponding simplexes $t \in \pi_0(S_\bullet(h))_0$ and choose $\bs(t)$ to be arbitrary closed points (of suitable double and triple fibre products) for $t \in \pi_0(S_\bullet(h))_{2}$. We fix a maximal tree $T$ in $\Gamma=\pi_0(S_\bullet(h))_{\leq 1}$ and fix paths $\gamma_{t',t}:\bs(t') \rarr S_\bullet(h)(\partial)\bar{s}(t)$. Thus, we get  $\pipet(X,\bx) \cong \Big(\big(*^N_{v} \pipet(\tX_v,x_v) *^N \pi_1(\Gamma,T)\big)/\overline{H}\Big)^{\rmNoohi}$ where $H$ is defined as in Cor.\ \ref{geomvK}.

Observe now that $\tX_v \times_k Y$ are connected (as $k$ is algebraically closed) and that $h\times\id_Y:\tX \times Y \rarr X \times Y$ is an effective descent morphism for geometric coverings. So we might use Cor.\ \ref{geomvK} in this setting. As $(\tX_v \times Y)\times_{X \times Y} (\tX_w \times Y) = (\tX_v \times_X X_w) \times_k Y$, and similarly for triple products, we can identify in a natural way $i^{-1}:\pi_0(S_\bullet(h\times\id_Y)) \stackrel{\sim}{\rarr} \pi_0(S_\bullet(h))$. In particular we can identify the graph $\Gamma_Y =\pi_0(S_\bullet(h\times\id_Y))_{\leq 1}$ with $\Gamma$ and we choose the maximal tree $T_Y$ of $\Gamma_Y$ as the image of $T$ via this identification. For $t \in \pi_0(S_\bullet(h))$ choose $(\bs(t),\by)$ as the closed base points for $i(t) \in \pi_0(S_\bullet(h\times\id_Y))$.
Denote by $\alpha_{ijk}$ elements of various $\pipet(\tX_v)$ defined as in Cor.\ \ref{geomvK} and by $\vec{e}$ elements of $\pi_1(\Gamma,T)$. By the choices and identifications above we can identify $\pi_1(\Gamma_Y,T_Y)$ with $\pi_1(\Gamma,T)$. Using van Kampen and the assumption, we write
\begin{eqnarray*}
\pipet(X \times Y,(\bx,\by)) \cong \Big(\big(*^N_{v} \pipet(\tX_v \times Y,(x_v,\by)) *^N \pi_1(\Gamma_Y,T_Y)\big)/\overline{H_Y}\Big)^{\rmNoohi}\\ 
\cong \Big(\big(*^N_{v} (\pipet(\tX_v,x_v) \times \pipet(Y,\by)_v) *^N \pi_1(\Gamma,T)\big)/\overline{H_Y}\Big)^{\rmNoohi}.
\end{eqnarray*}
Here $\pipet(Y,\by)_v$ denotes a ``copy'' of $\pipet(Y,\by)$ for each $v$. By Lm.\ \ref{Kunneth-sujectivity}, for $T \in \pi_0(S_\bullet(h))$ the natural map $\pipet(T \times Y,(\bs(T),\by)) \rarr \pipet(T,\bs(T))\times \pipet(Y,\by)$ is surjective. It follows that the relations defining $H_Y$ (as in Cor.\ \ref{geomvK}) can be written as 
\begin{eqnarray*}
\pipet(\partial_1)(g)h_{y,1} \vec{e} = \vec{e} \pipet(\partial_0)(g)h_{y,0} \text{, } e\in e(\Gamma) \text{, } g \in \pipet(e,\bs(e)) \text{, } e \in S_1(h) \text{, } h_y \in \pipet(Y,\by)\\
\overrightarrow{(\partial_2f)}\alpha_{102}\alpha_{120}^{-1}\overrightarrow{(\partial_0f)}\alpha_{210}\alpha_{201}^{-1}\Big(\overrightarrow{(\partial_1f)}\Big)^{-1} \alpha_{021} \alpha_{012}^{-1}=1, f \in S_2(h)
\end{eqnarray*}
where $\alpha$'s in the second relation are elements of suitable $\pipet(\tX_v)$'s and are the same as in the corresponding generators of $H$. The $h_{y,i}$ denotes a copy of element $h_y \in \pipet(Y,\by)$ in a suitable $\pipet(Y,\by)_v$. Varying $e$ and $h_y$ while choosing $g=1 \in \pipet(e,\bs(e))$ for every $e$, gives that $h_{y,1} \vec{e} = \vec{e}h_{y,0}$. For $e \in T$ we have $\vec{e}=1$ and so the first relation reads $h_{y,1}=h_{y,0}$, i.e.\ it identifies $\pipet(Y,\by)_v$ with $\pipet(Y,\by)_w$ for $v,w$ - ends of the edge $e$. As $T$ is a maximal tree in $\Gamma$, it contains all the vertices, so the first relation identifies $\pipet(Y,\by)_v=\pipet(Y,\by)_w$ for any two vertices $v,w$ and we will denote this subgroup (of the quotient) by $\pipet(Y,\by)$. This way $h_{y,1} \vec{e} = \vec{e}h_{y,0}$ reads simply $h_{y} \vec{e} = \vec{e}h_{y}$, so elements of $\pipet(Y,\by)$ commute with elements of $\pi_1(\Gamma,T)$. Moreover, elements of $\pipet(Y,\by)$ commute with elements of each $\pipet(\tX_v,x_v)$, as this was true for $\pipet(Y,\by)_v$. On the other hand, choosing $h_y=1$ in the first relation and looking at the second relation, we see that $H_Y$ contains all the relations of $H$. Using notations from the above discussion, we can sum it up by writing
\begin{displaymath}
H_Y = \langle \langle \text{relations generating $H$}, h_{y,0}=h_{y,1}, h_y\vec{e}=\vec{e}h_y, h_yg=gh_y \text{ (} g \in \pipet(\tX_v,x_v) \text{)}  \rangle \rangle.
\end{displaymath}
Putting this together, we get equivalences of categories
\begin{eqnarray*}
\Big(\big(*^N_{v} (\pipet(\tX_v,x_v) \times \pipet(Y,\by)_v) *^N \pi_1(\Gamma,T)\big)/\overline{H_Y}\Big)-\rmSets\\
\cong \{S \in \big(*^N_{v} (\pipet(\tX_v,x_v) \times \pipet(Y,\by)_v) *^N \pi_1(\Gamma,T)\big)-\rmSets | H_Y \text{ acts trivially on } S\} \\
\stackrel{\spadesuit}{\cong} \{S \in \big((*^N_{v} \pipet(\tX_v,x_v) *^N \pi_1(\Gamma,T))\times \pipet(Y,\by) \big)-\rmSets | H \text{ acts trivially on } S\}\\
\cong \Big(\big(\big(*^N_{v} \pipet(\tX_v,x_v) *^N \pi_1(\Gamma,T)\big)/\overline{H}\big)\times \pipet(Y,\by)\Big)-\rmSets\\
\cong \Big(\pipet(X,\bx)\times \pipet(Y,\by)\Big)-\rmSets,
\end{eqnarray*}
where equality $\spadesuit$ follows from the fact that for topological groups $G_1,G_2$ there is equivalence $(G_1 \times G_2) - \rmSets \cong \{S \in G_1 *^N G_2 - \rmSets | \forall_{\substack{g_1\in G_1\\ g_2 \in G_2}}\forall_{s \in S} g_1g_2s=g_2g_1s\}$ (see Lm.\ \ref{setsofsemi}).\\
This finishes the proof of the Claim in the ``either'' case. After noting that each $\tilde{Y}_u$ is still proper, the ``or'' case follows in a completely symmetrical manner. We have proven a particular case of the proposition. Let us now go ahead and prove the full statement.

\textbf{General case.} The general case follows from the claim proven above in the following way: let $\sqcup_v \tX_v = \tX \rarr X$ and $\sqcup_u \widetilde{Y}_u = \widetilde{Y} \rarr Y$ be decompositions into connected components of the normalizations of $X$ and $Y$. Fix $v$ and note that $\pipet(\tX_v \times_k Y) = \pipet(\tX_v) \times \pipet(Y)$ by applying the claim to $Y$ and $\tX_v$. This is possible, as $\widetilde{Y}_u$'s, $\tX_v$ and the products $\widetilde{Y}_u\times_k\tX_v$ (for all $u$) are normal varieties and so their pro-\'etale fundamental groups are equal to the usual \'etale fundamental groups (by Lm.\ \ref{proetale-of-normal}) for which the equality $\piet(\widetilde{Y}_u\times_k\tX_v)=\piet(\widetilde{Y}_u)\times \piet(\tX_v)$ is known (see Fact \ref{usual-van-Kampen}). Thus, for any $v$, we have that $\pipet(\tX_v \times_k Y) = \pipet(\tX_v) \times \pipet(Y)$. We can now apply the claim to $X$ and $Y$ and finish the proof in the general case.
\end{proof}

\subsection{Invariance of $\pipet$ of a proper scheme under a base-change $K \supset k$ of algebraically closed fields}
\begin{proposition}\label{algclosed-to-algclosed}
Let $X$ be a proper scheme over an algebraically closed field $k$. Let $K \supset k$ be another algebraically closed field. Then the pullback induces an equivalence of categories 
\begin{displaymath}
F:\rmCov_X \rarr \rmCov_{X_K}
\end{displaymath}
In particular, if $X$ is connected,
$X_K \rarr X$ induces an isomorphism
\begin{displaymath}
\pipet(X_K) \stackrel{\sim}{\rarr} \pipet(X).
\end{displaymath}
\end{proposition}
\begin{proof}
Let $X^\nu \rarr X$ be the normalization. It is finite, and thus a morphism of effective descent for geometric coverings. Let us show that the functor $F$ is essentially surjective. Let $Y' \in \rmCov_{X_K}$. As $k$ is algebraically closed and $X^\nu$ is normal, we conclude that $X^\nu$ is geometrically normal, and thus the base change $(X^\nu)_K$ is normal as well (see \cite[Tag 038O]{StacksProject}). Pulling $Y'$ back to $(X^\nu)_K$ we get a disjoint union of schemes finite \'etale over $(X^\nu)_K$ with a descent datum. It is a classical result (\cite[Exp. X, Cor.\ 1.8]{SGA1}) that the pullback induces an equivalence $\mathrm{F}\et_{X^\nu} \rarr \mathrm{F}\et_{X^\nu_K}$ of finite \'etale coverings and similarly for the double and triple products $X^\nu_2 = X^\nu\times_XX^\nu$, $X^\nu_3 = X^\nu\times_XX^\nu\times_X X^\nu$. These equivalences obviously extend to categories whose objects are (possibly infinite) disjoint unions of finite \'etale schemes (over $X^\nu$, $X^\nu_2$, $X^\nu_3$ respectively) with \'etale morphisms as arrows. These categories can be seen as subcategories of $\rmCov_{X^\nu}$ and so on. These subcategories are moreover stable under pullbacks between $\rmCov_{X^\nu_i}$. Putting this together we see, that $Y'' = Y' \times_{X_K}(X^\nu)_K$ with its descent datum is isomorphic to a pullback of a descent datum from $X^\nu$. Thus, we conclude that there exists $Y \in \rmCov_X$ such that $Y' \simeq Y_K$. Full faithfulness of $F$ is shown in the same way. If $X$ is connected, it can be also proven more directly, as $F$ being fully faithful is equivalent to preserving connectedness of geometric coverings, but any connected $Y \in \rmCov_X$ is geometrically connected, and thus $Y_K$ remains connected by Lm.\ \ref{dictionary} \ref{denseimageequivalentconditions}. Note that in the above argument we do not claim that the double and triple intersections $X_2^\nu, X_3^\nu$ are normal, as this is in general false. Instead, we are only using that all the considered geometric coverings of those schemes came as pullbacks from $X^\nu$, and thus were already split-to-finite.
\end{proof}

\section{Fundamental exact sequence}

\subsection{Statement of the results and examples}

The main result of this chapter is the following theorem.
\begin{theorem*}(see Theorem \ref{exactness-in-geometric-to-arithmetic-as-abstract} below)
  Let $k$ be a field and fix an algebraic closure $\bk$. Let $X$ be a geometrically connected scheme of finite type over $k$. Then the sequence of abstract groups
  \begin{displaymath}
  1 \rarr \pipet(X_{\bk}) \rarr \pipet(X) \rarr \Gal_k \rarr 1
  \end{displaymath}
  is exact. 
  
  Moreover, the map $\pipet(X_{\bk}) \rarr \pipet(X)$ is a topological embedding and the map $\pipet(X) \rarr \Gal_k$ is a quotient map of topological groups.
\end{theorem*}
One shows the near exactness first and obtains the above version as a corollary with an extra  argument. The most difficult part of the sequence is exactness on the left. We will prove it as a separate theorem and its proof occupies an entire subsection.

\begin{theorem*}(see Theorem \ref{injectivity-on-the-left} below)
Let $k$ be a field and fix an algebraic closure $\bk$ of $k$. Let $X$ be a scheme of finite type over $k$
such that the base change $X_{\bk}$ is connected. Then the induced map
\begin{displaymath}
\pipet(X_{\bk}) \rarr \pipet(X)
\end{displaymath}
is a topological embedding.
\end{theorem*}
By Prop.\ \ref{dictionary}, it translates to the following statement in terms of coverings: every geometric covering of $X_{\bk}$ can be dominated by a covering that embeds into a base-change to $\bk$ of a geometric covering of $X$ (i.e.\ defined over $k$). In practice, we prove that every connected geometric covering of $X_{\bk}$ can be dominated by a (base-change of a) covering of $X_l$ for $l/k$ finite.

For finite coverings, the analogous statement is very easy to prove simply by finiteness condition. But for general geometric coverings this is non-trivial and maybe even slightly surprising as we show by counterexamples (Ex.\ \ref{counterexample-with-picture} and Ex.\ \ref{counterexample-with-matrices}) that it is not always true that a connected geometric covering of $X_{\bk}$ is isomorphic to a base-change of a covering of $X_l$ for some finite extension $l/k$. This last statement is, however, stronger than what we need to prove, and thus does not contradict our theorem. Observe, that the stronger statement is true for finite coverings and, even more generally, whenever $\pipet(X_{\bk})$ is prodiscrete, as proven in Prop.\ \ref{injectivity-for-prodiscrete}.

Let us proceed to proving the easier part of the sequence first.

\begin{obs}
By Prop.\ \ref{topoinv}, the category of geometric coverings is invariant under universal homeomorphisms. In particular, for a connected $X$ over a field and $k'/k$ purely inseparable, there is $\pipet(X_{k'})=\pipet(X)$.
Similarly, we can replace $X$ by $X_{\mathrm{red}}$ and so assume $X$ to be reduced when convenient. In this case, base change to separable closure $X_{k^s}$ is reduced as well. We will often use this observation without an explicit reference.
\end{obs}

We start with the following lemmas.
\begin{lemma}\label{galois-invariant-connected}
  Let $k$ be a field. Let $k\subset k'$ be a (possibly infinite) Galois extension. Let $X$ be a connected scheme over $k$. Let $\overline{T}_0 \subset \pi_0(X_{k'})$ be a non-empty closed subset preserved by the $\Gal(k'/k)$-action. Then $\overline{T}_0 = \pi_0(X_{k'})$.
  \end{lemma}
  \begin{proof}
  Let $\overline{T}$ be the preimage of $\overline{T}_0$ in $X_{k'}$ (with the reduced induced structure). By \cite[Lemma 038B]{StacksProject}, $\overline{T}$ is the preimage of a closed subset $T \subset X$ via the projection morphism $p:X_{k'}\rarr X$. On the other hand, by \cite[Lemma 04PZ]{StacksProject}, the image $p(\overline{T})$ equals the entire $X$. Thus, $T = X$ and $\overline{T}=X_{k'}$, and so $\overline{T}_0 = \pi_0(X_{k'})$.
\end{proof}

\begin{lemma}\label{finitegeomconncomp}
    Let $X$ be a connected scheme over a field $k$ with an $l'$-rational point with $l'/k$ a finite field extension. Then $\pi_0(X_{k^{\rmsep}})$ is finite, the $\Gal_k$ action on $\pi_0(X_{k^{\rmsep}})$ is continuous and there exists a finite separable extension $l/k$ such that the induced map $\pi_0(X_{k^\rmsep}) \rarr \pi_0(X_l)$ is a bijection. Moreover, there exists the smallest field (contained in $k^{\rmsep}$) with this property and it is Galois over $k$.
\end{lemma}
\begin{proof}
    Let us first show the continuity of the $\Gal_k$-action. The morphism $\Spec(l') \rarr X$ gives a $\Gal_k$-equivariant morphism $\Spec(l'\otimes_k k^{\rmsep}) \rarr X_{k^{\rmsep}}$ and a $\Gal_k$-equivariant map $\pi_0(\Spec(l'\otimes_k k^{\rmsep})) \rarr \pi_0(X_{k^{\rmsep}})$. Denote by $M \subset \pi_0(X_{k^{\rmsep}})$ the image of the last map. It is finite and $\Gal_k$-invariant, and by Lm.\ \ref{galois-invariant-connected}, $M=\pi_0(X_{k'})$. We have tacitly used that $M$ is closed, as $\pi_0(X_{k'})$ is Hausdorff (as the connected components are closed). As $\Gal_k$ acts continuously on $\pi_0(\Spec(l'\otimes_k k^{\rmsep}))$ (for example by \cite[Lemma 038E]{StacksProject}), we conclude that it acts continuously on $\pi_0(X_{k^{\rmsep}})$ as well. From Lm.\ \ref{galois-invariant-connected} again and from \cite[Tag 038D]{StacksProject}, we easily see that the fields $l \subset k^{\rmsep}$ such that $\pi_0(X_{k^\rmsep}) \rarr \pi_0(X_l)$ is a bijection are precisely those that $\Gal_l$ acts trivially on $\pi_0(X_{k^\rmsep})$. To get the minimal field with this property we choose $l$ such that $\Gal_l = \ker(\Gal_k \rarr \Aut(\pi_0(X_{k^\rmsep}) ))$.
\end{proof}

\begin{theorem}\label{arithtogeom}
Let $k$ be a field and fix an algebraic closure $\bk$. Let $X$ be a geometrically connected scheme of finite type over $k$. Let $\bx:\Spec(\bk) \rarr X_{\bk}$ be a geometric point on $X_{\bk}$. Then the induced sequence
\begin{displaymath}
\pipet(X_{\bk},\bx) \stackrel{\iota}{\rightarrow} \pipet(X,\bx) \stackrel{p}{\rightarrow} \Gal_k \rightarrow 1
\end{displaymath}
of topological groups is nearly exact in the middle (i.e.\ the thick closure of $\rmim(\iota)$ equals $\ker(p)$) and $\pipet(X) \rightarrow \Gal_k$ is a topological quotient map.
\end{theorem}

\begin{proof}
\begin{enumerate}
\item The map $p$ is surjective and open: let $U<\pipet(X)$ be an open subgroup. There is a geometric covering $Y$ of $X$ with a $\bk$-point $\by$ such that the morphism $\pipet(Y,\by) \rightarrow \pipet(X,\bx)$ is equal to $U \subset \pipet(X,\bx)$. As $Y$ is locally of finite type over $k$, the image of $\by$ in $Y$ has a finite extension $l$ of $k$ as the residue field. Thus, we get $\Gal_l \rarr \pipet (Y) \rarr \Gal_k$ and we see that the image $\pipet(Y) \rightarrow \Gal_k$ contains an open subgroup, so is open. We have shown that $p$ is an open morphism. In particular the image of $\pipet(X)$ in $\Gal_k$ is open and so also closed. On the other hand, this image is dense as we have the following diagram

\begin{center}
\begin{tikzpicture}
\matrix(a)[matrix of math nodes,
row sep=1.8em, column sep=2em,
text height=1ex, text depth=0.25ex]
{&[-1.7em] \pipet(X) & \pipet(\Spec(k)) \\ {\reallywidehat{{\pi_1^{\mathrm{pro\acute{e}t}}(X)}}}^{\mathrm{prof}} &[-1.7em] \piet(X) & \piet(\Spec(k))\\};

\path[->,font=\scriptsize] (a-1-2) edge (a-2-2);
\draw[double distance = 1.5pt] (a-1-3) -- (a-2-3);
\path[->] (a-1-2) edge (a-1-3);
\path[->>] (a-2-2) edge (a-2-3);
\draw[double distance = 1.5pt] (a-2-1) -- (a-2-2);

  \end{tikzpicture}
\end{center}

where $\widehat{\ldots}^{\mathrm{prof}}$ means the profinite completion. In the diagram, the left vertical map has dense image and the lower horizontal is surjective. This shows that $\pipet(X) \rightarrow \Gal_k$ is surjective.

\item The composition $\pipet(X_{\bk},\bx) \rightarrow \pipet(X,\bx) \rightarrow \Gal_k $ is trivial: this is clear thanks to Prop.\ \ref{dictionary} and the fact that the map $X_{\bk} \rarr \Spec(k)$ factorizes through $\Spec(\bk)$. 

\item The thick closure of $\rmim(\iota)$ is normal: as remarked above, $\pipet(X_{\bk})=\pipet(X_{k^s})$, where $k^s$ denotes the separable closure. Thus, we are allowed to replace $\bk$ with $k^s$ in the proof of this point. Moreover, by the same remark, we can and do assume $X$ to be reduced. Let $Y \rightarrow X$ be a connected geometric covering such that there exists a section $s:X_{k^s} \rightarrow Y\times_XX_{k^s}=Y_{k^s}$ over $X_{k^s}$. Observe that any such section is a clopen immersion: this follows immediately from the equivalence of categories of $\pipet(X_{k^s})-\rmSets$ and geometric coverings.
Define $\bar{T}:= \bigcup_{\sigma \in Gal(k)} {}^\sigma s(X_{k^s}) \subset Y_{k^s}$. Observe that two images of sections in the sum either coincide or are disjoint as $X_{k^s}$ is connected and they are clopen. Now, $\bar{T}$ is obviously open, but we claim that it is also a closed subset. This follows from Lm.\ \ref{finitegeomconncomp} (which implies that $\pi_0(Y_{k^s})$ is finite), but one can also argue directly by using that $Y_{k^s}$ is locally noetherian and ${}^\sigma s(X_{k^s})$ are clopen. Now by \cite[Tag 038B]{StacksProject}, $\bar{T}$ descends to a closed subset $T \subset Y$. It is also open as $T$ is the image of $\bar{T}$ via projection $Y_{k^s} \rightarrow Y$ which is surjective and open map. Indeed, surjectivity is clear and openness is easy as well and is a particular case of a general fact, that any map from a scheme to a field is universally open (\cite[Tag 0383]{StacksProject}). By connectedness of $Y$ we see that $T=Y$. So $Y_{k^s}=\bar{T}$. But this last one is a disjoint union of copies of $X_{k^s}$, which is what we wanted to show by Prop.\ \ref{dictionary}.

\item The smallest normal thickly closed subgroup of $\pipet(X)$ containing $\rmim(\iota)$ is equal to $\ker(p)$: as we already know that this image is contained in the kernel and that the map $\pipet(X) \rarr \Gal_k$ is a quotient map of topological groups, we can apply Prop.\ \ref{dictionary}. Let $Y$ be a connected geometric covering of $X$ such that $Y_{\bk}=Y\times_X X_{\bk}$ splits completely. Denote $Y_{\bk}=\bigsqcup_\alpha X_{\bk,\alpha}$, where by $ X_{\bk,\alpha}$ we label different copies of $X_{\bk}$. By Lm.\ \ref{finitegeomconncomp}, $\pi_0(Y_{\bk})$ is finite, and thus the indexing set $\{\alpha\}$ and the covering $Y \rarr X$ are finite. But in this case, the statement follows from the classical exact sequence of \'etale fundamental groups due to Grothendieck.
\end{enumerate}
\end{proof}

As promised above, we give examples of geometric coverings of $X_{\bk}$ that cannot be defined over any finite field extension $l/k$.

\begin{example}\label{counterexample-with-picture}
Let $X_i= \bbG_{m,\bbQ}$, $i=1,2$. Define $X$ to be the gluing $X=\cupt X_i$ of these schemes at the rational points $1_i: \Spec(\bbQ) \rarr X_i$ corresponding to $1$. Fix an algebraic closure $\overline{\bbQ}$ of $\bbQ$ and so a geometric point $\bar{b}$ over the base $\Spec(\bbQ)$. This gives geometric points $\bx_i$ on $\bX_i=X_{i,\overline{\bbQ}}$ and $X_i$ lying over $1_i$, which we choose as base points for the fundamental groups involved. Similarly, we get a geometric point $\bx$ over the point of gluing $x$ that maps to $\bar{b}$. Then Example \ref{exofgluingmany} gives us a description of the fundamental group $\pipet(X,\bx) \simeq \Big(*^{\rmtop}_{i=1,2} (\piet(\bX_i,\bx_i) \rtimes  \Gal_{\bbQ,i})/\overline{\langle \langle \iota_1(\sigma)=\iota_{2}(\sigma)|\sigma \in \Gal_\bbQ \rangle \rangle} \Big)^{\rmNoohi}$ and of its category of sets:
\begin{multline*}
\pipet(X,\bx) - \rmSets \simeq\\ \simeq\Big\{ S \in \Big(*^{\rmtop} \piet(\bX_1) *^{\rmtop} \piet(\bX_2) *^{\rmtop}  \Gal_{\bbQ} \Big) - \rmSets \Big| \forall_{\sigma \in \Gal_{\bbQ}}\forall_{i} \forall_{\gamma \in \piet(\bX_i)}\forall_{s \in S} \sigma \cdot (\gamma \cdot s) = {}^{\sigma}\gamma \cdot(\sigma \cdot s) \Big\}
\end{multline*}

For the base change $\bX$ to $\overline{\bbQ}$, we have $\pipet(\bX,\bx) \simeq \piet(\bX_1,\bx_1) *^N \piet(\bX_2,\bx_2)$. Recall that the groups $\piet(\bX_i,\bx_i)$ are isomorphic to $\hbbZ(1) = \varprojlim \mu_{n}$ as  $\Gal_{\bbQ}$-modules. Fix these isomorphisms. Let $S = \bbN_{>0}$. Let us define a $\pipet(\bX,\bx)$-action on $S$, which means giving actions by $\piet(\bX_1,\bx_1)$ and $\piet(\bX_2,\bx_2)$ (no compatibilities of the actions required).
Let $\ell$ be a fixed odd prime number (e.g.\ $\ell=3$). We will give two different actions of $\bbZ_\ell(1)$ on $S$ which will define actions of $\hbbZ(1)$ by projections on $\bbZ_\ell(1)$. We start by dividing $S$ into consecutive intervals labelled $a_1,a_3,a_5, \ldots$ of cardinality $\ell^1, \ell^3,\ell^5,\ldots$ respectively. These will be the orbits under the action of $\piet(\bX_1,\bx_1)$. Similarly, we divide $S$ into consecutive intervals $b_2, b_4, b_6,\ldots$ of cardinality $\ell^2, \ell^4, \ldots$.

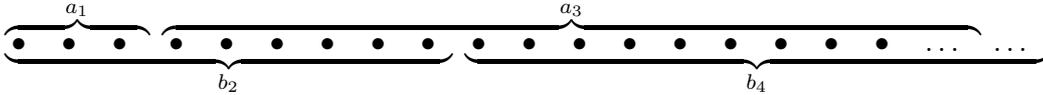
\begin{displaymath}
    \mathrlap{\overbrace{\phantom{\bullet \quad \bullet \quad \bullet \quad}}^{a_1}}
    \mathrlap{\underbrace{\phantom{\bullet \quad \bullet \quad \bullet \quad \bullet \quad \bullet \quad \bullet \quad \bullet \quad \bullet \quad \bullet \quad}}_{b_2}}
    \bullet \quad \bullet \quad \bullet \quad
    \mathrlap{\overbrace{\phantom{\bullet \quad \bullet \quad \bullet \quad \bullet \quad \bullet \quad \bullet \quad \bullet \quad \bullet \quad \bullet \quad \bullet \quad \bullet \quad \bullet \quad  \bullet \quad \bullet \quad \bullet \quad \ldots \quad}}^{a_3}}
    \bullet \quad \bullet \quad \bullet \quad  \bullet \quad \bullet \quad \bullet \quad
    \mathrlap{\underbrace{\phantom{\bullet \quad \bullet \quad \bullet \quad \bullet \quad \bullet \quad \bullet \quad  \bullet \quad \bullet \quad \bullet \quad \ldots \quad \ldots \quad}}_{b_4}}
    \bullet \quad \bullet \quad \bullet \quad \bullet \quad \bullet \quad \bullet \quad  \bullet \quad \bullet \quad \bullet \quad \ldots \quad \ldots
\end{displaymath}  

We still have to define the action on each $a_m$ and $b_m$. We choose arbitrary identifications $b_m \simeq \mu_{\ell^m}$ as $\bbZ_\ell(1)$-modules. Now, fix a compatible system of $\ell^n$-th primitive roots of unity $\zeta=(\zeta_{\ell^n}) \in \bbZ(1)$. For $a_m$'s, we choose the identifications with $\mu_{\ell^m}$ arbitrarily with one caveat: we demand that for any even number $m$, the intersection $b_m \cap a_{m+1}$ contains the elements $1, \zeta_{\ell^{m+1}} \in \mu_{\ell^{m+1}}$ via the chosen  identification $a_{m+1} \simeq \mu_{\ell^{m+1}}$.
As $|b_m \cap a_{m+1}|>0$ and $|b_m \cap a_{m+1}|\equiv 0 \mod \ell$, the intersection $b_m \cap a_{m+1}$ contains at least two elements and we see that choosing such a labelling is always possible.

Assume that $S$ corresponds to a covering that can be defined over a finite Galois extension $K/\bbQ$. Fix $s_0 \in a_1 \cap b_2$. By increasing $K$, we might and do assume that $\Gal_K$ fixes $s_0$. Let $p$ be a prime number $\neq \ell$ that splits completely in $K$ and $\mathfrak{p}$ be a prime of $\calO_K$ lying above $p$. Let $\phi_{\frakp} \in \Gal_K$ be a Frobenius element (which depends on the choice of the decomposition group and the coset of the inertia subgroup). It acts on $\bbZ_\ell(1)$ via $t \mapsto t^p$ and this action is independent of the choice of $\phi_{\frakp}$. Choose $N>0$ such that $p^N \equiv 1 \mod \ell^2$ and let $m$ be the biggest number such that $p^N  \equiv 1 \mod \ell^m$. If $m$ is odd, we look at $p^{\ell N}$ instead. In this case $m+1$ is the biggest number such that $p^{\ell N}  \equiv 1 \mod \ell^{m+1}$ and so, by changing $N$ if necessary, we can assume that $m$ is even, $>1$. The whole point of the construction is the following: if $s \in a_i\cap b_j$ with $i,j<m$ is fixed by $\phi_{\frakp}^N$, then so are $g \cdot s$ and $h\cdot s$ (for $h \in \piet(\bX_1,\bx_1)$ and $g\in \piet(\bX_2,\bx_2)$).  Then moving such $s$ with $g$'s and $h$'s to $b_m\cap a_{m+1}$ leads to a contradiction. Indeed, let $s_1 \in b_m \cap a_{m+1} \subset S$ correspond to $1 \in \mu_{\ell^{m+1}} \simeq a_{m+1}$ (it is possible by the choices made in the construction of $S$). 
Write $s_1= g_m h_{m-1}\ldots h_3g_2h_1 \cdot s_0$ with $h_i \in \piet(\bX_1,\bx_1)$ and $g_j \in \piet(\bX_2,\bx_2)$ (this form is not unique, of course). This is possible thanks to the fact that the sets $a_i$, $b_j$ form consecutive intervals separately such that $b_j$ intersects non-trivially $a_{j-1}$ and $a_{j+1}$. By the construction of $S$ again, there is an element $s_2 \in b_m \cap a_{m+1}$ corresponding to $\zeta_{\ell^{m+1}} \in \mu_{\ell^{m+1}}$ via $a_{m+1} \simeq \mu_{\ell^{m+1}}$. We can now write $s_2$ in two ways:
\begin{displaymath}
s_2 = \zeta \cdot s_1 = g \cdot s_1,
\end{displaymath}
where $g \in \piet(\bX_2,\bx_2)$ and $\zeta$ is the chosen element in $\piet(\bX_1,\bx_1) \simeq \hat{\bbZ}(1)$. By the choices made, the action of $\phi_{\frakp}^N$ fixes the elements $s_1$ and $g \cdot s_1$, while it moves $\zeta \cdot s_1$. Indeed, $\phi_{\frakp}^N \cdot (\zeta \cdot s_1) = (\phi_{\frakp}^N  \zeta \phi_{\frakp}^{-N}) \cdot (\phi_{\frakp}^N \cdot s_1) = \zeta^{p^N} \cdot (\phi_{\frakp}^N \cdot s_1) = \zeta^{p^N} \cdot s_1 = \zeta_{\ell^{m+1}}^{p^N} \neq \zeta_{\ell^{m+1}} = \zeta \cdot s_1 \in \mu_{\ell^{m+1}} \simeq a_{m+1}$ -- a contradiction.
\end{example}

\begin{example}\label{counterexample-with-matrices}
Let $X_i= \bbG_{m,\bbQ}$, $i=1,2,3$ and let $X_4, X_5$ be the nodal curves obtained from gluing $1$ and $-1$ on $\bbP^1_{\bbQ}$ (see Ex.\ \ref{exnodal}). Define $X$ to be the gluing $X=\cupt X_i$ of all these schemes at the rational points corresponding to $1$ (or the image of $1$ in the case of the nodal curves). We fix an algebraic closure $\overline{\bbQ}$ of $\bbQ$ and so fix a geometric point $\bar{b}$ over the base $\Spec(\bbQ)$. We get geometric points $\bx_i$ on $\bX_i = X_i \times_{\bbQ} \overline{\bbQ}$ lying over $1$. We have and fix the following isomorphisms of $\Gal_{\bbQ}$-modules. For $1 \leq i \leq 3$, $\piet(\bX_i,\bx_i) \simeq \hbbZ(1)$ and for $4 \leq j \leq 5$, we have $\pipet(\bX_j,\bx_j) \simeq \langle t^\bbZ \rangle$ (i.e.\ $\bbZ$ written multiplicatively). Let $t_i \in \pipet(\bX_i,\bx_i)$ be the elements corresponding via these isomorphisms to a fixed inverse system of primitive roots $\zeta \in \hbbZ(1)$ (for $i=1,2,3$) and to $t \in \langle t^\bbZ \rangle$ (for $i=4,5$). Example \ref{exofgluingmany} gives a description of the fundamental group 
\[\pipet(X,\bx) \simeq \bigg(*^N_{i=1,2,3} (\hbbZ(1)_i \rtimes  \Gal_{\bbQ,i}) *^N_{j=4,5} (\langle t^\bbZ \rangle \times \Gal_{\bbQ,j})\Big/\big\langle \big\langle \overline{\iota_i(\sigma)=\iota_{i'}(\sigma)\big|{\footnotesize \begin{array}{c}\sigma \in \Gal_\bbQ \\[-3pt]  i,i' = 1, \ldots, 5 
\end{array}}} \big\rangle \big\rangle \bigg)^{\rmNoohi}\]
 and of its category of sets:
  \begin{multline*}
  \pipet(X,\bx) - \rmSets \simeq 
  \Big\{ S \in \Big(*^{\rmtop}_{1 \leq i \leq 3} \hbbZ(1) *^{\rmtop}\langle t^\bbZ \rangle^{*2}*^{\rmtop} \Gal_{\bbQ} \Big) - \rmSets \Big|\\ \Big| \forall_{\sigma \in \Gal_{\bbQ}}\forall_{1\leq i \leq 3} \forall_{\substack{\gamma \in \bbZ(1)_i\\ w \in \langle t^\bbZ \rangle^{*2}}}\forall_{s \in S} \sigma \cdot (\gamma \cdot s) = ^{\sigma}\gamma \cdot(\sigma\cdot s) \text{ and } \sigma \cdot (w \cdot s) = w \cdot (\sigma \cdot s) \Big\}
  \end{multline*}
  Let $G=\Big\{\left( \begin{array}{cc}
  * & * \\
    & *  \end{array}\right)\Big\} \subset \GL_2(\bbQ_\ell)$ be the subgroup of upper triangular matrices. Fix $u_1 \in \bbZ_\ell^\times$ such that $u_1^p \neq u_1$. Let $H= *^{\rmtop}_{i} \piet(\bX_i,\bx_i)$ and define a continuous homomorphism $\psi: H \rarr G$ by:
    \begin{gather*}
      \psi(t_1) = \left( \begin{array}{cc}
      u_1 &  \\
       & 1  \end{array}\right),
       \text{ } \psi(t_2)= \left( \begin{array}{cc}
     1 &  \\
       & u_1  \end{array}\right),
       \text{ } \psi(t_3)= \left( \begin{array}{cc}
     1 & 1 \\
       & 1  \end{array}\right),
        \text{ } \psi(t_4)= \left( \begin{array}{cc}
     \ell &  \\
       & 1  \end{array}\right),
        \text{ } \psi(t_5)= \left( \begin{array}{cc}
     1 &  \\
       & \ell  \end{array}\right).
  \end{gather*}
  It is easy to see that $\psi$ is surjective.
  
  Let $U \subset G$ be the subgroup of matrices with elements in $\bbZ_\ell$, i.e.\ $U=\Big\{\left( \begin{array}{cc}
  * & * \\
    & *  \end{array}\right)\Big\} \subset \GL_2(\bbZ_\ell)$. It is an open subgroup of $G$. Thus, using $\psi$ and the fact that $H^{\rmNoohi}=\pipet(\bX,\bx)$, we get that $S:=G/U$ defines a $\pipet(\bX,\bx)$-set. It is connected (i.e.\ transitive) and so corresponds to a connected geometric covering of $\bX$. 
  Assume that it can be defined over a finite extension $L$ of $\bbQ$. We can assume $L/\bbQ$ is Galois. By the description above, it means that there is a compatible action of groups $\bbZ(1)_i$, $\bbZ^{*2}$ and $\Gal_L$ on $S$. By increasing $L$, we can assume moreover that $\Gal_L$ fixes $[U]$.
  
  Choose $p\neq \ell$ that splits completely in $\Gal_L$, fix a prime $\frakp$ of $L$ dividing $p$ and let $\phi_{\frakp} \in \Gal_L$ denote a fixed Frobenius element. Let $t_3^{u_1}$ denote the unique element of $\psi_{|\piet(\bX_3,\bx_3)}^{-1}\left(\left( \begin{array}{cc}
    1 & u_1 \\
      & 1  \end{array}\right)\right)$. Let $n \gg 0$. An easy calculation shows that $\psi(t_4^{-n}t_1 t_3 t_1^{-1}t_3^{-u_1} t_4^n)=1_{\GL_2(\bbQ_\ell)} \in U$. Then $\phi_{\frakp}\cdot [U]=\phi_{\frakp} \cdot (t_4^{-n}t_1 t_3 t_1^{-1}t_3^{-u_1} t_4^n \cdot [U])= {}^{\phi_{\frakp}}(t_4^{-n}t_1 t_3 t_1^{-1}t_3^{-u_1} t_4^n)\cdot(\phi_{\frakp}\cdot [U]) = t_4^{-n}t_1^p t_3^p t_1^{-p}t_3^{-pu_1} t_4^n\cdot [U]$. But
  \begin{align*}
  \psi(t_4^{-n}t_1^p t_3^p t_1^{-p}t_3^{-pu_1} t_4^n) 
  &=\left( \begin{array}{cc}
   \ell^{-n} &  \\
    & 1  \end{array}\right)
  \left( \begin{array}{cc}
   u_1^p &  \\
    & 1  \end{array}\right)
   \left( \begin{array}{cc}
  1 & p \\
    & 1  \end{array}\right) 
  \left( \begin{array}{cc}
   u_1^{-p} &  \\
    & 1  \end{array}\right)
    \left( \begin{array}{cc}
  1 & -pu_1\\
    & 1  \end{array}\right)
    \left( \begin{array}{cc}
  \ell^n &  \\
    & 1  \end{array}\right)\\
  &=\left( \begin{array}{cc}
  1 & \ell^{-n}p(u_1^p-u_1)\\
    & 1  \end{array}\right) \notin U.
  \end{align*}
  As $n \gg 0$ and $u_1^p \neq u_1$, it follows that $\phi_{\frakp}\cdot[U] \neq [U]$ -- a contradiction.
  \end{example}

It is important to note, that the above (counter-)examples are possible only when considering the geometric coverings that are not trivialized by an \'etale cover (but one really needs to use the pro-\'etale cover to trivialize them). In \cite{BhattScholze}, the category of geometric coverings trivialized by an \'etale cover on $X$ is denoted by $\rmLoc_{X_{\et}}$ and the authors prove the following
\begin{fact}(\cite[Lemma 7.4.5]{BhattScholze})
Under $\rmLoc_X \simeq \pipet(X)-\rmSets$, the full subcategory $\rmLoc_{X_{\et}} \subset \rmLoc_X$ corresponds to the full subcategory of those $\pipet(X)-\rmSets$ where an open subgroup acts trivially.
\end{fact}
We are now going to prove:
\begin{proposition}\label{injectivity-for-prodiscrete}
Let $X$ be a geometrically connected separated scheme of finite type over a field $k$. Let $Y \in \rmCov_{X_{\bk}}$ be such that $Y \in \rmLoc_{(X_{\bk})_\et}$. Then there exists a finite extension $l/k$ such and $Y_0 \in \rmCov_{X_l}$ such that $Y \simeq Y_0 \times_{X_l} X_{\bk}$.
\end{proposition}
\begin{proof}By the topological invariance (Prop.\ \ref{topoinv}), we can replace $\bk$ by $k^{\rmsep}$ if desired.
By the assumption $Y \in \rmLoc_{(X_{\bk})_\et}$, there exists an \'etale cover of finite type that trivializes $Y$. Being of finite type, it is a base-change $X'_{\bk}=X'\times_{\Spec(l)}\Spec(\bk) \rarr X_{\bk}$ of an \'etale cover $X' \rarr X_{l}$ for some finite extension $l/k$. Thus, $Y_{|X'_{\bk}}$ is constant (i.e.\ $\simeq \sqcup_{s \in S} X' = \underline{S}$) and the isomorphism between the pull-backs of $Y_{|X'_{\bk}}$ via the two projections $X_{\bk}' \times_{X_{\bk}} X_{\bk}' \rightrightarrows X_{\bk}'$ is expressed by an element of a constant sheaf $\underline{\rmAut(S)}(X_{\bk}' \times_{X_{\bk}} X_{\bk}')=\rmAut(\underline{S})(X_{\bk}' \times_{X_{\bk}} X_{\bk}')$ (we use the fact that $X'_{\bk}$ is \'etale over $X_{\bk}$, and thus $\pi_0(X_{\bk}' \times_{X_{\bk}} X_{\bk}')$ is discrete, in this case even finite). By enlarging $l$, we can assume that the connected components of the schemes involved: $X'$, $X'\times_{X_l} X'$ etc. are geometrically connected over $l$. Define $Y_0' = \sqcup_{s \in S} X'$. The discussion above shows that the descent datum on $Y_{|X'_{\bk}}$ with respect to $X'_{\bk} \rarr X_{\bk}$ is in fact the pull-back of a descent datum on $Y_0'$ with respect to $X' \rarr X_l$. As \'etale covers are morphisms of effective descent for geometric coverings (this follows from the fpqc descent for fpqc sheaves and the equivalence $\rmCov_{X_l} \simeq \rmLoc_{X_l}$ of \cite[Lemma 7.3.9]{BhattScholze}), the proof is finished. 
\end{proof}

\begin{rmk}
  Over a scheme with a non-discrete set of connected components, $\underline{\rmAut(S)}$ might not be equal to $\rmAut(\underline{S})$.
\end{rmk}

Proposition \ref{injectivity-for-prodiscrete} shows that our main theorem is significantly easier for $\pisga$.
\begin{corollary}\label{cor:injectivity-for-prodiscrete-cor}
  Let $X$ be a geometrically connected separated scheme of finite type over a field $k$. Fix an algebraic closure $\bk$ of $k$. Then
  \begin{displaymath}
    \pisga(X_{\bk}) \rarr \pisga(X)
  \end{displaymath}
  is a topological embedding.
\end{corollary}

\subsection{Preparation for the proof of Theorem \ref{injectivity-on-the-left}}
We are going to use the following proposition.
\begin{proposition}\label{dominbyrat}
Let $X$ be a scheme of finite type over a field $k$ with a $k$-rational point $x_0$ and assume that $X_{\bk}$ is connected. Let $Y_1, \ldots, Y_N$ be a set of connected finite \'etale coverings of $X_{\bk}$. Then there exists a finite Galois \'etale covering $Y$ of $X$ such that for all $1 \leq i \leq N$, there exists a surjective map $Y_{\bk} \epirarr Y_i$ of coverings of $X_{\bk}$.
\end{proposition}

\begin{proof}
There is a finite connected Galois covering of $X_{\bk}$ dominating $Y_1, \ldots, Y_N$. Thus, we can assume $N=1$ and $Y_1$ is Galois. Fix a geometric point $\bx_0$ over $x_0$. The $k$-rational point $x_0$ gives a splitting $s:\Gal_k \rarr \piet(X,\bx_0)$, allowing to write $\piet(X,\bx_0) \simeq \piet(X_{\bk},\bx_0) \rtimes \Gal_k$ and so an action of $\Gal_k$ on $\piet(X_{\bk},\bx_0)$.
Fix a geometric point $\by$ on $Y_1$ over $\bx_0$. The group $U = \piet(Y_1,\by)$ is a normal open subgroup of $\piet(X_{\bk},\bx_0)$. As $Y_1$ is defined over a finite Galois field extension $l/k$ (contained in $\bk$), it is easy to check that $\Gal_l \subset \Gal_k$ fixes $U$, i.e.\ ${}^\sigma U = U$ for $\sigma \in \Gal_l$. It follows that the set of conjugates ${}^\sigma U$ is finite, of cardinality bounded by $[l:k]$. Define $V = \cap_{\sigma \in \Gal_k} {}^\sigma U$. It follows that this is an open subgroup of $\piet(X_{\bk},\bx_0)$ fixed by the action of $\Gal_k$. Moreover, it is normal, as $g(\cap_{\sigma \in \Gal_k} {}^\sigma U)g^{-1} = \cap_{\sigma \in \Gal_k} g{}^\sigma U g^{-1} = \cap_{\sigma \in \Gal_k} {}^\sigma(({}^{\sigma^{-1}}g) U ({}^{\sigma^{-1}}g^{-1})) = \cap_{\sigma \in \Gal_k} g{}^\sigma U g^{-1}$, due to normality of $U$. The open normal subgroup $V \cdot \Gal_k = V \rtimes \Gal_k < \piet(X_{\bk},\bx_0) \rtimes \Gal_k$ corresponds to a covering with the desired properties.
\end{proof}

Before starting the proof, we need to collect some facts about the Galois action on the geometric $\piet$. They are discussed, for example, in \cite[Ch.\ 2]{Stix-book}. The existence, functoriality and compatibility with compositions of the action can be readily seen to generalize to $\pipet$ as well, but \uline{note (see the last point below) that one has to be careful when discussing continuity}. For a connected topologically noetherian scheme $W$ and geometric points $\bw_1,\bw_2$, let $\pipet(W,\bw_1,\bw_2) = \mathrm{Isom}_{\rmCov_{W_{\bk}}}(F_{\bw_1},F_{\bw_2})$ denote the set of isomorphisms of the two fibre functors, topologized in a way completely analogous to the case when $\bw_1 = \bw_2$. By Cor. \ref{base-points}, it is a bi-torsor under $\pipet(W,\bw_1)$ and $\pipet(W,\bw_2)$. The bi-torsors under profinite groups $\piet(W,\bw_1,\bw_2)$ are defined similarly and are rather standard. For a geometrically unibranch $W$, the two notions match.
\begin{lemma}\label{lem:Gal-action-on-paths}
  For a scheme $W$ of finite type over $k$ and two geometric points $\bar{w}_1, \bar{w}_2$ on $W$ lying over $k$-points, there is an abstract $\Gal_k$-action on $\piet(W_{\bk},\bw_1,\bw_2)$ and $\pipet(W_{\bk},\bw_1,\bw_2)$ such that
  \begin{enumerate}[label=\alph*)]
    \item It is given by $\psi_\sigma = \piet(\id_W \times_{\Spec(k)} \Spec(\sigma^{-1}),\bw_1,\bw_2)$ or an analogously defined automorphism of $\pipet(W_{\bk},\bw_1,\bw_2)$. This makes sense as $\bw_1, \bw_2$ are $\Gal_k$-invariant.
    \item The morphism $\pipet(W_{\bk},\bw_1,\bw_2) \to \piet(W_{\bk},\bw_1,\bw_2)$ is $\Gal_k$-equivariant. Similarly, maps of schemes $(W,\bw_1,\bw_2)  \to (W',\bw_1,\bw_2)$  induce $\Gal_k$-equivariant maps on $\piet$ and $\pipet$.
   \item\label{lem:Gaop:pt:composition-compatibility} For three geometric points $\bw_1, \bw_2, \bw_3$, the Galois action is compatible with the composition maps, i.e.\ for any $\sigma \in \Gal_k$, the following diagram commutes
  \[
    \xymatrix{
    \pipet(W_{\bk},\bw_2,\bw_3) \times \pipet(W_{\bk},\bw_1,\bw_2) \ar[r]^(0.62){(-)\circ(-)} \ar[d]_{(\psi_\sigma, \psi_\sigma)} & \pipet(W_{\bk},\bw_1,\bw_3) \ar[d]^{\psi_\sigma} \\
    \pipet(W_{\bk},\bw_2,\bw_3)\times \pipet(W_{\bk},\bw_1,\bw_2) \ar[r]^(0.62){(-)\circ(-)} & \pipet(W_{\bk},\bw_1, \bw_3)
    }
  \]
  Similarly for $\piet$. Inductively, this also holds for arbitrary (finite) composition maps.

\item\label{lem:pt:Gal-action-as-conjugation} Let $s_{w_1},s_{w_2}$ be the sections of the maps $\piet(W,\bw_i) \to \Gal_k$ coming from rational points $w_1, w_2$. Then $\psi_\sigma(\gamma) = s_{w_2}(\sigma) \circ \gamma \circ s_{w_1}(\sigma^{-1})$ for $\gamma \in \piet(W_{\bk},\bw_1,\bw_2) \subset \piet(W,\bw_1,\bw_2)$. Note that this only makes sense thanks to the fundamental exact sequence for $\piet$ (and its version for the sets of paths, see \cite[Prop.\ 18]{Stix-book}). Thus, at this stage, we cannot make an analogous statement for $\pipet$.
  \end{enumerate}
In terms of continuity of $\psi_\sigma$, there is a priori a huge difference in how much we can say about $\piet$ and $\pipet$.
\begin{enumerate}[label=\alph*)]\addtocounter{enumi}{5}
  \item For each $\sigma$, the map $\psi_\sigma$ is continuous as an automorphism of $\piet(W_{\bk},\bw_1,\bw_2)$ and $\pipet(W_{\bk},\bw_1,\bw_2)$.
\item The action $\Gal_k \times \piet(W_{\bk},\bw_1,\bw_2) \to \piet(W_{\bk},\bw_1,\bw_2)$ is continuous. Note, however, that {\bf at this stage of the proof we do not know whether this is true for $\pipet$}. In fact, this is closely related to the main result we need to prove. 
\end{enumerate}  
\end{lemma} 

\subsection{Proof that $\pipet(X_{\bk}) \rarr \pipet(X)$ is a topological embedding}
In this subsection we finally prove our main result.
\begin{theorem}\label{injectivity-on-the-left}
  Let $k$ be a field and fix an algebraic closure $\bk$ of $k$. Let $X$ be a scheme of finite type over $k$
  such that the base-change $X_{\bk}$ is connected. Let $\bx$ be a $\Spec(\bk)$-point on $X_{\bk}$. Then the induced map
  \begin{displaymath}
  \pipet(X_{\bk},\bx) \rarr \pipet(X,\bx)
  \end{displaymath}
  is a topological embedding.
  \end{theorem}

Then, we will derive the final form of the fundamental exact sequence.

\begin{theorem}\label{exactness-in-geometric-to-arithmetic-as-abstract}
With the assumptions as in Thm.\ \ref{injectivity-on-the-left}, the sequence of abstract groups
\begin{displaymath}
1 \rarr \pipet(X_{\bk},\bx) \rarr \pipet(X,\bx) \rarr \Gal_k \rarr 1
\end{displaymath}
is exact.

Moreover, the map $\pipet(X_{\bk},\bx) \rarr \pipet(X,\bx)$ is a topological embedding and the map $\pipet(X,\bx) \rarr \Gal_k$ is a quotient map of topological groups.
\end{theorem}

In the proof, after some preparatory steps (e.g.\ extending the field $k$), we define the set of \emph{regular loops} in $\pipet(X_{\bk})$ with respect to a fixed open subgroup $U <^{\circ} \pipet(X_{\bk},\bx)$ and use it to construct an Galois invariant open subgroup $V$ inside of $U$ (see Steps II and III below). {\bf There is also an alternative approach} to proving the existence of $V$ that avoids the direct construction involving regular loops. {\bf We sketch it in Rmk.\ \ref{rmk:quick-approach}}. While this latter approach is quicker, it is less instructive: as explained in Rmk.\ \ref{rmk:example-not-regular-loop} below, the notion of a regular loop provides an insight of what goes wrong in the counterexample Ex.\ \ref{counterexample-with-picture}. Still, it might be worth having a look at, as our main approach is rather lengthy.

{\bf \emph{Step I: Setting things up and applying van Kampen}}

For any finite extension $\bk/l/k$ of $k$, the map $\pipet(X_l,\bx) \rarr \pipet(X,\bx)$ is an embedding of an open subgroup and we have a factorization $\pipet(X_{\bk},\bx) \rarr \pipet(X_l,\bx) \rarr \pipet(X,\bx)$. Here, we have  tacitly lifted $\bx$ to $X_l$. Thus, we can start by replacing $k$ by a finite extension. Considering the normalization $X^\nu \rarr X$, base-changing the whole problem to a finite extension $l$ of $k$, considering the factorization $l/l'/k$ into separable and purely inseparable extension of fields, and using first that the base-change along a separable field extension of a normal scheme is normal and then the topological invariance of $\pipet$, we can assume that we have a surjective finite morphism $h:\tX \rarr X$ such that the connected components of $\tX$, $\tX \times_X \tX$, $\tX\times_X\tX \times_X \tX$ are geometrically connected, have rational points and for each $W \in \pi_0(\tX)$, there is $\pipet(W) = \piet(W)$ and $\pipet(W_{\bk}) = \piet(W_{\bk})$. 

Let $\tX = \sqcup_{v \in \rmVert} \tX_v$ be the decomposition into connected components. Note that the indexing set $\rmVert$ is finite. For each $t \in \pi_0(\tX) \cup \pi_0(\tX \times_X \tX) \cup \pi_0(\tX \times_X \tX \times_X \tX))$, we fix a $k$-rational point $x(t)$ on $t$ and a $\bk$-point $\bx(\bt)$ on $\bt=t_{\bk}$ lying over $x(t)$. We will often write $\bx_t$ to mean $\bx(t)$. Let us fix $v_{\bx} \in \rmVert$ for the rest of the text and say that the image of $\bx(\tX_{v_{\bx},\bk})$ in $X_{\bk}$ will be the fixed geometric point $\bx$ of $X_{\bk}$ and its image in $X$ the fixed geometric point of $X$. For any $W_{\bk
}, W_{\bk}' \in \pi_0(S_\bullet(\bh))$ and every boundary map $\bpartial: W_{\bk} \rarr W_{\bk}'$, we fix paths $\gamma_{W_{\bk}',W_{\bk}} \in \pipet(W_{\bk}', \bx_{W_{\bk}'}, \bpartial(\bx_{W_{\bk}}))$ between the chosen geometric points, as in Cor.\ \ref{geomvK}. This is possible thanks to Lm.\ \ref{geomfunctorscompatibility}. We define $\gamma_{W',W}$ to be the image of this path.

Let $\bh: \tX_{\bk} \rarr X_{\bk}$ be the base-change of $h$. We choose a maximal tree $T$ (resp.\ $T'$) in the graph $\Gamma =\pi_0(S_\bullet(h))_{\leqslant 1}$ (resp.\ $\Gamma' =\pi_0(S_\bullet(\bh))_{\leqslant 1}$). After making these choices, we can apply Cor.\ \ref{geomvK} with Rmk.\ \ref{remark-relations-in-vK-for-normalization} to write the fundamental groups of $(X,\bx)$ and $(X_{\bk},\bx)$. This way we get a diagram

    \[
    \xymatrix{
    \bigg(\Big(\big(*^{\rmtop}_v \piet(\tX_{v,\bk},\bx_v)\big) \topfp \pi_1(\Gamma',T')\Big)/ \overline{\langle R_1', R_2' \rangle^{nc}} \bigg)^{\rmNoohi}\ar[r]^(0.75){\simeq }\ar[d] & \pipet(X_{\bk},\bx)   \ar[d]  \\
      \bigg(\Big(\big(*^{\rmtop}_v \piet(\tX_v,\bx_v)\big)  \topfp \pi_1(\Gamma,T)\Big)/ \overline{\langle R_1, R_2 \rangle^{nc}} \bigg)^{\rmNoohi} \ar[r]^(0.75){\simeq }  & \pipet(X,\bx) 
    }
  \]

where $\overline{(\ldots)}$ denotes the topological closure, $\langle R \rangle^{nc}$ denotes the normal subgroup generated by the set $R$, and $R_1, R_1', R_2, R_2'$ are as in Rmk.\ \ref{remark-relations-in-vK-for-normalization}.

Note that, while the (connected components of the) fibre products $\tX \times_X \tX$, $\tX \times_X \tX \times_X \tX$ are not necessarily normal nor satisfy $\pipet(W) = \piet(W)$, we can effectively work as if this was the case, see Rmk. \ref{remark-relations-in-vK-for-normalization}.
\begin{obs}\label{obs:setup}
The maps and groups above enjoy the following properties.
\begin{enumerate}[label=\alph*)]
    \item By Lm.\ \ref{vKfunctoriality}, the left vertical map is the Noohi completion of the obvious map of the underlying quotients of free topological products.
    \item By geometrical connectedness of the schemes in sight, we can (and do) identify
    \[
    \pi_0(S_\bullet(h)) = \pi_0(S_\bullet(\bh)), \quad  \Gamma' = \Gamma \quad \textrm{ and } \quad T' = T 
    \]
    \item \label{obs:pt:R_2'=R_2} As $\gamma_{W',W}$'s are chosen to be the images of $\gamma_{W_{\bk}',W_{\bk}}$'s, we see that $\alpha_{abc}^{(f)}$'s appearing in $R_2$, and so a priori elements of $\piet(\tX_v,\bx_v)$'s, are in fact in $\piet(\tX_{v,\bk},\bx_v)$. It follows that
    \[
    R_2' = R_2
    \]
    \item The $k$-rational points $x(W)$ give identification
    \[
    \piet(W,\bx_W)\simeq \piet(W_{\bk},\bx_W)\rtimes \Gal_k
    \]
    When $W = \tX_v$ for $v \in \rmVert$, we will write $\Gal_{k,v}$ in the identification above to distinguish between different copies of $\Gal_k$ in the van Kampen presentation of $\pipet(X,\bx)$.
    \item As $\gamma_{W',W}$ is the image of the path $\gamma_{W_{\bk}',W_{\bk}}$ on $W_{\bk}'$, it maps to the trivial element of\\ ${\piet(\Spec(k), \bx(W),\bx(W')) =  \Gal_k}$. It implies, that the following diagram commutes
\begin{center}
\begin{tikzpicture}
\matrix(a)[matrix of math nodes,
row sep=1em, column sep=1.5em,
text height=1ex, text depth=0.25ex]
{\piet(W_{\bk},\bx(W)) &  & \piet(W,\bx(W))  &  {} &   \\  &  &  &   \Gal_k \\ \piet(W'_{\bk},\bx(W')) & &  \piet(W',\bx(W'))  &  {} &  \\};

\path[->,font=\scriptsize] (a-1-3) edge  (a-2-4);
\path[->,font=\scriptsize] (a-3-3) edge (a-2-4);
\path[->] (a-1-1) edge (a-3-1);
\path[->] (a-1-3) edge (a-3-3);
\path[->] (a-1-1) edge (a-1-3);
\path[->] (a-3-1) edge (a-3-3);

\end{tikzpicture}
\end{center}
\end{enumerate}
\end{obs}

Let $P$ be a walk in $\Gamma$, i.e.\ a sequence of consecutive edges (with possible repetitions) $e_1, \ldots, e_m$ in $\Gamma$ with an orientation such that the terminal vertex of $e_i$ is the initial vertex of $e_{i+1}$. Using the orientation of $\Gamma$, it can be written as $\epsilon_1 e_1 \ldots \epsilon_m e_m$ with $\epsilon_i \in \{\pm\}$ indicating whether the orientation agrees or not. This will come handy as follows: define $\partial^+_0= \partial_0, \partial^-_0= \partial_1, \partial^+_1= \partial_1, \partial^-_1= \partial_0$.

For each $P$ as above with a vertex sequence $(v_1,v_2,\ldots,v_{m+1})$, there is a map

    \begin{center}
        \begin{tikzcd}[row sep=0.01em]
          \piet(\tX_{v_{m+1},\bk},\partial_1^{\epsilon_m}(\bx_{e_m}),\bx_{v_{m+1}})  \times \piet(\tX_{v_m,\bk},\bx_{v_m},\partial_0^{\epsilon_m}(\bx_{e_m}))  \times & {} & {} \\ \vdots & \ni &  (\gamma_{2m},\ldots,\gamma_1) \arrow[mapsto]{dddd}\\
          \times \piet(\tX_{v_2,\bk},\partial_1^{\epsilon_1}(\bx_{e_1}),\bx_{v_2}) \times \piet(\tX_{v_1,\bk},\bx_{v_1},\partial_0^{\epsilon_1}(\bx_{e_1})) \ar{ddd}  & {} &  {}\\    {} & {} & {} \\ {} & {} & {} \\
          \pipet(X_{\bk},\bx_{v_1},\bx_{v_{m+1}}) & \ni &  \gamma_{2m} \circ\ldots \circ \gamma_1
      \end{tikzcd}
    \end{center}

In the following, we will use $\circ_{?}$ to denote the ``composition of \'etale paths'' and $\bullet_{?}$ to denote the multiplication in some group(oid) $?$. When $? = \pipet(X_{\bk},\bx) \textrm{ or } \pipet(X,\bx)$, we will skip the subscript. While we could just use $\circ_{?}$ everywhere, it is sometimes convenient to keep track of when some paths ``have been closed'' by using $\bullet_{?}$.

{\bf \emph{Step II: Defining regular loops in $\pipet(X_{\bk},\bx)$}}

\begin{definition}\label{defn:paths-special-form}
  An element $\gamma \in \mathrm{Isom}_{\rmCov_{X_{\bk}}}(F_{\bx_w},F_{\bx_v})$ is called an \emph{\'etale path of special form supported on $P$} if it lies in the image of the composition map above for some walk $P$ starting in $w$ and ending in $v$.
  
  Any element $(\gamma_{2m},\ldots,\gamma_1)$ in the preimage of such $\gamma$ will be called a \emph{presentation} of $\gamma$ with respect to $P$.
\end{definition}

For a walk $P$, denote by $l(P)$ the length of $P$, i.e.\ the number of consecutive edges (not necessarily different) it is composed of.

\begin{obs}\label{obs:rho-v-special}
  A useful example of a path of special form is the following. In the van Kampen presentation, the maps $\piet(\tX_{v,\bk},\bx_v) = \pipet(\tX_{v,\bk},\bx_v) \to \pipet(X_{\bk},\bx)$ are given by
  \[
    \rho_v(-) = \gamma_v^{-1} \circ (-) \circ \gamma_v
  \]
  where $\gamma_v \in \pipet(\tX_{v,\bk},\bx,\bx_v)$ is defined as follows: if $P_{v_{\bx},v} \subset T$ denotes the unique shortest path in the tree $T \subset \Gamma$ (forgetting the orientation) from $v_{\bx}$ to $v$, then the choices of paths $\gamma_{W_{\bk}',W_{\bk}}$ made when applying the van Kampen theorem give a unique \'etale path of special form $\gamma_v$ supported on $P_{v_{\bx},v}$.
\end{obs}

Before introducing the main objects of the proof, we note a simple result.
\begin{lemma}\label{lem:continuity-psi-sigma-gamma}
  For a fixed path $\gamma \in \pipet(X_{\bk},\bx,\by)$ of special form, the map $\Gal_k \ni \sigma \mapsto \psi_\sigma(\gamma) \in \pipet(X_{\bk},\bx,\by)$ is continuous.
\end{lemma}
\begin{proof}
  This follows from the continuity of the composition maps of paths and the fact that the statement is true for $\piet$.
\end{proof}

To prove Thm.\ \ref{injectivity-on-the-left}, {\bf it is enough to prove the following statement:} any connected geometric covering $Y$ of $X_{\bk}$ can be dominated by a covering defined over a finite separable extension $l/k$.

Indeed, let $Y' \in \rmCov_{X_l}$ be a connected covering that dominates $Y$ after base-change to $\bk$. By looking at the separable closure of $k$ in $l$ and using the topological invariance of $\pipet$, we can assume $l/k$ is separable. The composition $Y'' = Y' \to X_l \to X$ is an element of $\rmCov_X$ and there is a diagonal embedding $Y'\times_{\Spec(l)} \Spec(\bk) \to Y''\times_{\Spec(k)} \Spec(\bk)$. By Prop.\ \ref{dictionary}\ref{dictionary-kernel}, the proof will be finished.

Let us fix a connected $Y \in \rmCov$ till the end of the proof and denote by $S=Y_{\bx}$ the corresponding $\pipet(X_{\bk},\bx)$-set. Fix some point $s_0 \in S$ and let $U = \rmStab_{\pipet(X_{\bk},\bx)}(s_0)$.

\begin{definition}
  For each $v \in \rmVert$, define
\[
  O_v^N = \left\{s \in F_{\bx_v}(Y) | \exists_{\substack{\textrm{walk } P \text{, }\\ l(P) \leq N}} \exists_{\substack{\gamma \textrm{ of sp. form }\\\textrm{supp. on } P}} s = \gamma \cdot s_0 \right\}
\]
and call it the set of ``elements at $v$ reachable in at most $N$ steps''.
\end{definition}
The following is a crucial observation regarding $O^N_v$.
\begin{lemma*}
  For any $v$ and $N$, the set $O^N_v$ is finite.
\end{lemma*}
\begin{proof}
  We proceed by induction on $N$. For $N=1$, the walks of length not greater than $N$ starting in $v_0$ (are either trivial or) consist of a single edge whose initial vertex is necessarily $v_{\bx}$. As $\Gamma$ is finite, there are only finitely many such edges. Let us fix one, named $e$, with vertices $v_0,w$.  
  We need to show that the set
\[
  \left\{(\theta\circ \delta) \cdot s_0 \in F_{\bx_w}(Y) | \delta \in \piet(\tX_{v_{\bx}, \bk},\bx,\partial_1^{\epsilon(e)}(\bx_e)), \quad \theta \in \piet(\tX_{w,\bk},\partial_0^{\epsilon(e)}(\bx_e),\bx_w) \right\}
\]
is finite. However, as in general the sets $\piet(W,\bx_1,\bx_2)$ are (bi-)torsors under profinite groups (namely $\piet(W,\bx_1)$ and $\piet(W,\bx_2)$) and the maps and actions in sight are continuous, we see that the finiteness of this last set follows directly from finiteness of orbits of points in discrete sets under an action by a profinite group.

Now, to see the inductive step, assume that the claim is true for $N$. To prove it for $N+1$, note that any element in $O_v^{N+1}$ can be connected by a single edge to an element of $O_w^N$ (for some vertex $w$). As $O_w^N$ is finite and as we have just explained that, starting from a fixed point,  one can only reach finitely many points by applying \'etale paths of special forms supported on a single edge, the result follows.
\end{proof}

Now, for each $v\in \rmVert$ and $N \in \mathbb{N}$, define $C_v^N \in \piet(\tX_v,\bx_v) - \rmSets$ so that
\begin{enumerate}
  \item The set $C_v^N$ is Galois as a $\piet(\tX_v,\bx_v)$-set.
  \item The pullback of $C_v^N$ to $\piet(\tX_{v,\bk},\bx_v) - \rmSets$ dominates each of the $\piet(\tX_{v,\bk},\bx_v)$-orbits of elements of $O_v^N$. 
  \item There is a surjection $C^{N+1}_v \epirarr C^N_v$ of $\piet(\tX_{v},\bx_v)$-sets.
\end{enumerate}

We can find sets satisfying the first two conditions by applying Prop.\ \ref{dominbyrat}, and the last condition can be guaranteed by choosing the $C_v^N$'s inductively (for a given $v$).

We now proceed to define a subgroup of $\pipet(X,\bx)$ that will lead to the desired $\pipet(X,\bx)$-set. For that we need to find a suitably large subgroup of elements of $U$ that are well behaved under the Galois action.

\begin{definition}\label{defn:regular-loops}
We will call an element $g \in \pipet(X_{\bk},\bx)$ a \emph{regular loop} (with respect to $U$) if there exists $v, m$, a walk $P$ of length $m$ from $v_1 = v_{\bx}$ to $v_{m+1} = v$, \'etale paths $\gamma, \gamma'$ of special form supported on $P$ and $P^{-1}$, respectively, and $\beta \in \piet(\tX_{v,\bk},\bx_v)$ such that
\begin{itemize}
  \item  $g = \gamma' \circ \beta \circ \gamma$;
  \item $\beta$ acts trivially on $C_v^m$, i.e.
  \[
    \beta \in \ker
    \left(\piet(\tX_{v,\bk},\bx_v) \to \rmAut(C_v^m)\right)
  \]
  \item there exist presentations $(\gamma_{2m},\ldots,\gamma_1) $ and $(\gamma_1',\ldots, \gamma_{2m}')$ of $\gamma$ and $\gamma'$ such that the following condition is satisfied. For any $1 \leqslant i \leqslant m$, there is
\[
  \gamma'_{2i -1} \circ \gamma_{2i -1} \in \mathrm{ker}\left(\piet(\tX_{v_{i},\bk},\bx_{v_{i}}) \to \Aut(C_{v_{i}}^{i})\right)
\]  
and
\[
  \gamma_{2i} \circ \gamma_{2i}' \in \mathrm{ker}\left(\piet(\tX_{v_{i+1},\bk},\bx_{v_{i+1}}) \to \Aut(C_{v_{i+1}}^{i})\right)
\] 
 
\end{itemize} 
\end{definition}
The following picture might be useful to visualize the definition.
\begin{center}
  \begin{tikzcd}
    \scalebox{1.6}{$\bullet$} \arrow[bend left]{r}{\gamma_1} &\bullet \arrow[bend left]{l}{\gamma_1'} \arrow[bend left]{r}{\gamma_2} &\scalebox{1.6}{$\bullet$} \arrow[bend left]{l}{\gamma_2'} \cdots  \scalebox{1.6}{$\bullet$} \arrow[bend left]{r}{\gamma_{2m-1}} &\bullet \arrow[bend left]{l}{\gamma_{2m-1}'} \arrow[bend left]{r}{\gamma_{2m}} &\scalebox{1.6}{$\bullet$} \arrow[bend left]{l}{\gamma_{2m}'} \arrow[loop right]{}{\beta}
  \end{tikzcd}
\end{center}
Here, the larger bullets correspond to $\bx_{v_i}$'s and  the smaller ones to $\partial^{\epsilon}_{0 \textrm{ or } 1}(\bx(e_i))$. 

\begin{rmk}
  We find the definition involving $C_v^N$'s quite convenient. One could, however, avoid introducing $C_v^N$'s and make a slightly different definition. Define $O_v^{N,+}$ to be the set of (isomorphism classes of) $\Gal_k$-conjugates of the $\piet(\tX_{v,\bk},\bx_v)$-sets in $O_v^N$. Prop.\ \ref{dominbyrat} then implies that $O_v^{N,+}$ are finite as well. Moreover, for each $v$, both $O_v^N$ and $O_v^{N,+}$ are increasing with $N$. We could then require $\beta$'s and $(\gamma'_{2i-1} \circ \gamma_{2i-1})$'s as above to act trivially on each element of $O_v^{m,+}$ and $O_{v_i}^{i,+}$, correspondingly.
\end{rmk}

{\bf \emph{Step III: Defining the desired open subgroup $V$, checking its properties and finishing the proof}}

We make the following central definition.

\begin{center}
  \begin{minipage}{0.8\linewidth}
\emph{Let $V_0<\pipet(X_{\bk},\bx)$ denote the subgroup generated by the set of regular loops and let $V$ be its topological closure.}
  \end{minipage}
\end{center}

Let $G = \Big(\big(*^{\rmtop}_v \piet(\tX_{v,\bk},\bx_v)\big) \topfp \pi_1(\Gamma',T')\Big)/ \overline{\langle R_1', R_2' \rangle^{nc}}$ denote the topological group appearing in the van Kampen presentation above. We have that $G^{\rmNoohi} = \pipet(X_{\bk},\bx)$.  Let $\tilde{G} \subset \pipet(X_{\bk},\bx)$ denote the subgroup of all \'etale paths (or ``loops'', rather) of special form, supported on walks from $v_{\bx}$ to $v_{\bx}$. 
\begin{obs}\label{obs:tilde-Noohi-G}
  By Obs.\ \ref{obs:rho-v-special}, the map $G \to \pipet(X_{\bk},\bx) = G^{\rmNoohi}$ factorizes through $\tilde{G}$. Directly from the definitions, there is $V_0 < \tilde{G}$. We are thus in the situation of Lm.\ \ref{lem:V^Noohi-description}. We will use it below.
\end{obs}

For brevity, let us denote $\bG_v = \piet(\tX_{v,\bk},\bx_v)$ and $G_v = \piet(\tX_v,\bx_v) \simeq \bG_v \rtimes \Gal_k$ in the proofs below.

\begin{proposition}\label{prop:V-contained-open-invariant}
The following statements about the subgroup $V$ hold:
  \begin{enumerate}
    \item\label{prop:v-coi:contained} There is a containment $V < U$.
    \item\label{prop:v-coi:open} It is an open subgroup.
    \item\label{prop:v-coi:invariant} The groups $V_0$ and $V$ are invariant under the Galois action, i.e.\ $\psi_\sigma(V_0) = V_0$ and $\psi_\sigma(V) = V$ for all $\sigma \in \Gal_k$.
  \end{enumerate}
\end{proposition}
\begin{proof}
  \begin{enumerate}
    \item As any open subgroup is automatically closed, it is enough to show that any regular loop lies in $U$. Let $g$ be a regular loop and write $g = \gamma' \circ \beta \circ \gamma$ with $\gamma,\gamma'$  \'etale paths of special form supported on some walk (and its inverse) from $v_{\bx}$ to $v$ of length $m$, with presentations $(\gamma_1,\ldots, \gamma_{2m})$ and $(\gamma_{2m}',\ldots, \gamma_1')$ of $\gamma$ and $\gamma'$, and $\beta \in \ker(\piet(\tX_{v,\bk},\bx_v) \to C_v^m)$, as in the definition of a regular loop. Let us introduce the following notation (and analogously for $\gamma'$)
    \[
      \gamma_{i \leftarrow 1} = \gamma_i \circ \ldots \circ \gamma_2 \circ \gamma_1
    \]
    By definition, there is 
    \[
      \gamma_{2i \leftarrow 1} \cdot s_0 \in O_{v_i}^i
    \]
    For $i=m$, it follows from the condition on $\beta$ that $(\beta \circ \gamma) \cdot s_0 = \gamma \cdot s_0$. Similarly, the condition on $\gamma_{2m} \circ \gamma_{2m}'$ implies that $(\gamma_{2m}'^{-1} \circ \gamma_{2m}^{-1}) \circ \gamma_{2m \leftarrow 1} \cdot s_0 = \gamma_{2m \leftarrow 1} \cdot s_0$, and thus 
    \[
      (\gamma' \circ \beta \circ \gamma) \cdot s_0 = (\gamma' \circ \gamma) \cdot s_0 = ((\gamma_{1 \leftarrow 2m-1}' \circ \gamma_{2m}' \circ \gamma_{2m \leftarrow 1})) \cdot s_0 = \gamma_{1 \leftarrow 2m-1}' \circ \gamma_{2m-1 \leftarrow 1} \cdot s_0
    \]
    The process continues in a similar fashion to show that $g$ stabilizes $s_0$, and thus belongs to $U$.

    \item By Lm.\ \ref{lem:V^Noohi-description}, it is enough to check that the map $G \to \rmAut(\tilde{G}/V_0)$ is continuous when $\tilde{G}/V_0$ is considered with the discrete topology.
    
    Using the universal property of free topological products, continuity can be checked separately for $\bG_v$ and $D$. For $D$, this is automatic, as $D$ is discrete. To see the result for $\bG_v$'s, we need to show that the stabilizers of the action of $\bG_v$ on $\tilde{G}/V_0$ induced by $\bG_v \to G$ are open. Fix $[gV_0] \in \tilde{G}/V_0$ and $g \in \tilde{G}$ representing it. The element $g$ is represented by some \'etale path (or a ``loop'', in fact) of special form $\rho$ supported on a walk $P_\rho$ of length $l(P_\rho)$. By Obs.\ \ref{obs:rho-v-special}, the morphism $\bG_v \to \tilde{G} \subset \pipet(X_{\bk},\bx)$ is also defined using an \'etale path of special form $\gamma_v$ supported on a walk $P_{v_{\bx},v}$ in the tree $T \subset \Gamma$. Let $H_v = \ker\left(\bG_v \to \rmAut\big(C^{l(P_\rho) + l(P_v)}_v\big)\right) < \bG_v$. Then $H_v$ is open in $\bG_v$ and its image in $\tilde{G}$ can be written as $\{\gamma_v^{-1} \circ \beta \circ \gamma_v | \beta \in H_v\}$. It follows from the setup that for $\beta \in H_v$,
\[
  g^{-1}\circ\gamma_v^{-1}\circ\beta \circ \gamma_v \circ g \in V_0
\]
   and so any element in the image of $H_v$ fixes $[gV_0]$ in $G/V$. Thus, the stabilizer of $[gV_0]$ in $\bG_v$ is also open, as desired.
  
    \item For each $\sigma$, the map $\psi_\sigma$ is continuous. As $V = \overline{V_0}^{G^{\rmNoohi}}$, it is thus enough to prove that $V_0$ is $\Gal_k$-invariant. By Lm.\ \ref{lem:Gal-action-on-paths}, it follows that under the action of $\Gal_k$, an \'etale path of special form supported on a walk $P$ is mapped again to an \'etale path of special form supported on $P$. Consequently, checking that the action of $\sigma \in \Gal_k$ maps a regular loop $g$ to another regular loop boils down to checking the following fact. If $g$ has a presentation $g = \gamma' \circ \beta \circ \gamma$ as in the definition of a regular loop, then
  \begin{itemize}
    \item $\psi_\sigma(\beta)$ still acts trivially on $C_v^{l(P)}$;
    \item $\psi_\sigma(\gamma_i' \circ \gamma_i)$ or $\psi_\sigma(\gamma_i \circ \gamma_i')$, depending on parity, still acts trivially on $C_{v_{\lfloor i \rfloor +1}}^{\lceil i \rceil}$ for every $i$.
  \end{itemize}
   However, as the automorphism $\psi_\sigma$ on $\piet(\tX_{v_j,\bk},\bx_{v_j})$ matches conjugation by $\sigma$ in $\piet(\tX_{v_j},\bx_{v_j})$ restricted to its normal subgroup $\piet(\tX_{v_j,\bk},\bx_{v_j})$ and the sets $C_{v_j}^j$ were Galois as $\piet(\tX_{v_j},\bx_{v_j})$-sets, the result follows.

  \end{enumerate}
\end{proof}

\begin{center}
  \begin{minipage}[t][1cm]{0.8\linewidth}
\emph{Denote by $S'$ the quotient $\pipet(X_{\bk},\bx)/V$ considered as a $\pipet(X_{\bk},\bx)$-set.}
  \end{minipage}
\end{center}

Let $\rho_v : \pipet(X_{\bk},\bx_v) \to \pipet(X_{\bk},\bx)$ be the isomorphism defined using the fixed (\'etale) path $\gamma_v$ between $\bx_v$ and $\bx$, as in Obs.\ \ref{obs:rho-v-special}.
We have an action given by $\sigma \mapsto \psi_\sigma$ on both $\pipet(X_{\bk},\bx)$ and $\pipet(X_{\bk},\bx_v)$. We already know that $V < \pipet(X_{\bk},\bx)$ is invariant under this action, but this is not necessarily true for $\rho_v^{-1}(V) < \pipet(X_{\bk},\bx_v)$. This holds after a finite base field extension.
\begin{lemma}\label{Galois-invariance-of-v-pullback}
  For each $v \in \rmVert$, define an (abstract) $\Gal_{k,v}$-action on $\pipet(X_{\bk},\bx)$ to be 
  \[
    {}^{\sigma_v} g = \rho_v\Big(\psi_\sigma\big(\rho_v^{-1}(g)\big)\Big)
  \]
  Then there exists a finite extension $l/k$, such that for all $v\in \rmVert$, there is
  \begin{enumerate}[label=\alph*)]
    \item $\Gal_{l,v}$ fixes $V$;
    \item The obtained induced $\Gal_{l,v}$-action on $S'$ can be written as
    \[
      \sigma_v \cdot [gV] = [(\gamma_v^{-1}\circ \psi_\sigma(\gamma_v))\bullet \psi_\sigma(g)V]
    \]
    \item The induced $\Gal_{l,v}$ action on $S'$ is continuous and compatible with the $\bG_v$-action.
    
  \end{enumerate}
\end{lemma}
\begin{proof}
  As there are finitely many vertices $v$, it is enough to prove the statements for a single fixed $v$. Let $g \in V$. By definition of $\rho_v$, there is
  \[
    {}^{\sigma_v} g = \gamma_v^{-1}\circ \big(\psi_\sigma(\gamma_v \circ g \circ \gamma_v^{-1})\big) \circ \gamma_v = \big(\gamma_v^{-1}\circ \psi_\sigma(\gamma_v)\big) \bullet \psi_\sigma(g) \bullet \big(\psi_\sigma(\gamma_v^{-1})\circ \gamma_v \big)
\]
By Prop.\ \ref{prop:V-contained-open-invariant}, we have $\psi_\sigma(g) \in V$ and we only need to show that $\gamma_v^{-1} \circ \psi_\sigma(\gamma_v) \in V$. By Lm.\ \ref{lem:continuity-psi-sigma-gamma} and Obs.\ \ref{obs:rho-v-special}, the map $\Gal_k \ni \sigma \mapsto \psi_\sigma(\gamma_v) \in \pipet(X_{\bk},\bx,\bx_v)$ is continuous, and we conclude that for an open subgroup of $\sigma \in \Gal_k$ we have the desired containment.

It follows from the previous point that we get an induced action of $\Gal_{l,v}$ on $S'$. Using that $\gamma_v^{-1} \circ \psi_\sigma(\gamma_v) \in V$, the alternative formula in the statement follows from the computation
\[
 \sigma_v \cdot [gV] = [\rho_v(\psi_\sigma(\rho_v^{-1}(g)))V] = [(\gamma_v^{-1}\circ \psi_\sigma(\gamma_v))\bullet \psi_\sigma(g) \bullet (\psi_\sigma(\gamma_v)^{-1}\circ \gamma_v)V] = [(\gamma_v^{-1}\circ \psi_\sigma(\gamma_v))\bullet \psi_\sigma(g)V]
\]

Let us move to the last point. Compatibility with the $\bG_v$-action follows from Lm.\ \ref{lem:Gal-action-on-paths}(\ref{lem:pt:Gal-action-as-conjugation} and the fact that the map $\bG_v \to \pipet(X_{\bk},\bx)$ is defined by postcomposing with $\rho_v$. To check continuity, fix $[gV]$. By Lm.\ \ref{lem:V^Noohi-description}, this class is represented by a path (loop) of special form, and so we can assume this about $g$. Checking that the stabilizer of $[gV]$ is open boils down to checking that for an open subgroup of $\sigma$'s in $\Gal_{l,v}$, one has $g^{-1}\bullet (\gamma_v^{-1}\circ \psi_\sigma(\gamma_v))\bullet \psi_\sigma(g) \in V$. However, this follows from the openness of $V$ and Lm.\ \ref{lem:continuity-psi-sigma-gamma}.
\end{proof}

\begin{prop}\label{prop:ending-prop}
  There is a (continuous) $\pipet(X_l,\bx)$-action on $S'$ that extends the $\pipet(X_{\bk},\bx)$-action.
\end{prop}
\begin{proof}
By the van Kampen theorem for $\pipet(X_l,\bx)$, it is enough to show that there are continuous actions of $\bG_v\rtimes \Gal_{l,v}$'s and $D$ compatible with the $\bG_v$ and $D$ actions that $S'$ is already equipped with, and such that the van Kampen relations are satisfied. We already have a continuous action by $D$ on $S'$, and by Lm.\ \ref{Galois-invariance-of-v-pullback}, we get an action of $\bG_v \rtimes \Gal_{l,v}$.

Let us now check that the van Kampen relations  are preserved. In the case of relation $R_2$, this is automatic by Obs.\ \ref{obs:setup} \ref{obs:pt:R_2'=R_2}. This is because we have left the $\bG_v$-actions intact. To check that relation $R_1$ is respected, it suffices to check that $\piet(\partial_1)(\sigma )\vec{E} = \vec{E} \piet(\partial_0)(\sigma)$ for $\sigma \in \Gal_{l,E}$ and $E$ an edge in $\Gamma$ with vertices $v_-, v_+$. Let $\delta_{W',W} = \gamma_{W',W}$ denote the fixed paths from the van Kampen setup in the computation below to make the distinction from $\gamma_v$'s clearer. Using Lm. \ref{lem:Gal-action-on-paths} \ref{lem:pt:Gal-action-as-conjugation}, we compute that
\[
  \piet(\partial_0)(\sigma) = \delta_{v+,E}^{-1} \circ_{G_{v+}} \sigma_E \circ_{G_{v+}} \delta_{v_+,E} = \delta_{v+,E}^{-1}\circ_{G_{v+}} \psi_{\sigma}(\delta_{v+,E})\circ_{G_{v+}} \sigma_{v+}= 
  (\delta_{v+,E}^{-1}\circ_{G_{v+}} \psi_{\sigma}(\delta_{v+,E}))\bullet_{G_{v+}} \sigma_{v+}
\]
The image of $(\delta_{v+,E}^{-1}\circ_{G_{v+}} \psi_{\sigma}(\delta_{v+,E}))$ in $\pipet(X,\bx)$ via $\rho_{v+}$ is $\gamma_{v+}^{-1} \circ \delta_{v+,E}^{-1}\circ \psi_{\sigma}(\delta_{v+,E}) \circ \gamma_{v+}$.

By definition, $\vec{E} \in \pipet(X,\bx)$ can be written as $\vec{E} = \gamma_{v-}^{-1}\circ \delta_{v-,E}^{-1}\circ \delta_{v+,E}\circ \gamma_{v+}$. Putting this together and using the formula of Lm. \ref{Galois-invariance-of-v-pullback}, we have that $\vec{E} \bullet \piet(\partial_0)(\sigma) \cdot [hV]$ equals
\begin{align*}
  (\gamma_{v-}^{-1}\circ \delta_{v-,E}^{-1}\circ \delta_{v+,E}\circ \gamma_{v+}) \circ (\gamma_{v+}^{-1} \circ \delta_{v+,E}^{-1}\circ \psi_{\sigma}(\delta_{v+,E}) \circ \gamma_{v+}) \cdot [(\gamma_{v+}^{-1}\circ \psi_\sigma(\gamma_{v+}))\bullet \psi_\sigma(h)V]\\ = 
  [(\gamma_{v-}^{-1} \circ \delta_{v-,E}^{-1} \circ \psi_\sigma(\delta_{v+,E} \circ \gamma_{v+}))\bullet \psi_\sigma(h)V]
\end{align*}
A similar computation (left to the reader) shows that \[  \piet(\partial_1)(\sigma) \bullet\vec{E} \cdot [hV] = [(\gamma_{v-}^{-1} \circ \delta_{v-,E}^{-1} \circ \psi_\sigma(\delta_{v+,E} \circ \gamma_{v+}))\bullet \psi_\sigma(h)V] \] as well.
This finishes the proof of the Proposition.
\end{proof}

\begin{proof}
  {\bf (\emph{End of the proof of Thm.\ \ref{injectivity-on-the-left}})}

 We have proven that for a transitive $\pipet(X_{\bk},\bx)$-set $S$ there exists a finite extension $l/k$ and a transitive $\pipet(X_l,\bx)$-set $S'$ that dominates $S$ as $\pipet(X_{\bk},\bx)$-sets. As explained above, this finishes the proof.
\end{proof}

We have finished our main proof, and thus the most difficult part of the exact sequence is now proven. We now obtain the final form of the fundamental exact sequence. 

\begin{proof}{\bf (\emph{End of the proof of Thm.\ \ref{exactness-in-geometric-to-arithmetic-as-abstract}})}
  
  We already know the statements of the ``moreover'' part and the near exactness in the middle of the sequence. All we have to prove is that $\pipet(X_{\bk},\bx)$ is thickly closed in $\pipet(X,\bx)$.
As $\pipet(X_{\bk},\bx) \rarr \pipet(X,\bx)$ is a topological embedding of Ra{\u \i}kov complete groups, $\pipet(X_{\bk},\bx)$ is a closed subgroup of $\pipet(X,\bx)$ (see e.g.\ \cite[Prop.\ 6.2.7.]{Dikranjan}). By Lm.\ \ref{lepage-lemma}, the proof will be finished if we show that $\pipet(X_{\bk},\bx)$ is normal in $\pipet(X,\bx)$. Observe that checking whether $\overline{\pipet(X_{\bk},\bx)}=\overline{\overline{\pipet(X_{\bk},\bx)}}$ can be performed after replacing $\pipet(X,\bx)$ by any open subgroup $U$ such that $\pipet(X_{\bk},\bx) < U <^\circ \pipet(X,\bx)$. Choosing a suitably large finite field extension $l/k$ and looking at $U=\pipet(X_l,\bx)$, we are reduced to the situation as in the proof of Thm.\ \ref{injectivity-on-the-left}, i.e.\ we have enough rational points on the connected components we are interested in when applying van Kampen. Let $\tilde{G} < \pipet(X_{\bk},\bx)$ be the dense subgroup defined above Prop.\ \ref{prop:V-contained-open-invariant}. Note that by the van Kampen theorem applied to $\pipet(X,\bx)$ together with the observations in Obs.\ \ref{obs:setup}, it follows that the subgroup generated by $\tilde{G}$ and $\Gal_{k,v}$'s is dense in $\pipet(X,\bx)$. Putting this together, it follows that it is enough to check that, for each $v$, conjugation by elements of $\Gal_{k,v}$ fixes $\tilde{G}$ in $\pipet(X,\bx)$. This, however, follows from Lm.\ \ref{lem:Gal-action-on-paths} \ref{lem:Gaop:pt:composition-compatibility} \ref{lem:pt:Gal-action-as-conjugation} and the fact that $\Gal_{k,v} \to \pipet(X,\bx)$ is defined as the composition $\Gal_{k,v} \to \pipet(X,\bx_v) \stackrel{\rho_v}{\to} \pipet(X,\bx)$, where $\rho_v = \gamma_v^{-1} \circ ( - ) \circ \gamma_v$ with $\gamma_v \in \pipet(X_{\bk},\bx,\bx_v)$ of special form.   
\end{proof}

\begin{rmk}\label{rmk:example-not-regular-loop}
  Let us revisit the counterexample of Ex.\ \ref{counterexample-with-picture} from the point of view of the proof above. We will freely use the notation set there. In this example, we have started from the fixed point $s_0$, and used the group elements to reach point $s_1 = g_mh_{m-1}\ldots h_3 g_2 h_1 \cdot s_0$. We have then concluded that $s_2 = \zeta_{\ell^{m+1}}\cdot s_1 = g \cdot s_1$ and justified that the setup forces that this equality contradicts the possibility of extending the Galois action to the set $S$. The problem here is caused by the fact that, denoting $\gamma = g_mh_{m-1}\ldots h_3 g_2 h_1$, the element
  \[
\gamma^{-1} \circ g^{-1} \circ \zeta_{\ell^{m+1}} \circ \gamma
  \]
  stabilizes $s_0$, but it is not a ``regular loop'' in the language introduced above.

  Of course, this only means that this particular ``obvious'' presentation is not as in the definition of a regular loop. But, by now, we know that it provably cannot be a regular loop with any presentation.
\end{rmk}

\begin{rmk}\label{rmk:quick-approach}
  We sketch a slightly different approach to the central part of the main proof. It is a bit quicker, but less constructive, i.e.\ does not ``explicitly'' construct the desired Galois invariant open subgroup in terms of regular loops. We will freely use the fact that a surjective map from a compact space onto a Hausdorff space is a quotient map. 
  
  Assume that we have already done the preparatory steps of the main proof, i.e.\ we have increased the base field to have many rational points and applied the van Kampen theorem.
  We want to prove that the action
  \[
    \Gal_k \times \pipet(X_{\bk},\bx) \to \pipet(X_{\bk},\bx)
  \]
  given by $\psi_\sigma$ is continuous.
  Let $G, \tilde{G}$ be as introduced above Obs.\ \ref{obs:tilde-Noohi-G}. 

  Firstly, one checks that any element of $\tilde{G}$, so a path of special form, can be in fact rewritten with a presentation that makes it visibly an image of an element of $G$, at the expense of the presentation possibly getting longer. Another words, the map $G \to \tilde{G}$ is surjective. By default, $\tilde{G}$ is considered with the subspace topology from $\pipet(X_{\bk},\bx)$. Let us denote $(\tilde{G}, \mathrm{quot})$ the same group but considered with the quotient topology from $G$. We thus have a continuous bijection $(\tilde{G},\mathrm{quot}) \to \tilde{G}$.
  
  The group $G$ is a topological quotient of the free topological product of finitely many compact groups $G_v$ and a finitely generated free group $D \simeq \bbZ^{*r}$. One checks from the universal properties that this free product can be written as a quotient of the free topological group $F(Z)$ (see \cite[Ch.\ 7.]{AT}) on a compact space of generators $Z = \sqcup_v G_v \sqcup_{\{1,\ldots,r\}} *$, i.e.\ the disjoint union of $G_v$'s and $r$ singletons.

  By \cite[Thm.\ 7.4.1]{AT}, $F(Z)$ is, as a topological space, a colimit of an increasing union $\ldots \subset B_n \subset B_{n+1} \subset \ldots$  of compact subspaces. These spaces are explicitly described as words of bounded length in $F(Z)$ (this makes sense, as the underlying group of $F(Z)$ is the abstract free group on $Z$). From this, it follows that (as a topological space) $(\tilde{G},\mathrm{quot}) = \mathrm{colim} K_n$, with $K_n = \rmim(B_n)$. 

  Working directly with $K_n$'s is inconvenient for our purposes, as these sets are not necessarily preserved by the Galois action. The reason is that the van Kampen presentation as a quotient of a free product uses fixed paths, while
   applying Galois action will usually move the paths. One then has to conjugate by a suitable element to ``return'' to the paths fixed in van Kampen, possibly increasing the length of the word.

  Instead, we can consider subsets $K_n' \subset \tilde{G}$ of elements that are paths of special form of length $\leqslant n$, i.e.\ possessing a presentation as a path of special form of length $\leqslant n$ (see Defn. \ref{defn:paths-special-form}). By a reasonably simple combinatorics, one can cook up ``brute force'' bounds $f(n, d), g(n,d) \in \bbN$ in terms of $n$ and the diameter $d = \mathrm{diam}(\Gamma)$ of $\Gamma$ such that there is
  \[
    K_n \subset K_{f(n,d)}' \quad \textrm{ and } \quad K_n' \subset K_{g(n,d)}
  \]
  In conclusion, $(\tilde{G},\mathrm{quot}) = \mathrm{colim} K_n'$ in $\mathrm{Top}$.
  
  By Lm.\ \ref{lem:Gal-action-on-paths}, the $\Gal_k$-action preserves the sets $K_n'$ and $\Gal_k \times K_n' \to K_n'$ is continuous. As $\Gal_k$ is compact, $\Gal_k\times( - )$ has a right adjoint $\mathrm{Maps}_{\mathrm{cts}}(\Gal_k, - )$ in $\mathrm{Top}$ and so $\Gal_k \times (\mathrm{colim}_{n \in \bbN} K_n') = \mathrm{colim}_{n \in \bbN}(\Gal_k\times K_n')$.
  From this, we immediately get that $\Gal_k \times (\tilde{G},\mathrm{quot}) \to (\tilde{G},\mathrm{quot})$ is continuous. As $\Gal_k$-action respects the group action of $\tilde{G}$, it quickly follows that the action is still continuous when $(\tilde{G},\mathrm{quot})$ is equipped with the weakened topology $\tau$ making open subgroups a base at $1$, as in Lm.\ \ref{lem:weakened-top}. By (the easier part of)  Lm.\ \ref{lem:V^Noohi-description}, this weakened topology on $(\tilde{G},\mathrm{quot})$ matches that of $\tilde{G}$. It follows that $\Gal_k \times \tilde{G} \to \tilde{G}$ is continuous.

  By Lm.\ \ref{lem:weakened-top} again, one has to check that the continuity is not lost when passing to the Ra{\v \i}kov completion of the maximal Hausdorff quotient of $(G,\tau)$. This in turn can be justified by similar arguments as in the proof of Lm.\ \ref{lem:V^Noohi-description}. This finishes the sketch. See also \cite[Prop.\ 4.3.3]{BhattScholze}. 

\end{rmk}

\bibliographystyle{alpha}
\bibliography{fes-biblio}

\end{document}